\documentclass[a4paper,12pt]{article}
\usepackage{amssymb,amsmath,amsfonts,amsthm}

\numberwithin{equation}{section}
\usepackage[left=1 in, right=1 in,top=1 in, bottom=1 in]{geometry}

\providecommand{\abs}[1]{\left\vert#1\right\vert}
\providecommand{\norm}[1]{\left\Vert#1\right\Vert}
\providecommand{\pnorm}[2]{\left\Vert#1\right\Vert_{L^{#2}}}

\providecommand{\Rn}[1]{\mathbb{R}^{#1}}

\providecommand{\csubset}{\subset\subset}

\def\nab{\nabla}
\def\dt{\partial_t}
\def\dr{\partial_r}
\def\dx{\partial_x}
\def\dz{\partial_z}

\def\hal{\frac{1}{2}}

\def\ep{\varepsilon}

\def\tep{\tilde{\ep}}
\def\td{\tilde{\delta}}

\def\B{\mathcal{B}}

\def\Z{\mathcal{Z}}
\def\W{\mathcal{W}}
\def\V{\mathcal{V}}
\def\D{\mathcal{D}}
\def\E{\mathcal{E}}
\def\F{\mathcal{F}}

\def\({\left(}
\def\){\right)}
\def\SX{\mathfrak{X}}
\def\SN{\mathfrak{N}}
\def\SC{\mathfrak{C}}

\def\sir{\frac{\sigma}{\rho_0}}

\def\rest{\hskip 1pt{\hbox to 10.8pt{\hfill\vrule height 7pt width 0.4pt depth 0pt\hbox{\vrule height 0.4pt
width 7.6pt depth 0pt}\hfill}}}

\def\evalu{\hskip 1pt{\hbox to 2pt{\hfill \vrule height -6pt width 0.4pt depth0pt}}}

\DeclareMathOperator{\diverge}{div}

\newtheorem{lem}{Lemma}[section]
\newtheorem{cor}[lem]{Corollary}
\newtheorem{prop}[lem]{Proposition}
\newtheorem{thm}[lem]{Theorem}
\newtheorem{remark}[lem]{Remark}

\title{Instability theory of the Navier-Stokes-Poisson equations}
\author{Juhi Jang\footnote{Supported by NSF Grant DMS-0908007}\, and Ian Tice\footnote{Supported by an NSF
Postdoctoral Research Fellowship}}

\begin{document}

\maketitle

\begin{abstract}
The stability question of the Lane-Emden stationary gaseous star configurations is an interesting problem arising in astrophysics.  We establish both linear and nonlinear dynamical instability results for the Lane-Emden solutions in the framework of the Navier-Stokes-Poisson system with adiabatic exponent $6/5 < \gamma < 4/3$. 
\end{abstract}

%%%%%%%%%%%%%%%%%%%%%%%%%%%%%%%%%%%%%%%%%%%%%%%%%%%%%%
\section{Introduction and formulation} %of the problem}
%%%%%%%%%%%%%%%%%%%%%%%%%%%%%%%%%%%%%%%%%%%%%%%%%%%%%%

One of the simplest fundamental hydrodynamical models to describe the motion of self-gravitating  viscous gaseous stars is the compressible Navier-Stokes-Poisson system, which can be written in Eulerian coordinates as follows: 
\begin{equation}\label{eulerian_equations}
 \begin{cases}
  \dt \varrho + \diverge(\varrho \mathbf{u}) = 0  \\
  \dt(\varrho  \mathbf{u}) + \diverge(\varrho \mathbf{u} \otimes \mathbf{u}) + \diverge{S} = -\varrho \nab \Phi  \\
   \Delta \Phi = 4\pi \varrho,
 \end{cases}
\end{equation}
where $(\mathbf{x},t)\in \Rn{3}\times \Rn{+}$, $\varrho(\mathbf{x},t) \ge 0$ is the density, $\mathbf{u}(\mathbf{x},t) \in \Rn{3}$ is  the velocity vector field of the gas,  $\Phi(\mathbf{x},t) \in \Rn{}$ is the potential function of the self-gravitational force, and the stress tensor $S$ is given by 
\begin{equation}
 S = P I_{3\times 3} - \ep \left(\nab \mathbf{u} +\nab \mathbf{u}^t - \frac{2}{3} (\diverge{\mathbf{u}}) I_{3\times 3} \right) -\delta(\diverge{\mathbf{u}})I_{3\times 3}, 
\end{equation}
where $P$ is the pressure of the gas, $\ep >0$ is the shear viscosity, $\delta \ge 0$ is the bulk viscosity, and  $\nab \mathbf{u}^t$ denotes the transpose of $\nab \mathbf{u}$.  We consider polytropic gases for which the equation of state is given by 
\begin{equation}
 P = P(\varrho) = K \varrho^\gamma,
 \end{equation}
where $K$ is an entropy constant and $\gamma>1$ is an adiabatic exponent. Values of $\gamma$ have their own physical significance \cite{Ch}; for instance, $\gamma=5/3$ corresponds to a monatomic gas, $\gamma=7/5$ a diatomic gas, and $\gamma\rightarrow 1^+$ for heavier molecules. 

In the simplest setting, which we consider, solutions to \eqref{eulerian_equations} are spherically symmetric.  For  $r=\abs{\mathbf{x}}$, this allows us to write
\begin{equation}
 \mathbf{u}(\mathbf{x},t) = u(r,t) \frac{\mathbf{x}}{r} \text{ for } u:[0,\infty)\times [0,\infty) \rightarrow \Rn{}
\end{equation}
and 
\begin{equation}
 \varrho(\mathbf{x},t) = \varrho(r,t).
\end{equation}
The equations \eqref{eulerian_equations} then reduce to the pair
\begin{equation}\label{eulerian_radial_1}
  \dt \varrho + u \dr \varrho + \frac{\varrho}{r^2} \dr(r^2 u) =0 
\end{equation}
and
\begin{equation}\label{eulerian_radial_2}
 \varrho(\dt u + u \dr u) + \dr P = -\frac{4\pi \varrho}{r^2}\int_0^r  \varrho(s,t) s^2ds + \dr \left(\frac{(4\ep/3 + \delta)}{r^2} \dr(r^2 u) \right).
 \end{equation}
Stationary solutions $\varrho=\varrho_0(r)$ and $u=0$, which correspond to non-moving gaseous spheres in hydrostatic equilibrium, satisfy the following equation for $P_0 = K \varrho_0^\gamma$:
\begin{equation}\label{LE}
\dr P_0 + \frac{4\pi \varrho_0}{r^2}\int_0^r  \varrho_0(s,t)s^2 ds =0.
\end{equation}
This equation can be solved by transforming it into the well-known Lane-Emden equation \cite{Ch}. The solutions to \eqref{LE} can be characterized by the values of $\gamma$ in the following fashion \cite{lin}: for given finite total mass $M>0$, if $\gamma\in(6/5,2)$, there exists at least one compactly supported solution $\varrho_0$. For $\gamma\in(4/3,2)$, every solution is compactly supported and unique. If $\gamma=6/5$, the unique solution admits an analytic expression, and it has infinite support. On the other hand, for $\gamma\in(1,6/5)$, there are no solutions with finite total mass. 

The stability of the Lane-Emden steady star configurations has been a question of great interest,
and it has been conjectured by astrophysicists that stationary solutions for $\gamma<4/3$ are unstable. The linear stability theory of the above stationary solutions was studied in \cite{lin} in  the inviscid case, namely the Euler-Poisson system, by studying the eigenvalue problem associated to the linearized Euler-Poisson system: any stationary solution is linearly stable when $\gamma\in(4/3,2)$ and unstable when $\gamma\in(1,4/3)$.  In accordance with the linear stability theory, a nonlinear stability for $\gamma>4/3$  was established in \cite{Rein} by using a variational approach. In the case  $\gamma=4/3$, the analysis of  \cite{DLTY} identified an instability in which any small perturbation can cause part of the system to go off to infinity.  In \cite{J1}, a nonlinear instability of the Lane-Emden steady star for $\gamma=6/5$ was proved based on the bootstrap argument, as pioneered in \cite{guo_strauss}.  It is worthwhile to mention that the stability question for the Euler-Poisson system with $6/5 < \gamma < 4/3$  remains an outstanding open problem. 

The same stability question can also be asked in the presence of viscosity. There have been interesting studies on the stabilization effect of viscosity in the Navier-Stokes-Poisson system   for $\gamma>4/3$ under various assumptions \cite{DZ,ZF}.  On the other hand, up to our knowledge, no rigorous stability theories are available for $\gamma<4/3$, the instability regime in the inviscid case. In this regime for viscous gaseous stars, a particularly interesting problem is to investigate whether or not the viscosity would dominate the gravitational force and stabilize the whole system.  The purpose of this article is to establish the instability theory of the Lane-Emden steady stars whose dynamics are governed by the Navier-Stokes-Poisson system for $6/5<\gamma<4/3$.   

We now formulate the problem. We begin by introducing a vacuum free boundary.

%%%%%%%%%%%%%%%%%%%%%%%%%%%%%%%%%%%%%%%%%%%%%%%%%%%%%%%%%%%%%%%%%%%%%%5
\subsection{Vacuum free boundary}
%%%%%%%%%%%%%%%%%%%%%%%%%%%%%%%%%%%%%%%%%%%%%%%%%%%%%%%%%%%%%%%%5

When $\gamma>6/5$, letting $R>0$ be the radius of the steady star, it is well-known \cite{lin} that 
\begin{equation}\label{stat_asymp}
 \varrho_0(r) \sim (R-r)^{1/(\gamma-1)} \text{ for } r \text{ near }R.
\end{equation}
This boundary behavior near vacuum causes a degeneracy in the equations \eqref{eulerian_radial_1} and \eqref{eulerian_radial_2}, and it is not trivial to deal with such a degeneracy even for the local-in-time existence question; we refer, for instance, to \cite{J2,NOM,OM} and also \cite{JM1,JM2} for the compressible Euler case.  It turns out that in order to capture boundary behavior such as \eqref{stat_asymp} in the dynamical setting, one has to consider a free boundary problem associated to \eqref{eulerian_radial_1} and \eqref{eulerian_radial_2} as in \cite{J2,NOM,OM}. We are interested in the evolution of  compactly supported stars with a free boundary where the star meets vacuum.  This is implemented by assuming there is a radius $R = R(t)>0$ so that 
\begin{equation}
 \varrho(r,t) >0 \text{ for }r\in[0,R(t)) \text{ and } \varrho(R(t),t)=0.
\end{equation}
At the free boundary we impose the kinematic condition 
\begin{equation}
 \frac{d}{dt}R(t) = u(R(t),t),
\end{equation}
as well as the continuity of the normal stress, $S \nu =0$ at the surface $r=R(t)$.  The latter condition reduces to
\begin{equation}
 P - \frac{4\ep}{3}\left(\dr u - \frac{u}{r}\right) - \delta\left( \dr u + \frac{2u}{r}\right) = 0 \text{ for } r=R(t), t\ge 0.
\end{equation}
Note that $P(R(t),t) = K\varrho^\gamma(R(t),t) = 0$, so this can be reduced to a relationship between $\dr u $ and $u$ at $r = R(t)$.  Finally, in order for $\mathbf{u} = u(r,t) \mathbf{x}/r$ to be continuous, we require $u(0,t) = 0$ for $t\ge 0$. 

Since the boundary $R(t)$ is free to move in time in Eulerian coordinates, it is convenient to introduce Lagrangian coordinates so that the boundary becomes fixed. Following the framework used in \cite{J2,NOM,OM}, we study our instability problem in  Lagrangian mass coordinates. 

%%%%%%%%%%%%%%%%%%%%%%%%%%%%%%%%%%%%%%%%%%%%%%%%%%%%%%%%%%%%%%%%%%%%%%5
\subsection{Formulation in Lagrangian mass coordinates}
%%%%%%%%%%%%%%%%%%%%%%%%%%%%%%%%%%%%%%%%%%%%%%%%%%%%%%%%%%%%%%%%5

We now reformulate the problem in Lagrangian mass coordinates.  We set
\begin{equation}
 x(r,t) = \int_0^r 4\pi s^2 \varrho(s,t) ds = \int_{B(0,r)} \varrho(y,t) dy
\end{equation}
for the mass contained in an Eulerian ball of radius $r$ at time $t$.  Note that
\begin{equation}
 \dr x(r,t) = 4 \pi r^2 \varrho(r,t)
\end{equation}
and that
\begin{multline}
 \dt x(r,t) = \int_{B(0,r)} \dt \varrho(y,t) dy = - \int_{B(0,r)} \diverge(\varrho \mathbf{u}) dy \\
= -\int_{\partial B(0,r)} \varrho \mathbf{u}\cdot \mathbf{\nu} = - 4 \pi r^2 \varrho(r,t) u(r,t).
\end{multline} 
In particular, this implies that $\dt x(R(t),t) =0$, which means that the total mass $M>0$ is preserved in time.  The  domain of $x$ is then $[0,M]$.  Switching to Lagrangian mass coordinates $(x,t)\in[0,M] \times [0,\infty)$ and letting the unknowns be 
\begin{equation}
 \rho(x,t) = \varrho(r,t) \text{ and } v(x,t) = u(r,t),
\end{equation}
we get the equations
\begin{equation}\label{lagrangian_1}
 \dt \rho + 4 \pi \rho^2 \dx (r^2 v) =0
\end{equation}
and 
\begin{equation}\label{lagrangian_2}
\dt v + 4\pi r^2 \dx P + \frac{x}{r^2} = 16 \pi^2 r^2 \dx((4\ep/3+\delta)\rho \dx(r^2 v)).
\end{equation}
In Lagrangian coordinates our boundary conditions reduce to
\begin{equation}\label{bc0}
 v(0,t) = 0, \; \rho(M,t) =0,
\end{equation}
and 
\begin{equation}\label{bc}
  P - \frac{4\ep}{3}\left(4\pi r^2 \rho \dx v  - \frac{v}{r}\right) - \delta\left( 4\pi r^2 \rho \dx v + \frac{2v}{r}\right) = 0  \text{ at } x=M \text{ for all }t\ge 0.
\end{equation}
In each of these equations, we have written
\begin{equation}
 r(x,t) = \left(\frac{3}{4\pi} \int_0^x \frac{dy}{\rho(y,t)} \right)^{1/3}.
\end{equation}
A simple computation, employing \eqref{lagrangian_1}, shows that $\dt r(x,t) = v(x,t)$.

A stationary solution $\rho = \rho_0(x)$, $v =0$, $P_0 = K \rho_0^\gamma$ to \eqref{lagrangian_1} and \eqref{lagrangian_2} satisfies the equation
\begin{equation}\label{steadyL}
 4\pi r_0^2(x)\dx P_0(x) +  \frac{x}{ r_0^2(x)}=0
\end{equation}
where
\begin{equation}\label{r0_def}
 r_0(x) = \left(\frac{3}{4\pi} \int_0^x \frac{dy}{\rho_0(y)} \right)^{1/3}.
\end{equation}
This is the Lagrangian version of \eqref{LE}. We denote such a Lane-Emden solution in Lagrangian mass coordinates by $\rho_0$ with pressure $P_0 = K \rho_0^\gamma$. 
 In Lagrangian $x$ coordinates, the boundary behavior \eqref{stat_asymp} is expressed as follows 
\begin{equation}\label{LEL}
 \rho_0(x) \sim (M-x)^{1/\gamma} \text{ for } x \text{ near } M,
\end{equation}
which can be also seen from \eqref{steadyL}. 

The existence and uniqueness of strong solutions to the vacuum free boundary problem of the Navier-Stokes-Poisson system \eqref{lagrangian_1} and \eqref{lagrangian_2} featuring the behavior \eqref{LEL} of Lane-Emden solutions was established in \cite{J2} when $\delta={2\ep}/{3}>0$. The same methodology can be applied to our current setting as long as $\ep>0$ and $\delta>0$, and we will take those strong solutions for granted in proving our nonlinear instability result.  We remark that a well-posedness result in our energy space can be also proved based on our new a priori energy estimates for the fully nonlinear Navier-Stokes-Poisson system, described in Section \ref{4}. 

%%%%%%%%%%%%%%%%%%%%%%%%%%%%%%%%%%%%%%%%%%%%%%%%%%%%%%%%%%%%%%%%%%%%%%5
\subsection{Main results}
%%%%%%%%%%%%%%%%%%%%%%%%%%%%%%%%%%%%%%%%%%%%%%%%%%%%%%%%%%%%%%%%%%%%%%5

Throughout the paper we assume that 
\begin{equation}\label{parameter_assumptions}
 \ep>0, \; \delta >0, \; K >0, \; \text{and } 6/5 < \gamma < 4/3
\end{equation}
are all fixed.  Note that although the only physical requirement on the bulk viscosity is $\delta \ge 0$, the assumption $\delta>0$ is critical for both our linear and nonlinear analysis.  We will also write $M, R>0$ for the mass and radius of a stationary solution to \eqref{steadyL}.

To state the main results, we first write the system in a perturbation form. For small perturbed solutions $\sigma:=\rho-\rho_0$ and $v$ around the steady states satisfying \eqref{steadyL}, the Navier-Stokes-Poisson system \eqref{lagrangian_1} and \eqref{lagrangian_2} can be written as follows: 
\begin{equation}\label{PNSP_0}
\begin{split}
\dt\sigma +4\pi\rho^2\dx(r^2 v)&=0\\
\dt v +4\pi r^2\dx P- 4\pi r_0^2 \dx P_0 + \frac{x}{r^2}-\frac{x}{r_0^2}&= 16 \pi^2 r^2 \dx((4\ep/3+\delta)\rho \dx(r^2 v))
\end{split}
\end{equation} 
with boundary conditions \eqref{bc0} and \eqref{bc}. 

Our first main result concerns the existence of the largest growing mode of the linearized Navier-Stokes-Poisson system around Lane-Emden solutions, which shows a linear instability in the sense of Lin's stability criteria \cite{lin}. 

\begin{thm}\label{thm1.1} 
Suppose \eqref{parameter_assumptions}.  There exist $\lambda>0$ and $\sigma(x),v(x)$ so that $\sigma(x)e^{\lambda t}$ and $v(x) e^{\lambda t}$ solve the linearized Navier-Stokes-Poisson system \eqref{linearized_1} and \eqref{linearized_2} with the linearized boundary conditions \eqref{linearized_30} and \eqref{linearized_3}.  Moreover, this growing mode yields the largest possible growth rate to the linearized system. 
\end{thm}

The precise statement of Theorem \ref{thm1.1} with the estimates is given in Theorem \ref{growing_mode} and Theorem \ref{lin_estimates}.  Our second main result establishes the fully nonlinear dynamical instability of the Lane-Emden solutions to the Navier-Stokes-Poisson system.  In the statement of the theorem, for any given $\iota >0$ and $\theta >\iota$, we write 
\begin{equation}\label{T_iota}
 T^{\iota}:=\frac{1}{\lambda} \ln \frac{\theta}{\iota},
\end{equation}
where $ \lambda$ is the sharp linear growth rate obtained in Theorem \ref{thm1.1}. 

\begin{thm}\label{thm1.2} 
Suppose \eqref{parameter_assumptions}.   There exist function spaces $X$ and $Y$ as well as constants $\theta>0$ and $C>0$ such that for any sufficiently small $\iota>0$, there exist solutions $(\sigma^{\iota}(t), v^{\iota}(t))$ to \eqref{PNSP_0} for $t \in [0,T)$ with $T > T^\iota$ so that
\begin{equation*}
\norm{(\sigma^\iota (0),v^\iota (0))}_Y\le C\iota, \text{  but } \sup_{0\leq t
\leq T^{\iota}} \norm{(\sigma^\iota(t), v^\iota(t))}_X \ge \theta.
\end{equation*}
\end{thm}

The precise statement of Theorem \ref{thm1.2} is given in Theorem \ref{thm} and the spaces $X$ and $Y$ will be clarified in Section \ref{4} and \ref{5}. 

\begin{remark} Our results show that regardless of how large the viscosity parameters $\ep, \delta$ are, and no matter how small smooth initial perturbed data are taken to be, there is no stabilization of the system.  We conclude from this that all Lane-Emden steady star configurations for $6/5 < \gamma < 4/3$ are unstable, regardless of viscosity. 
\end{remark}

\begin{remark} The escape time $T^\iota$ is determined through \eqref{T_iota} by the linear growth rate $\lambda$.  We note that the instability occurs before the possible breakdown or any collapse of strong solutions.  We also remark that the instability occurs in the $X$ norm, which can be characterized by the instantaneous physical energy  
\begin{equation*}
\begin{split}
 &\int_0^M\(\hal |v|^2 +\frac{1}{\gamma-1}\frac{P}{\rho} - \frac{x}{r}\)dx \text{ in Lagrangian mass coordinates, or}\\
&\int_{\Rn{3}}\(\hal \varrho |\mathbf{u}|^2 +\frac{1}{\gamma-1}P\) d\mathbf{x} -\hal \int_{\Rn{3}\times \Rn{3}} \frac{\varrho(\mathbf{x})\varrho(\mathbf{y})}{|\mathbf{x}-\mathbf{y}|} d\mathbf{x}d\mathbf{y} \text{ in Eulerian coordinates.}
\end{split}
\end{equation*}
Of course, this is not a coincidence: the Lane-Emden solutions for $\gamma<4/3$ do not minimize the physical energy functional  and thus one might expect some kind of instability.  Note that they do minimize for $\gamma>4/3$ (see for instance \cite{J1,Rein}). 
\end{remark}

The presence of viscosity and the nonlinear boundary condition \eqref{bc} for the Navier-Stokes-Poisson system make the problem distinguishable and interesting not only from a physical point of view, but also from a mathematical point of view.  What follows now are some of the main mathematical difficulties we encounter in analyzing the system, and a brief discussion of our methods for resolving them.

The proof of Theorem \ref{thm1.1} is based on a variational analysis of equations obtained by linearizing \eqref{PNSP_0}.  The main difficulty that arises in constructing growing-mode solutions is that, due to the viscous terms, the growth rate (eigenvalue) appears in the problem with two different homogeneities.  This breaks the natural variational structure used in \cite{lin} to construct growing modes in the inviscid case.  To get around this difficulty, we employ a technique introduced in \cite{guo_tice}: we introduce a relaxed parameter that allows us to remove one of the eigenvalue homogeneities, study the resulting modified eigenvalue problem (which has a nice variational structure), and finally return to the original formulation through a fixed point argument.  While the solutions constructed in this manner are definitely growing modes, it is not clear a priori that they grow at the largest possible rate.  To verify this, we carry out a careful analysis, paying particular attention to the boundary behavior of the growing mode, which will be crucially used in the subsequent nonlinear bootstrap argument. 

The proof of Theorem \ref{thm1.2} is based on a bootstrap argument from linear instability to  nonlinear dynamical instability.  Passing from a linearized instability to  nonlinear instability requires much effort in the PDE context since the spectrum of the linear part is fairly complicated and the unboundedness of the nonlinear part usually yields a loss in derivatives. In order to get around these difficulties and to find the right space $Y$, we employ careful nonlinear energy estimates for the whole system so that: first, the nonlinear estimates can be closed; and second, their interplay with the linear analysis can complete the argument.  For this particular problem, the space $Y$ is minimally chosen so that the viscosity disturbance near the vacuum boundary can be controlled within $Y$.  We note that  in Lagrangian  mass coordinates, the continuity equation interacts well with the viscosity term, which allows us to derive nice estimates for  $\sigma/\rho_0$ and its temporal and spatial derivatives.  This plays an important role in  closing our nonlinear energy estimates.

Another delicate and important issue is the nonlinear boundary condition \eqref{bc}.  In order to carry out higher-order energy estimates that require integration by parts, we can only employ  differential operators that respect the boundary conditions, namely temporal derivatives.  This forces us to carefully use the structure of the equations in order to gain bounds on spatial derivatives.  A second difficulty with the boundary arises because we use Duhamel's principle to study the nonlinear problem with the linearized evolution operator.  The linearized boundary condition is homogeneous, but the nonlinear boundary condition is certainly not.  This forces us to introduce a corrector function that removes the boundary inhomogeneity.  While the construction of this function is not particularly delicate, the regularity required to do so dictates that we close our energy estimates at a higher order than we would otherwise.

The paper  proceeds as follows.  The first half is devoted to the development of the linear theory and the proof of Theorem \ref{thm1.1}.  In Section \ref{2}, we formulate a variational problem to find a growing-mode solution to the linearized Navier-Stokes-Poisson system.  In Section \ref{lin_est_sec}, we show that our growing-mode solution grows at the largest possible rate. In the second half of the paper, we carry out our nonlinear analysis.  In Section 4 we derive high-order nonlinear energy inequalities.  Based on the linear growth and the nonlinear estimates, we then prove the bootstrap argument and Theorem \ref{thm1.2} in Section 5.

%%%%%%%%%%%%%%%%%%%%%%%%%%%%%%%%%%%%%%%%%%%%%%%%%%%%%%%%%%%%%%%%%%%%%%5

%%%%%%%%%%%%%%%%%%%%%%%%%%%%%%%%%%%%%%%%%%%%%%%%%%%%%%
\section{Construction of a growing mode solution to the linearized equations}\label{2}
%%%%%%%%%%%%%%%%%%%%%%%%%%%%%%%%%%%%%%%%%%%%%%%%%%%%%%

\subsection{Linearization around a stationary solution}
%%%%%%%%%%%%%%%%%%%%%%%%%%%%%%%%%%%%%%%%%%%%%%%%%%%%%%%%%%%%%%%%5

We now linearize the equations in Lagrangian mass coordinates around the stationary solution $v=0$, $\rho=\rho_0$, $r=r_0$ (as defined by \eqref{r0_def}).  We will write $\sigma$ for the linearized density, and (by abuse of notation) $v$ for the linearized velocity.  Then the linearized equations are given by
\begin{equation}\label{linearized_1}
 \dt \sigma + 4 \pi \rho_0^2 \dx (r_0^2 v) =0
\end{equation}
and
\begin{equation}\label{linearized_2}
\dt v + 4 \pi r_0^2 \dx \tilde{P} + \frac{x}{\pi r_0^5}\int_0^x \frac{\sigma(y,t)}{\rho^2_0(y)}dy = 16 \pi^2 r_0^2 \dx((4\ep/3+\delta)\rho_0 \dx(r_0^2 v)), 
\end{equation}
where we have written $\tilde{P} = \gamma K \rho_0^{\gamma-1} \sigma.$  The linearized boundary conditions are
\begin{equation}\label{linearized_30}
 v(0,t) = 0, \; \sigma(M,t) =0,
\end{equation}
and 
\begin{equation}\label{linearized_3}
  \tilde{P} - \frac{4\ep}{3}\left(4\pi r_0^2 \rho_0 \dx v  - \frac{v}{r_0}\right) - \delta\left( 4\pi r_0^2 \rho_0 \dx v + \frac{2v}{r_0}\right) = 0  \text{ at } x=M \text{ for all }t\ge 0.
\end{equation}
Again, we can view \eqref{linearized_3} as a boundary condition only for $v$ since  $\tilde{P} = \gamma K  \rho_0^{\gamma-1} \sigma =0$ at $x=M$ for each $t \ge 0$.

It will often be useful for us to analyze a variant of this system, where we analyze the unknowns $\sigma$ and $w:= r_0^2 v$.  For these unknowns, the equations \eqref{linearized_1}--\eqref{linearized_3} become
\begin{equation}\label{first_order_1}
\begin{split}
&\dt \sigma + 4 \pi \rho_0^2 \dx w =0 \\ 
& \dt w + 4 \pi r_0^4 \dx (\gamma K \rho_0^{\gamma-1} \sigma) - 4 r_0 \dx P_0 \int_0^x \frac{\sigma(y,t)}{\rho_0^2(y)}dy = 16 \pi^2 r_0^4 \dx \left[ \left(\frac{4 \ep}{3} + \delta\right)  \rho_0 \dx w\right] 
\end{split}
\end{equation}
along with the boundary conditions
\begin{equation}\label{first_order_2}
 \frac{w}{r_0^2}(0,t) = \sigma(M,t) =0 \text{ and } \frac{4\ep}{3} \left(4 \pi r_0^3 \rho_0 \dx \left(\frac{w}{r_0^3} \right) \right) +
\delta(4\pi \rho_0 \dx w) =0 \text{ at } x=M.
\end{equation}

%%%%%%%%%%%%%%%%%%%%%%%%%%%%%%%%%%%%%%%%%%%%%%%%%%%%%%%%%%%%%%%%%%%%%
\subsection{Growing mode solution}
%%%%%%%%%%%%%%%%%%%%%%%%%%%%%%%%%%%%%%%%%%%%%%%%%%%%%%%%%%%%%%%%%%%%%%

We want to construct a growing mode solution to the linearized equations.  We do so by looking for a solution of the form
\begin{equation}\label{growing_ansatz}
 \sigma(x,t) = \sigma(x) e^{\lambda t} \text{ and } v(x,t) = v(x) e^{\lambda t}
\end{equation}
for some $\lambda >0$.  If we can find such a solution, then we say the solution is a growing mode since $\abs{e^{\lambda t}} \rightarrow \infty$ as $t \rightarrow \infty$.   Plugging the ansatz \eqref{growing_ansatz} into the linearized equations  \eqref{linearized_1}--\eqref{linearized_3} and eliminating the time exponentials, we arrive at a pair of equations for $\sigma(x)$ and $v(x)$:
\begin{equation}\label{growing_1}
 \lambda \sigma + 4 \pi \rho_0^2 \dx (r_0^2 v) =0
\end{equation}
and
\begin{equation}\label{growing_2}
\lambda v + 4 \pi r_0^2 \dx \tilde{P} + \frac{x}{\pi r_0^5}\int_0^x \frac{\sigma(y)}{\rho^2_0(y)}dy = 16 \pi^2 r_0^2 \dx((4\ep/3+\delta)\rho_0 \dx(r_0^2 v)), 
\end{equation}
along with boundary conditions
\begin{equation}\label{growing_3}
 v(0) = \sigma(M) =0 \text{ and }
  - \frac{4\ep}{3}\left(4\pi r_0^2 \rho_0 \dx v  - \frac{v}{r_0}\right) - \delta\left( 4\pi r_0^2 \rho_0 \dx v + \frac{2v}{r_0}\right) = 0  \text{ at } x=M.
\end{equation}

Our main result of this section establishes the existence of such a growing mode.

\begin{thm}\label{growing_mode}
There exist $\lambda >0$ and $\sigma,v:(0,M) \to \Rn{}$  that solve \eqref{growing_1}--\eqref{growing_3} and satisfy  the following.
\begin{enumerate}
 \item $\sigma$ and $v$ are smooth on $(0,M)$ and satisfy \eqref{growing_1}--\eqref{growing_2} classically for $x \in (0,M)$.  
 \item It holds that
\begin{equation}\label{gm_01}
 \limsup_{x\to 0} \frac{\abs{v(x)}}{r_0(x)} + \limsup_{x\to 0}  \abs{ \sigma(x)} + \limsup_{x\to 0}   \abs{ \dx(r_0^2 v)(x)}   < \infty.
\end{equation}
In particular, $v(0) =0$.

\item Let $\mathfrak{D}$ denote the linear operator $\mathfrak{D} =\rho_0 \dx$.  Then $\mathfrak{D}^k v$ and $\mathfrak{D}^k (\sigma / \rho_0)$ have well-defined traces at $x=M$ for every integer $k \ge 0$.  In particular, $\sigma(M) =0$.

\item $\lambda >0$ satisfies the following variational characterization:
\begin{multline}\label{gm_02}
  \lambda  \int_0^M  \left(\delta \rho_0 \abs{\dx  \theta}^2  + \frac{4\ep}{3} \rho_0 \abs{r_0^3 \dx\left(\frac{ \theta}{r_0^3} \right) }^2  \right) dx \\
+ \int_0^M \left( \frac{\gamma P_0 \rho_0}{2} \abs{\dx \theta}^2 + \frac{ \dx P_0}{2 \pi r_0^3} \abs{\theta}^2 \right) dx  
\ge - \lambda^2 \int_0^M \frac{\abs{ \theta }^2}{16 \pi^2 r_0^4} dx
\end{multline}
for every $\theta$ satisfying $\sqrt{\rho_0} \dx \theta \in L^2((0,M))$ and $\theta/(r_0^2 \sqrt{\rho_0}) \in L^2((0,M))$.  Note that for such $\theta$, it holds that $\theta/(r_0^3 \sqrt{\rho_0}) \in L^2((0,M))$, which means that all of the integrals in \eqref{gm_02} are well-defined.

\item It holds that
\begin{multline}\label{gm_03}
\int_0^M \left(\abs{\frac{\sigma}{\rho_0}}^2 + r_0^2 \abs{\dx \frac{\sigma}{\rho_0}}^2 \right) dx
\\ 
+ \int_0^M \left(\frac{\abs{r_0^2 v}^2}{r_0^6 \rho_0} + \rho_0 \abs{\dx (r_0^2 v)}^2 + r_0^2 \abs{\dx(\rho_0 \dx (r_0^2 v))}^2\right) dx < \infty.
\end{multline}

\end{enumerate}

\end{thm}

The proof of Theorem \ref{growing_mode} will be completed in Section \ref{grow_proof_sec}.  Throughout the rest of the section we develop the tools needed in the proof.  First we reformulate \eqref{growing_1}--\eqref{growing_3} to involve a single unknown function, $\phi$.  The resulting problem for $\phi$ does not possess a standard variational structure since $\lambda$ appears both linearly and quadratically.  To construct a solution using variational methods (required for proving \eqref{gm_02}, which is essential for the linear estimates of Section \ref{lin_est_sec}), we employ the technique of Guo-Tice \cite{guo_tice}, which proceeds as follows.  We modify the problem by replacing the linear appearance of $\lambda$ by an arbitrary parameter $s>0$.  The resulting family (every $s>0$) of problems is amendable to solution by the constrained minimization of an energy functional, and for a range of $s$ we show that $\lambda = \lambda(s) >0$.  We then study the behavior of $\lambda(s)$ as a function of $s$ and show that it is possible to find a unique fixed point so that $\lambda(s)=s>0$.  This then yields the desired solution $\phi$,  which in turn yields the solution to \eqref{growing_1}--\eqref{growing_3}.

We begin by reducing to the study of a single unknown by introducing the function
\begin{equation}\label{phi_def}
 \phi(x) := \int_0^x \frac{\sigma(y)}{\rho^2_0(y)}dy.
\end{equation}
We may then use \eqref{growing_1}--\eqref{growing_3} to  compute
\begin{equation}\label{phi_swap}
  v = -\frac{\lambda}{4\pi r_0^2} \phi,\;
  \sigma = \rho_0^2 \dx \phi, \;
\text{and } 
 \dx \tilde{P} = \dx (\gamma \rho_0 P_0 \dx \phi),
\end{equation}
where $P_0 = K \rho_0^\gamma$.  Using these and replacing in \eqref{growing_2}, we arrive at a second-order equation for $\phi$:
\begin{equation}\label{phi_form_1}
 -\dx \left( \left( \frac{4 \lambda  \ep}{3} + \lambda \delta + \gamma P_0\right) \rho_0 \dx \phi   \right) + \frac{\dx P_0}{\pi r_0^3} \phi = -\frac{\lambda^2}{16\pi^2 r_0^4} \phi.
\end{equation}
The corresponding boundary conditions are
\begin{equation}\label{phi_form_2}
 \frac{\phi}{r_0^2}(0)  =0 \text{ and } \frac{4\ep}{3} \lambda \left(4 \pi r_0^3 \rho_0 \dx \left(\frac{\phi}{r_0^3} \right) \right) +
\delta \lambda (4\pi \rho_0 \dx \phi) =0 \text{ at } x=M.
\end{equation}

%%%%%%%%%%%%%%%%%%%%%%%%%%%%%%%%%%%%%%%%%%%%%%%%%%%%%%%%
\subsection{Modification of the problem}
%%%%%%%%%%%%%%%%%%%%%%%%%%%%%%%%%%%%%%%%%%%%%%%%%%%%%%%%%%

Note that  Theorem \ref{growing_mode} is phrased in Lagrangian mass coordinates.  This is because we will use these coordinates in our nonlinear analysis later in the paper.  However,  constructing the solution to \eqref{phi_form_1}--\eqref{phi_form_2} is somewhat easier if we make a change of variables back to the Eulerian radial coordinates associated to the stationary solution.  To avoid confusion with the Eulerian radial coordinate for the nonlinear problem, we will call our new variable $z = r_0(x)$, where $r_0$ is given by \eqref{r0_def}.  If $x \in (0,M)$ for $M$ the mass of the stationary star, then $z \in (0,R)$ for $R>0$ its radius.   We will write $\varrho_0(z) = \rho_0(x)$ for the stationary density, $P_0 = K \varrho^\gamma$, and  $\varphi(z) = \phi(x)$ for the new unknown in $z$ coordinates.  Then 
\begin{equation}
 \dx = \frac{1}{4\pi z^2 \varrho_0} \dz. 
\end{equation}
In these coordinates, the equations \eqref{phi_form_1} becomes
\begin{equation}\label{linearized}
 -\dz \left( \left( \frac{4\lambda \ep}{3} + \lambda \delta + \gamma P_0 \right) \frac{\dz \varphi}{z^2}   \right) + 4 \frac{\dz P_0}{z^3} \varphi = -\frac{\lambda^2 \varrho_0}{z^2} \varphi.
\end{equation}
For the boundary condition at $z=R$ we use \eqref{phi_form_2} to see that
\begin{equation}\label{linearized_bc}
 \lambda \delta  \frac{\dz \varphi(R)}{R^2}  +  \frac{4 \lambda \ep}{3} \left( \frac{\dz \varphi(R)}{R^2} - 3 \frac{\varphi(R)}{R^3}\right) = 0.
\end{equation}
At $z=0$ we enforce the boundary condition $\varphi(0)=0$.  Once we have a solution in hand, we will show that, in fact, $\varphi(z)/z^2 \rightarrow 0$ as $z \rightarrow 0$, which allows us to switch back to the boundary condition $(\phi/r_0^2)(0) =0$.

There is a difficulty in viewing \eqref{linearized}--\eqref{linearized_bc} in a variational or Sturm-Liouville framework because of the appearance of $\lambda$ with two different homogeneities. To get around this issue, we temporarily modify the problem in order to restore the variational structure.  Ultimately we will undo the modification and return to the proper formulation.

Fix $s>0$ and define
\begin{equation}
 \tep  = s \ep \text{ and } \td  = s \delta.
\end{equation}
Instead of \eqref{linearized}, we will analyze the equation
\begin{equation}\label{modified}
 -\dz \left( \left( \frac{4\tep}{3} + \td + \gamma P_0 \right) \frac{\dz \varphi}{z^2}   \right) + 4 \frac{\dz  P_0}{z^3} \varphi = -\frac{\lambda^2 \varrho_0}{z^2} \varphi
\end{equation}
for arbitrary $s>0$.  We couple this equation to the boundary conditions $\varphi(0)=0$ and
\begin{equation}\label{modified_bc}
  \td  \frac{\dz \varphi(R)}{R^2}  +  \frac{4 \tep}{3} \left( \frac{\dz \varphi(R)}{R^2} - 3 \frac{\varphi(R)}{R^3}\right) = 0.
\end{equation}

Modifying the problem in this way  restores the variational structure.  Indeed, in \eqref{modified} the $\lambda^2$ term can be viewed as an eigenvalue.  Thinking of the principal eigenvalue $\lambda$ as a function of $s$, i.e. $\lambda= \lambda(s)$, we will show that it is possible to choose $s$ so that $\lambda(s)>0$ and $s = \lambda(s),$ which returns us to the original problem and yields a growing-mode solution.

%%%%%%%%%%%%%%%%%%%%%%%%%%%%%%%%%%%%%%%%%%%%%%%%%%%%%%%%
\subsection{Constrained minimization}
%%%%%%%%%%%%%%%%%%%%%%%%%%%%%%%%%%%%%%%%%%%%%%%%%%%%%%%%%%

In order to construct solutions to \eqref{modified}--\eqref{modified_bc}, we will employ a constrained minimization.  To begin, we define the function space on which the energy functionals will be defined.  For $\tau >0$ we define the weighted Sobolev space $H^1_\tau((0,R))$ as the completion of $\{ u\in C^\infty([0,R])  \;\vert\; u(0)=0  \}$  with respect to the norm
\begin{equation}
\norm{u}^2_{H^1_\tau} = \int_0^R \frac{\abs{u'(z)}^2 + \abs{u(z)}^2}{z^\tau}dz,
\end{equation}
where $' = d/dz$.  This weighted Sobolev space possesses the same sort of embedding (continuous and compact) properties as the usual space $H^1$.  Since these results embeddings are not widely available in the literature, we record them in the following lemma.

\begin{lem}\label{sob_embeds}
The following hold.
\begin{enumerate}
 \item   For $u\in H^1_\tau((0,R))$, we have the inequalities
\begin{equation}
 \sup_{0\le z \le R} \abs{u(z)z^{-(\tau+1)/2}} \le \frac{1}{\sqrt{1+\tau}}\left(  \int_0^R \frac{\abs{u'(z)}^2}{z^\tau}dz \right)^{1/2}
\end{equation}
and
\begin{equation}\label{s_em_0}
 \int_0^R \frac{\abs{u(z)}^2}{z^{\tau+2}} dz \le \frac{4}{(1+\tau)^2} \int_0^R \frac{\abs{u'(z)}^2}{z^\tau}dz.
\end{equation}

\item Let $0\le \alpha <1$.  We have the compact embedding $H^1_\tau((0,R))\csubset L^2_{\tau+1+\alpha}((0,R))$, where the latter space is the weighted $L^2$ space with norm
\begin{equation}
\norm{u}_{L^2_{\tau+1+\alpha}}^2 =  \int_0^R \frac{\abs{u(z)}^2}{z^{\tau+1+\alpha}}dz.
\end{equation}
\end{enumerate} 
\end{lem}

\begin{proof}
We begin with the inequalities in item $1$.  By approximation we may assume that $u$ is smooth and $u(0)=0$.  Then
\begin{multline}
 \abs{u(z)} = \abs{u(z) - u(0)} \le \int_0^z \abs{u'(t)}dt \le \left( \int_0^z t^\tau dt \right)^{1/2}  \left( \int_0^z \frac{\abs{u'(t)}^2}{t^\tau}dt \right)^{1/2} \\
 \le \left( \frac{z^{\tau+1}}{\tau+1} \right)^{1/2}  \left( \int_0^R \frac{\abs{u'(t)}^2}{t^\tau}dt \right)^{1/2},  
\end{multline}
which yields the first inequality.  To get the second, we recall an inequality due to G. H. Hardy:
\begin{equation}
 \left(\int_0^\infty \left(\int_0^z \abs{f(t)}dt \right)^p \frac{dz}{z^{b+1}}   \right)^{1/p} \le \frac{p}{b} \left(\int_0^\infty \abs{f(z)}^p z^{p-b-1} dz \right)^{1/p}
\end{equation}
for $1\le p < \infty$ and $0 < b < \infty$.  Then $\abs{u(z)}\le \int_0^z \abs{u'(t)}dt$ implies that
\begin{equation}
\int_0^R \frac{\abs{u(z)}^2}{z^{\tau+2}} dz \le \int_0^\infty \left(\int_0^z \abs{u'(t)} dt\right)^2 \frac{dz}{z^{\tau+2}}.
\end{equation}
Applying Hardy's inequality to the right side with $f=u' \chi_{(0,R)}$,  $b = \tau+1$, and $p=2$ yields
\begin{equation}
 \int_0^R \frac{\abs{u(z)}^2}{z^{\tau+2}} dz \le   \frac{4}{(\tau+1)^2} \int_0^R \frac{\abs{u'(z)}^2}{z^\tau}dz,
\end{equation}
which is the desired inequality.

We now prove the compactness result.  Assume that $\norm{u_n}_{H^1_\tau} \le C$ for $n\in \mathbb{N}$.  Fix $\kappa>0$.  We claim that there exists a subsequence $\{u_{n_i}\}$ so that 
\begin{equation}
 \sup_{i,j}\norm{u_{n_i} - u_{n_j}}_{L^2_{\tau+1+\alpha}} \le \kappa.
\end{equation}
To prove the claim, let $z_0\in(0,R)$ be chosen so that
\begin{equation}
 z_0^{1-\alpha} \frac{C^2 }{(1+\tau)(1-\alpha)} \le \frac{\kappa}{2}.
\end{equation}
Then since the subinterval $(z_0,R)$ avoids the singularity of $1/z^{\tau}$,  $u_n \vert_{(z_0,R)}$ is uniformly bounded in $H^1((z_0,R))$.  By the compact embedding $H^1((z_0,R)) \csubset C^0((z_0,R))$ we may extract a subsequence $\{u_{n_i}\}$ that converges in $L^\infty((z_0,R))$.  We are free to restrict the subsequence to sufficiently large values of $i$ so that 
\begin{equation}
 \norm{u_{n_i} - u_{n_j}}^2_{L^\infty((z_0,R))} \le \frac{\kappa {z_0}^{\tau+1+\alpha}}{2(R-z_0)}  \text{ for all }i,j.
\end{equation}
Then along this subsequence we can apply the first inequality in item $1$ to get 
\begin{multline}
 \int_0^R \frac{\abs{u_{n_i}(z) - u_{n_j}(z)}^2}{z^{\tau+1+\alpha}}dz =  \int_0^{z_0} \frac{\abs{u_{n_i}(z) - u_{n_j}(z)}^2}{z^{\tau+1+\alpha}}dz
+  \int_{z_0}^R \frac{\abs{u_{n_i}(z) - u_{n_j}(z)}^2}{z^{\tau+1+\alpha}}dz  \\
\le \frac{C^2}{1+\tau} \int_0^{z_0}  \frac{dz}{z^\alpha} +  \frac{R-z_0}{{z_0}^{\tau+1+\alpha}} \norm{u_{n_i} - u_{n_j}}^2_{L^\infty((z_0,R))} \le \kappa, 
\end{multline}
which proves the claim.  Now we may use the claim with $\kappa = 1/k$, $k\in \mathbb{N}$ and employ a standard diagonal argument to extract a subsequence converging in $L^2_{\tau+1+\alpha}((0,R))$.
\end{proof}

\begin{remark}
 The inequality \eqref{s_em_0} implies that we can take the norm on $H^1_\tau$ to be 
\begin{equation}
 \norm{u}^2_{H^1_\tau} = \int_0^R \frac{\abs{u'(z)}^2}{z^\tau}dz.
\end{equation}
\end{remark}

We can now define the energy functionals to use in the constrained minimization.  Let
\begin{equation}
 E(\varphi) = \int_0^R   \left( \td + \gamma P_0 \right) \frac{\abs{\dz \varphi}^2}{z^2}  + \frac{4\tep}{3z^2} \abs{ \dz \varphi - 3 \frac{\varphi}{z}  }^2   + 4 \frac{\dz P_0}{z^3} \abs{\varphi}^2 
\end{equation}
and
\begin{equation}
 J(\varphi) = \int_0^R \frac{\varrho_0}{z^2} \abs{\varphi}^2.
\end{equation}
By \eqref{s_em_0} in Lemma \ref{sob_embeds}, both $E$ and $J$ are well-defined on the space $H^1_2((0,R))$.  Note, though, that $E$ is not positive definite since $\dz P_0<0$.  Define the set 
\begin{equation}
 \mathcal{A} := \{ \varphi \in H^1_2((0,R)) \;\vert\; J(\varphi) = 1\}.
\end{equation}
We will build solutions to \eqref{modified} by minimizing $E$ over $\mathcal{A}$.  First we show that such a minimizer exists.

\begin{prop}\label{min_exist}
 $E$ achieves its infimum on the set $\mathcal{A}$.
\end{prop}
\begin{proof}
To begin, we show that $E$ is coercive on $\mathcal{A}$, which amounts to controlling the last term in $E$.  Recall that by \eqref{stat_asymp},  $\varrho_0(z) \sim (R-z)^{1/(\gamma-1)}$ for $z$ near $R$.  This implies that 
\begin{equation}
 \frac{\dz P_0}{\varrho_0} = \gamma K \varrho_0^{\gamma-2} \dz \varrho_0 = \frac{\gamma K}{\gamma-1} \dz(\varrho_0^{\gamma-1})
\end{equation}
is bounded near $z=R$.  Since $\varrho_0$ and $P_0 = K \varrho_0^\gamma$ are smooth and bounded below away from  $z=R$, this implies that 
\begin{equation}
 \norm{\frac{\dz P_0}{\varrho_0}}_{L^\infty((0,R))} < \infty.
\end{equation}
Then for any $z_0\in(0,R)$ we may bound
\begin{multline}
 \int_0^R \abs{\dz P_0}\frac{\abs{\varphi}^2}{z^3} =  \int_0^{z_0} \abs{\dz P_0}\frac{z \abs{\varphi}^2}{z^{4}} + \int_{z_0}^R \frac{\abs{\dz P_0}}{z\varrho_0 }\frac{\varrho_0 \abs{\varphi}^2}{z^2}  \\
\le z_0 \norm{\dz P_0}_{L^\infty} \int_0^{z_0} \frac{ \abs{\varphi}^2}{z^{4}} 
+ \frac{1}{z_0} \norm{\frac{\dz P_0}{\varrho_0}}_{L^\infty}  \int_{z_0}^R  \frac{\varrho_0 \abs{\varphi}^2}{z^2}  \\
\le z_0 \frac{4}{9} \norm{\dz P_0}_{L^\infty} \int_0^R  \frac{ \abs{\dz \varphi}^2}{z^{2}}
+ \frac{1}{z_0} \norm{\frac{\dz P_0}{\varrho_0}}_{L^\infty}.
\end{multline}
For the second inequality we have used Lemma \ref{sob_embeds} and the fact that $\varphi\in \mathcal{A}$.  Then by choosing $z_0$ sufficiently small, we have that 
\begin{equation}
 E(\varphi) \ge -C_z + \int_0^R \left(\frac{\td}{2}  + \gamma P_0   \right)\frac{\abs{\dz \varphi}^2}{z^2} + \frac{4\tep}{3z^2} \abs{ \dz \varphi - 3 \frac{\varphi}{r}  }^2  
\end{equation}
for a constant $C_z>0$ depending on the choice of $z$, which immediately yields the  desired coercivity since $\td >0$.

With the coercivity in hand, we may deduce the existence of a minimizer by using the standard direct methods, employing Lemma \ref{sob_embeds} for compactness.

\end{proof}

Since a minimizer exists, we can now define the function $\mu:(0,\infty)\rightarrow \Rn{}$ by
\begin{equation}\label{mu_def}
 \mu(s) = \inf_{\varphi \in \mathcal{A}} E(\varphi;s),
\end{equation}
where we have written $E(\varphi) = E(\varphi;s)$ to emphasize the dependence of $E$ on the parameter $s>0$, i.e.
\begin{equation}
E(\varphi;s) = s \int_{0}^{R} \delta \frac{\abs{\dz \varphi}^2}{z^2} +  \frac{4 \ep}{3z^2} \abs{ \dz \varphi - 3 \frac{\varphi}{z}  }^2  + \int_0^R \gamma P_0 \frac{\abs{\dz \varphi}^2}{z^2} + 4 \frac{\dz P_0}{z^3}  \abs{\varphi}^2. 
\end{equation}

The minimizer we have constructed satisfies Euler-Lagrange equations of the form \eqref{modified}.

\begin{prop}\label{e_l_eqns}
Let $\varphi \in \mathcal{A}$ be the minimizer of $E$ constructed in Proposition \ref{min_exist}.  Let $\mu := E(\varphi)$.  Then $\varphi$ is smooth on $(0,R]$ and satisfies
\begin{equation}\label{eigen_coupled_1}
 -\dz \left( \left( \frac{4\tep}{3} + \td + \gamma P_0 \right) \frac{\dz \varphi}{z^2}   \right) + 4 \frac{\dz  P_0}{z^3} \varphi = \frac{\mu \varrho_0}{z^2} \varphi
\end{equation}
along with the boundary conditions $\varphi(0) = 0$ and 
\begin{equation}\label{eigen_bc}
  \td  \frac{\dz \varphi(R)}{R^2}  +  \frac{4 \tep}{3} \left( \frac{\dz \varphi(R)}{R^2} - 3 \frac{\varphi(R)}{R^3}\right) = 0.
\end{equation}
\end{prop}

\begin{proof}
 Fix $\varphi_0 \in H_2^1((0,R))$.   Define
\begin{equation}
j(t,\tau)= J(\varphi+t\varphi_0 + \tau \varphi) 
%= \hal \int_{-m}^\ell   \rho_0  (u+t\varphi_0 + s u)^2 + \rho_0 (w+t\psi_0 + s w)^2
\end{equation}
and note that $ j(0,0) = 1$.  Moreover, $j$ is smooth and
\begin{equation}
 \frac{\partial j}{\partial t}(0,0) = 2 \int_{0}^R   \varrho_0   \frac{\varphi_0 \varphi }{z^2}  
\text{ and }
 \frac{\partial j}{\partial \tau}(0,0) = 2 \int_{0}^R   \varrho_0  \frac{\varphi^2}{z^2}  =2.
\end{equation}
So, by the inverse function theorem, we can solve for $\tau = \tau(t)$ in a neighborhood of $0$ as a $C^1$ function of $t$ so that $\tau(0)=0$ and $j(t,\tau(t))=1$.  We may differentiate the last equation to find
\begin{equation}
 \frac{\partial j}{\partial t}(0,0) + \frac{\partial j}{\partial \tau}(0,0) \tau'(0) = 0,
\end{equation}
and hence that
\begin{equation}
 \tau'(0) = -\hal \frac{\partial j}{\partial t}(0,0) = - \int_{0}^R   \varrho_0  \frac{ \varphi_0 \varphi}{z^2}.
\end{equation}

Since $\varphi$ is a minimizer over $\mathcal{A}$,  we then have
\begin{equation}
 0 = \left. \frac{d}{dt}\right\vert_{t=0} E(\varphi +t\varphi_0 + \tau(t) \varphi),
\end{equation}
which implies that 
\begin{multline}
0= \int_0^R \frac{\td + \gamma P_0}{z^2} \dz \varphi (\dz \varphi_0 + \tau'(0) \dz \varphi  ) 
+ \int_0^R 4\frac{\dz P_0}{z^3}\varphi(\varphi_0 + \tau'(0) \varphi) \\
+ \int_0^R \frac{4 \tep}{3 z^2}\left(\dz \varphi - 3 \frac{\varphi}{z} \right)\left(\dz \varphi_0 - 3 \frac{\varphi_0}{z}   + \tau'(0) \left( \dz \varphi  -3 \frac{\varphi}{z}\right)   \right).
\end{multline}
Rearranging and plugging in the value of $\tau'(0)$, we may rewrite this equation as
\begin{multline}\label{eigenvalue_form}
\mu \int_{0}^R   \frac{\varrho_0}{z^2}   \varphi_0 \varphi = 
\int_0^R \frac{\td + \gamma P_0}{z^2} \dz \varphi \dz \varphi_0    
+ \int_0^R 4\frac{\dz P_0}{z^3}\varphi \varphi_0  \\
+ \int_0^R \frac{4 \tep}{3 z^2}\left(\dz \varphi - 3 \frac{\varphi}{z} \right)\left(\dz \varphi_0 - 3 \frac{\varphi_0}{z}     \right),
\end{multline}
where the eigenvalue is $\mu = E(\varphi)$.

By making variations with $\varphi_0$ compactly supported in $(0,R)$ we find that $\varphi$ satisfies the equation \eqref{eigen_coupled_1} in a weak sense in $(0,R)$.  Standard bootstrapping arguments then show that $\varphi \in H^k((z_0,R))$  for all $k\ge 0$ and $0 < z_0 < R$, and hence $\varphi$ is smooth in $(0,R]$.  This implies that the equations are also classically satisfied.  Since $\varphi \in H^2((R/2,R))$, the traces of $\varphi, \dz \varphi$ are well-defined at the endpoint $z = R$.  Making variations with respect to arbitrary $\varphi_0 \in C_c^\infty((0,R])$, we find that the boundary condition \eqref{eigen_bc} is satisfied.  The condition  $\varphi(0)=0$  is satisfied by virtue of Lemma \ref{sob_embeds}.  

\end{proof}

We now want to show that the minimizers, which are solutions to \eqref{eigen_coupled_1}, satisfy the asymptotic condition $\abs{\varphi(z)}/z^2 \rightarrow 0$ as $z \rightarrow 0$.  As a preliminary step we record an asymptotic result for solutions to a more generic ODE.

\begin{lem}[Proposition A.1 of \cite{lin}]\label{gen_asymp}
Suppose that $\psi(\tau)$ solves
\begin{equation}
 \psi''(\tau) + (\alpha \tau^{-1} + g(\tau)) \psi'(\tau) + \tau^{-1} f(\tau) \psi(\tau) =0 \text{ on } 0 < \tau < \tau_0
\end{equation}
where $' = d/d\tau$, and $f,g \in C^0([0,\tau_0])$.   If $\alpha <0$, then  
either $\psi(0)\neq 0$ or
\begin{equation}
 \abs{\psi(\tau)} \le \frac{C}{\tau^{\alpha -1}} \text{ and } \abs{\psi'(\tau)} \le \frac{C}{\tau^\alpha}.
\end{equation}
\end{lem}
\begin{proof}
The case $\alpha >2$ is the content of Proposition A.1 of \cite{lin}, but the proof of the proposition also shows the result when $\alpha <0$.  
\end{proof}

Next we use this lemma to establish the asymptotics at $z=0$ for solutions to \eqref{eigen_coupled_1}.

\begin{lem}\label{soln_asymp}
Suppose $\varphi$ is a solution of \eqref{eigen_coupled_1}.  Then $\abs{\varphi(z)} \le C z^3$ and $\abs{\varphi'(z)}\le C z^2$ near $z=0$.
\end{lem}

\begin{proof}
We begin by rewriting \eqref{eigen_coupled_1}.  Define $X = \tep + \td + \gamma P_0 = \tep + \td + \gamma K \varrho_0^\gamma$ and $X_0 = \tep + P_0 = \tep + K \varrho_0^\gamma$.  Then \eqref{eigen_coupled_1} is equivalent to the equation
\begin{equation}\label{s_a_1}
 \varphi'' + \left(\frac{X'}{X} - \frac{2}{z} \right) \varphi'  - 4 \frac{X_0'}{zX}\varphi   = -\mu \frac{\varrho_0}{X}\varphi. 
\end{equation}

Note that $X'/X$, $X_0'/X$, and $\varrho_0/X$ are all continuous at $z=0$, so we may apply Lemma \ref{gen_asymp} with $\alpha = -2$ to deduce that either $\varphi(0)\neq 0$ or $\abs{\varphi(z)} \le C z^3$ and $\abs{\varphi'(z)}\le C z^2$ near $z=0$.  By Lemma \ref{sob_embeds}, the former condition cannot hold, so the latter conditions must be the case.

\end{proof}

%%%%%%%%%%%%%%%%%%%%%%%%%%%%%%%%%%%%%%%%%%%%%%%%%%%%%%%%
\subsection{Properties of the eigenvalue $\mu(s)$}
%%%%%%%%%%%%%%%%%%%%%%%%%%%%%%%%%%%%%%%%%%%%%%%%%%%%%%%%%%

We are ultimately concerned with finding $\mu = - \lambda^2$ for some $\lambda>0$.  This requires us to work in a range of $s$ so that $\mu(s)<0$.  Our next result shows that $\mu(s)<0$ for $s$ sufficiently small.

\begin{lem}\label{neg_inf}
There exist constants $C_1,C_2>0$ depending on $\gamma, K$ so that 
\begin{equation}
 \mu(s) \le s C_1 - C_2.
\end{equation}
In particular, $\mu(s) <0$ for $s$ sufficiently small. 
\end{lem}
\begin{proof}
Via the homogeneity of $E$ and $J$ we may write
\begin{equation}
 \mu(s) = \inf\{ E(\varphi;s) / J(\varphi) \;\vert\; \varphi\in H^1_2((0,R)) \}.
\end{equation}
The structure of $E$ points to a natural choice for a test function: $\varphi(z) = z^3$.  Using this, integration by parts reveals that
\begin{equation}
 E(\varphi;s) = \int_0^R \left( \td + \gamma P_0\right) 9 z^2 + 4 \dz P_0 z^3 =  \int_0^R 9 \delta s z^2   + (9\gamma-12) P_0 z^2.
\end{equation}
Since $\gamma < 4/3$ it holds that $(9\gamma -12) <0$.   This implies  that
\begin{equation}
 \mu(s) \le \frac{E(\varphi;s)}{J(\varphi)} = s C_1 - C_2, 
\end{equation}
where
\begin{equation}\label{C1_def}
 C_1 = \left(    3 \delta R^3 \right)   \left(\int_0^R \varrho_0(z) z^4 dz\right)^{-1} >0
\end{equation}
and 
\begin{equation}\label{C2_def}
 C_2 = \left( \int_0^R (12 -9\gamma) K \varrho_0^\gamma(z) z^2 dz \right)\left(\int_0^R \varrho_0(z) z^4 dz\right)^{-1} >0.
\end{equation}

\end{proof}

The next proposition proves some crucial monotonicity and continuity properties of $\mu(s)$ for $s>0$.

\begin{prop}\label{eigen_lip}
The following hold.
\begin{enumerate}
\item $\mu(s)$ is strictly increasing in $s$.

\item There exist constants $C_3,C_4>0$ so that  
\begin{equation}\label{e_l_00}
 \mu(s) \ge -C_3 + s C_4.
\end{equation}

\item $\mu \in C^{0,1}_{loc}((0,\infty))$, and in particular $\mu \in C^{0}((0,\infty))$.

\end{enumerate}

\end{prop}

\begin{proof}
We begin by establishing two bits of notation.  According to Proposition \ref{min_exist}, for each $s \in (0,\infty)$ we can find   $\varphi_s \in \mathcal{A}$ so that
\begin{equation}
 E(\varphi_\alpha;s) = \inf_{\varphi \in \mathcal{A}} E(\varphi;s) = \mu(s).
\end{equation}
Next, we decompose $E$ according to 
\begin{equation}\label{e_l_1}
E(\varphi;s) = E_0(\varphi) + s E_1(\varphi) 
\end{equation}
for 
\begin{equation}
 E_0(\varphi) := \int_{0}^{R} \gamma P_0 \frac{\abs{\dz \varphi}^2}{z^2} + 4 \frac{\dz P_0}{z^3}  \abs{\varphi}^2
\end{equation}
and 
\begin{equation}
 E_1(\varphi) := \int_{0}^{R} \delta \frac{\abs{\dz \varphi}^2}{z^2} + \frac{4\ep}{3z^2} \abs{\dz \varphi - 3 \frac{\varphi}{z}}^2  \ge 0.
\end{equation}
The non-negativity of $E_1$ implies that $E$ is non-decreasing in $s$  with $\varphi \in \mathcal{A}$ kept fixed.

To prove the first assertion, note that if  $s_1, s_2  \in (0,\infty)$ with $s_1 \le s_1$  then the minimality of $\varphi_{s_i}$ and  the non-negativity of $E_1$  imply that
\begin{equation}
 \mu(s_1) = E(\varphi_{s_1};s_1) \le E(\varphi_{s_2};s_1) \le  E(\varphi_{s_2};s_2) = \mu(s_2).
\end{equation}
This shows that $\mu$ is non-decreasing in $s$.  Suppose by way of contradiction that $\mu(s_1)=\mu(s_2)$.  Then the last inequality implies that  
\begin{equation}
 s_1  E_1(\varphi_{s_2}) = s_2  E_1(\varphi_{s_2}),
\end{equation}
which means that $E_1(\varphi_{s_2})=0$.  The vanishing of $E_1(\varphi_{s_2})$ implies that $\varphi_{s_2}=0$, which is impossible since $\varphi_{s_2}\in \mathcal{A}$.  Hence equality cannot be achieved, and $\mu$ is strictly increasing in $s$.

Now note that \eqref{e_l_1} and the non-negativity of $E_1$ imply that
\begin{equation}
\mu(s) \ge  \inf_{\varphi\in \mathcal{A}}E_0(\varphi)  + s  \inf_{\varphi\in \mathcal{A}} E_1(\varphi).
\end{equation}
The proof of Lemma \ref{neg_inf} shows that 
\begin{equation}\label{e_l_9}
-C_3:= \inf_{\varphi\in \mathcal{A}}E_0(\varphi) < 0,
\end{equation}
and it is a simple matter to see 
\begin{equation}
 C_4 := \inf_{\varphi\in \mathcal{A}} E_1(\varphi) >0.
\end{equation}
The second assertion follows.

Now fix  $Q = [a,b] \csubset (0,\infty)$, and fix any  $\psi \in \mathcal{A}$.  Again by the non-negativity of $E_1$ and the minimality of $\varphi_s$  we deduce that
\begin{equation}
 E(\psi;b) \ge E(\psi;s) \ge  E(\varphi_s;s) \ge  a E_1(\varphi_s)  -C_3
\end{equation}
for all $s \in Q$.  
This implies that there exists a constant $0<C = C(a,b,\psi,\gamma,K) < \infty$ so that
\begin{equation}\label{e_l_2}
 \sup_{s \in Q} E_1(\varphi_s)  \le C.
\end{equation}

Let $s_1,s_2 \in Q$.  Using the minimality of $\varphi_{s_1}$ compared to $\varphi_{s_2}$, we know that
\begin{equation}\label{e_l_3}
 \mu(s_1) = E(\varphi_{s_1};s_1) \le E(\varphi_{s_2};s_1),
\end{equation}
but from our decomposition \eqref{e_l_1}, we may bound
\begin{equation}\label{e_l_4}
E(\varphi_{s_2};s_1)
\le  E(\varphi_{s_2};s_2) + \abs{s_1 - s_2} E_1(\varphi_{s_2})  \\
= \mu(s_2)+ \abs{s_1 - s_2} E_1(\varphi_{s_2}). 
\end{equation}
Chaining these two inequalities together and employing \eqref{e_l_2}, we find that
\begin{equation}
\mu(s_1)
\le \mu(s_2) + C  \abs{s_1- s_2}.
\end{equation}
Reversing the role of the indices $1$ and $2$ in the derivation of this inequality gives the same bound with $s_1$ switched with $s_2$.  We deduce that 
\begin{equation}
\abs{\mu(s_1) - \mu(s_2)} \le  C \abs{s_1 - s_2},
\end{equation}
which proves item $3$.

\end{proof}

Now we know that the eigenvalue $\mu(s)$ is negative as long as $s<C_2/C_1$ and that $\mu$ is continuous on $(0,\infty)$.  We can then define the non-empty open set
\begin{equation}
 \Omega = \mu^{-1}((-\infty,0)) \subset (0,\infty),
\end{equation}
on which we can calculate $\lambda(s) = \sqrt{-\mu(s)} >0$.  

It turns out that the set $\Omega$ is sufficiently large to find $s>0$ so that $\lambda(s) = s$.  This inversion will then allow us to solve the original growing-mode equations. 

\begin{prop}\label{inversion}
There exists a unique $s \in \Omega$ so that $\lambda(s)=\sqrt{-\mu(s)}>0$ and $\lambda(s) = s$.
\end{prop}

\begin{proof}
According to Propositions \ref{neg_inf} and \ref{eigen_lip}, we know that $\mu(s) < 0$ for $s\in [0,C_2/C_1)$.  Moreover, the lower bound \eqref{e_l_00} implies that $\mu(s) \rightarrow +\infty$ as $s  \rightarrow \infty$.  This implies the existence of $s_0 \in (0,\infty)$ so that $\Omega = (0,s_0)$, which means that $\lambda(s_0)=0$.  Define the function $\Psi:(0,s_0) \rightarrow (0,\infty)$ by $\Psi(s) = s/\lambda(s)$.  The monotonicity and continuity properties of $\mu$ are inherited by $\Psi$, i.e. $\Psi$ is continuous on $(0,s_0)$ and strictly increases from $0$ to $+\infty$ as $s \rightarrow s_0$.  As such, we may apply the intermediate value theorem to find a unique $s\in(0,s_0)$ so that $\Psi(s) = 1$.  For this $s$ we then have that $s = \lambda(s)$, the desired result.
\end{proof}

%%%%%%%%%%%%%%%%%%%%%%%%%%%%%%%%%%%%%%%%%%%%%%%%%%%%%%%%
\subsection{Proof of Theorem \ref{growing_mode}}\label{grow_proof_sec}
%%%%%%%%%%%%%%%%%%%%%%%%%%%%%%%%%%%%%%%%%%%%%%%%%%%%%%%%%%

We now combine our above analysis to deduce the existence of a solution $\varphi$, $\lambda >0$ to \eqref{linearized}--\eqref{linearized_bc}.

\begin{thm}\label{growing_solution}
There exists $\lambda>0$ and $\varphi \in H^1_2((0,R))$, smooth on $(0,R]$, that solves \eqref{linearized} along with the boundary condition \eqref{linearized_bc}. The solution satisfies the asymptotics $\abs{\varphi(z)}\le C z^3$ and $\abs{\dz \varphi(z)} \le C z^2$  as $z \to 0$.
\end{thm}
\begin{proof}
Combining Proposition \ref{e_l_eqns} and Proposition \ref{inversion}, we see that there exists a solution to \eqref{modified} and \eqref{modified_bc} for $\lambda(s)= \sqrt{-\mu(s)}>0$, satisfying $ s= \lambda(s).$  This implies that the solution is actually a solution to \eqref{linearized} and \eqref{linearized_bc}.  The asymptotics at $z=0$ follow from Lemma \ref{soln_asymp}.
\end{proof}

\begin{remark}
 It is actually possible to get a lower bound for the size of $\lambda>0$ in terms of $\delta, \varrho_0,   \gamma$, and $R$.  Indeed, from Lemma \ref{neg_inf} it holds that $\lambda^2 + \lambda C_1 - C_2 \ge 0$, and hence
\begin{equation}
 \lambda \ge \frac{-C_1 + \sqrt{C_1^2 + 4C_2} }{2} > 0.
\end{equation}
Here the constants $C_1,C_2 >0$ are determined explicitly in terms of $\delta, \varrho_0, \gamma$, and $R$ according to \eqref{C1_def} and \eqref{C2_def}.
\end{remark}

An immediate consequence of Theorem \ref{growing_solution} is the existence of a solution to \eqref{phi_form_1}--\eqref{phi_form_2}.

\begin{cor}\label{phi_soln}
There exists $\lambda>0$ and $\phi(x)= \varphi(r_0(x))$, smooth on $(0,M)$, that solves \eqref{phi_form_1}--\eqref{phi_form_2}.  The solution satisfies  
\begin{equation}\label{phis_01}
 \limsup_{x\to 0} \frac{\abs{\phi(x)}}{r_0^3(x)} + \limsup_{x\to 0}  \abs{\dx \phi(x)}  < \infty.
\end{equation}
Let $\mathfrak{D}$ denote the linear operator $\mathfrak{D} \phi(x) = \rho_0(x) \dx \phi(x)$.  The solution satisfies the property that $\mathfrak{D}^k \phi$ has a well-defined trace at $x=M$ for every integer $k\ge 0$. 
\end{cor}
\begin{proof}
All of the conclusions, except those concerning $\mathfrak{D}$ follow directly from Theorem \ref{growing_solution}.  When $k=0$,  the trace $\mathfrak{D}^0 \phi(M) = \phi(M)$ is well-defined since $\varphi(R) = \varphi(r_0(M))$ is well-defined.  Note that
\begin{equation}
 \mathfrak{D}\phi(x)= \rho_0(x) \dx \phi(x) = \frac{\dz \varphi(r_0(x))}{4 \pi r_0^2(x)} \Rightarrow \mathfrak{D}\phi(M) = \frac{\dz \varphi(R)}{4 \pi R^2},
\end{equation}
so that $\mathfrak{D}\phi(M)$ is well-defined.  We may argue similarly, using the fact that $\dz^k \varphi(R)$ is well-defined for all $k\ge 0$, to deduce that $\mathfrak{D}^k\phi(M)$ is well-defined for all $k\ge 0$ as well.

\end{proof}

Now, with Corollary \ref{phi_soln} in hand, we are ready to present the 

\begin{proof}[Proof of Theorem \ref{growing_mode}]

Let $\lambda>0$ and $\phi(x)$  be the solution to \eqref{phi_form_1}--\eqref{phi_form_2} given in by Corollary \ref{phi_soln}.  Let us then define $v$ and $\sigma$ according to 
\begin{equation}\label{gm_1}
 v = -\frac{\lambda}{4\pi r_0^2} \phi \text{ and } \sigma = \rho_0^2 \dx \phi.
\end{equation}
Using these definitions of $v$ and  $\sigma$ in conjunction with the properties of $\phi$ recorded in Corollary \ref{phi_soln}, we easily deduce the first three items of the theorem. 

To prove the variational characterization of the fourth item, we return to the variational characterization of $\lambda$ in $z = r_0(x)$ coordinates.  According to Theorem \ref{growing_solution}, $\lambda>0$ satisfies 
\begin{multline}
 \lambda \int_0^R   \left( \delta    \frac{\abs{\dz \vartheta}^2}{z^2}  + \frac{4\ep}{3z^2} \abs{ \dz \vartheta - 3 \frac{\vartheta}{z}  }^2  \right) dz 
+\int_0^R  \left( \gamma P_0 \frac{\abs{\dz \vartheta}^2}{z^2}  + 4 \frac{\dz P_0}{z^3} \abs{\vartheta}^2\right) dz \\
\ge - \lambda^2 \int_0^R \frac{\varrho_0}{z^2} \abs{\vartheta}^2 dz
\end{multline}
for every $\vartheta \in H^1_2((0,R))$.  Then the variational characterization in \eqref{gm_02} follows by making a change of coordinates $\theta(x) = \vartheta(z)= \vartheta(r_0(x))$.  Note that $\vartheta \in H^1_2((0,R))$ if and only if $\sqrt{\rho_0} \dx \theta \in L^2((0,M))$ and $\theta/(r_0^2 \sqrt{\rho_0}) \in L^2((0,M))$.  Also, changing coordinates in \eqref{s_em_0} of Lemma \ref{sob_embeds} shows that $\theta/(r_0^3 \sqrt{\rho_0}) \in L^2((0,M))$, which means that all of the integrals in \eqref{gm_02} are well-defined.

We now turn to the proof of \eqref{gm_03}.  Using the inclusion $\varphi \in H^1_2((0,R))$, the above analysis implies that $\sqrt{\rho_0} \dx \phi, \phi/(r_0^3 \sqrt{\rho_0}) \in L^2((0,M))$.  From this and equation \eqref{phi_form_1} we may then deduce that
\begin{equation}
 \int_0^M \left( \frac{\abs{\phi}^2}{r_0^6 \rho_0} + \rho_0 \abs{\dx \phi}^2 + r_0^2 \abs{\dx(\rho_0 \dx \phi) }^2  \right) dx < \infty.
\end{equation}
This and \eqref{gm_1} then imply \eqref{gm_03}.

\end{proof}

%%%%%%%%%%%%%%%%%%%%%%%%%%%%%%%%%%%%%%%%%%%%%%%%%%%%%%%%
\section{Linear estimates}\label{lin_est_sec}
%%%%%%%%%%%%%%%%%%%%%%%%%%%%%%%%%%%%%%%%%%%%%%%%%%%%%%%%%%

Due to the indirect way in which we constructed growing mode solutions in Section \ref{2}, it is not immediately obvious that the $\lambda>0$ of Theorem \ref{growing_mode} is the largest possible growth rate.  However, because of  the inequality \eqref{gm_02}, we can show that no solution to the linearized problem \eqref{linearized_1}--\eqref{linearized_3} can grow in time at a rate faster than $e^{\lambda t}$.  Hence the growing mode constructed in Theorem \ref{growing_mode}  actually does grow in time at the fastest possible rate.  The proof of this result and its implications for solutions to the inhomogeneous linearized problem  are our the subject of  this section.

%%%%%%%%%%%%%%%%%%%%%%%%%%%%%%%%%%%%%%%%%%%%%%%%%%%%%%%%
\subsection{Estimates in the second-order formulation}
%%%%%%%%%%%%%%%%%%%%%%%%%%%%%%%%%%%%%%%%%%%%%%%%%%%%%%%%%%

First, we will prove estimates for solutions to the following second-order problem.
\begin{equation}\label{phi_linear_1}
 - \frac{\dt^2 \phi}{16 \pi^2 r_0^4} = \frac{\dx P_0}{\pi r_0^3} \phi - \dx \left[  
\left(\frac{4\ep}{3} + \delta \right)\rho_0 \dx \dt \phi  + \gamma P_0 \rho_0 \dx \phi \right] \text{ for } x \in (0,M)
\end{equation}
with boundary conditions 
\begin{equation}\label{phi_linear_2}
\phi(0,t) = 0 \text{ and } \frac{4\ep}{3} \left(4 \pi r_0^3 \rho_0 \dx \left(\frac{\phi}{r_0^3} \right) \right) +
\delta(4\pi \rho_0 \dx \phi) =0 \text{ at } x=M,
\end{equation}
and initial conditions $\phi(x,0)$ and $\dt \phi(x,0)$ given.  We will assume throughout that $\phi$ satisfies $\sqrt{\rho_0} \dx \phi \in L^2((0,M))$ and $\phi/(r_0^2 \sqrt{\rho_0}) \in L^2((0,M))$.

Solutions to this linear problem obey an energy evolution equation related to the inequality \eqref{gm_02}.  We record this now.

\begin{prop}\label{lin_en_ev}
Suppose $\phi$ is a solution to \eqref{phi_linear_1}--\eqref{phi_linear_2}.  Then 
\begin{multline}
  \dt \int_0^M \frac{\abs{\dt \phi}^2}{32 \pi^2 r_0^4}dx  + \int_0^M  \left(\delta \rho_0 \abs{\dx \dt \phi}^2  + \frac{4\ep}{3} \rho_0 \abs{r_0^3 \dx\left(\frac{\dt \phi}{r_0^3} \right) }^2 \right) dx
\\
= -\dt \int_0^M \left( \frac{\gamma P_0 \rho_0}{2} \abs{\dx \phi}^2 + \frac{ \dx P_0}{2 \pi r_0^3} \abs{\phi}^2 \right) dx.
 \end{multline}
\end{prop}

\begin{proof}
Multiply \eqref{phi_linear_1} by $\dt \phi$ and integrate over $x \in (0,M)$.  An integration by parts, an application of the boundary conditions \eqref{phi_linear_2}, and some simple algebra yield the desired equality.
\end{proof}

We can use this and the variational characterization of $\lambda$ given in Theorem \ref{growing_mode} to deduce some estimates of the following norms:
\begin{equation}
 \norm{\phi}_1^2 := \int_0^M \frac{\abs{ \phi}^2}{16 \pi^2 r_0^4} dx
\end{equation}
\begin{equation}
 \norm{\phi}_2^2 := \int_0^M  \left( \delta \rho_0 \abs{\dx  \phi}^2  + \frac{4\ep}{3} \rho_0 \abs{r_0^3 \dx \left( \frac{\phi}{r_0^3} \right)}^2 \right) dx
\end{equation}
\begin{equation}
 \norm{\phi}_3^2 := \int_0^M  \gamma P_0 \rho_0 \abs{\dx \phi}^2  dx.
\end{equation}

Our estimates for these norms are the subject of the next result.

\begin{thm}\label{lin_estimates}
Let $\phi$ solve \eqref{phi_linear_1}--\eqref{phi_linear_2}.  Then we have the following estimates:
\begin{equation}\label{lin_es_01}
 \norm{\phi(t)}_1^2 + \int_0^t \norm{\phi(s)}_2^2 ds \le e^{2 \lambda t} \norm{\phi(0)}_1^2 + \frac{K_1}{2\lambda} (e^{2\lambda t}-1),
\end{equation}
\begin{equation}\label{lin_es_02}
 \frac{1}{\lambda}\norm{\dt \phi(t)}_1^2 + \norm{\phi(t)}_2^2 \le e^{2\lambda t} \left(   2 \lambda \norm{\phi(0)}_1^2 + K_1 \right),
\end{equation}
and 
\begin{equation}\label{lin_es_03}
 \hal \norm{\phi(t)}_3^2 \le K_0 + C_0 \left[ e^{2 \lambda t} \norm{\phi(0)}_1^2 + \frac{K_1}{2\lambda} (e^{2\lambda t}-1) \right]. 
\end{equation}
Here
\begin{equation}
 K_0 = \norm{\dt \phi(0)}_1^2 + \hal \norm{\phi(0)}_3^2,  K_1 = \frac{2 K_0}{\lambda} + 2 \norm{\phi(0)}_2^2, 
\end{equation}
and
\begin{equation}
C_0 = 2 \sup_{x \in (0,M)} \frac{x}{r_0^3(x)} <\infty.
\end{equation}
\end{thm}
\begin{proof}

We integrate the result of Proposition \ref{lin_en_ev} in time from $0$ to $t$ to see that
\begin{multline}\label{lin_es_1}
  \int_0^M \frac{\abs{\dt \phi(t) }^2}{32 \pi^2 r_0^4} dx + \int_0^t \int_0^M  \left(\delta \rho_0 \abs{\dx \dt \phi(s)}^2  + \frac{4\ep}{3} \rho_0 \abs{r_0^3 \dx\left(\frac{\dt \phi(s)}{r_0^3} \right) }^2\right) dx ds
\\
= K_0 + \int_0^M \frac{ \dx P_0}{2 \pi r_0^3} \abs{\phi(0)}^2 dx - \int_0^M \left( \frac{\gamma P_0 \rho_0}{2} \abs{\dx \phi(t)}^2 + \frac{ \dx P_0}{2 \pi r_0^3} \abs{\phi(t)}^2\right) dx.
\end{multline}
Note that since $\dx P_0 = -x/(4\pi r_0^4)$, we have that
\begin{equation}\label{lin_es_1_2}
 \int_0^M \frac{ \dx P_0}{2 \pi r_0^3} \abs{\phi(0)}^2 dx = -\int_0^M \frac{x}{8\pi^2 r_0^7}\abs{\phi(0)}^2 dx \le 0.
\end{equation}
The variational characterization of $\lambda$ given in \eqref{gm_02} of Theorem \ref{growing_mode} allows us to estimate
\begin{multline}\label{lin_es_2}
 - \hal \int_0^M \left( \frac{\gamma P_0 \rho_0}{2} \abs{\dx \phi(t)}^2 + \frac{ \dx P_0}{2 \pi r_0^3} \abs{\phi(t)}^2\right) dx \\
- \frac{\lambda}{2}  \int_0^M  \left(\delta \rho_0 \abs{\dx  \phi(t)}^2  + \frac{4\ep}{3} \rho_0 \abs{r_0^3 \dx\left(\frac{ \phi(t)}{r_0^3} \right) }^2  \right) dx
\le \frac{\lambda^2}{2} \int_0^M \frac{\abs{ \phi(t) }^2}{16 \pi^2 r_0^4} dx.
\end{multline}
We may then combine \eqref{lin_es_1}--\eqref{lin_es_2} to see that
\begin{multline}
\int_0^M \frac{\abs{\dt \phi(t) }^2}{32 \pi^2 r_0^4} dx  + \int_0^t \int_0^M  \left(\delta \rho_0 \abs{\dx \dt \phi(s)}^2  + \frac{4\ep}{3} \rho_0 \abs{r_0^3 \dx\left(\frac{\dt \phi(s)}{r_0^3} \right) }^2\right) dx ds  \\
\le K_0 + \frac{\lambda^2}{2} \int_0^M \frac{\abs{ \phi(t) }^2}{16 \pi^2 r_0^4}dx 
+ \frac{\lambda}{2}  \int_0^M  \left(\delta \rho_0 \abs{\dx  \phi(t)}^2  + \frac{4\ep}{3} \rho_0 \abs{r_0^3 \dx\left(\frac{ \phi(t)}{r_0^3} \right) }^2 \right) dx,
\end{multline}
which we may rewrite as
\begin{equation}\label{lin_es_3}
 \hal \norm{\dt \phi(t)}_1^2 + \int_0^t \norm{\dt \phi(s)}_2^2 ds \le K_0 +\frac{\lambda^2}{2}\norm{\phi(t)}_1^2 + \frac{\lambda}{2} \norm{\phi(t)}_2^2.
\end{equation}

Integrating in time and using Cauchy's inequality,  we may bound
\begin{multline}\label{lin_es_4}
 \lambda \norm{\phi(t)}^2_2 = \lambda \norm{\phi(0)}^2_2 +\lambda \int_0^t 2 \langle \phi(s),\dt \phi(s) \rangle_2 ds\\
 \le \lambda \norm{\phi(0)}^2_2 +  \int_0^t \norm{\dt \phi(s)}_2^2  ds  + \lambda^2 \int_0^t \norm{\phi(s)}^2_2 ds .
\end{multline}
On the other hand
\begin{equation}
\lambda \dt \norm{\phi(t)}^2_1 = \lambda 2\langle \dt \phi(t), \phi(t) \rangle_1 \le \lambda^2 \norm{\phi(t)}_1^2 + \norm{\dt \phi(t)}^2_1.
\end{equation}
We may combine these two inequalities with \eqref{lin_es_3} to derive the differential inequality
\begin{equation}
 \dt   \norm{\phi(t)}^2_1 +  \norm{\phi(t)}^2_2 \le K_1 + 2\lambda \norm{\phi(t)}^2_1 + 2\lambda \int_0^t \norm{\phi(s)}^2_2 ds
\end{equation}
for $K_1$ as defined in the hypotheses.  An application of Gronwall's lemma then shows that
\begin{equation}\label{lin_es_5}
 \norm{\phi(t)}^2_1 + \int_0^t \norm{\phi(s)}^2_2 ds \le e^{2\lambda t}\norm{\phi(0)}^2_1 + \frac{K_1}{2\lambda} (e^{2\lambda t}-1)
\end{equation}
for all $t\ge 0$, which is the bound \eqref{lin_es_01}.

To derive the estimate \eqref{lin_es_02},  we return to \eqref{lin_es_3} and plug in \eqref{lin_es_4} and \eqref{lin_es_5} to see that
\begin{equation}
 \frac{1}{\lambda} \norm{\dt \phi(t)}_1^2 +   \norm{\phi(t)}^2_2 \le K_1 + \lambda \norm{\phi(t)}_1^2 +  2\lambda \int_0^t \norm{\phi(s)}^2_2 ds \le  e^{2\lambda t} \left( 2\lambda \norm{\phi(0)}^2_1 + K_1 \right).
\end{equation}
Finally, for \eqref{lin_es_03} we return to \eqref{lin_es_1} and employ \eqref{lin_es_1_2} to see that
\begin{equation}\label{lin_es_6}
 \hal \norm{\phi(t)}_3^2 \le K_0 - \int \frac{ \dx P_0}{2 \pi r_0^3} \abs{\phi(t)}^2
= K_0 + \int \frac{x}{8\pi^2 r_0^7}\abs{\phi(t)}^2.
\end{equation}
Since L'Hospital's theorem implies that
\begin{equation}
 \lim_{x \to 0} \frac{x}{r_0^3(x)} = \lim_{x \to 0} \frac{4\pi \rho_0(x)}{3} = \frac{4 \pi \rho_0(0)}{3} < \infty, 
\end{equation}
we may deduce that
\begin{equation}\label{lin_es_7}
 \sup_{x \in (0,M)} \frac{x}{r_0^3(x)} < \infty.
\end{equation}
The estimate \eqref{lin_es_03} then follows directly from \eqref{lin_es_6}, \eqref{lin_es_7}, and the estimate of $\norm{\phi(t)}_1^2$ in \eqref{lin_es_01}.

\end{proof}

%%%%%%%%%%%%%%%%%%%%%%%%%%%%%%%%%%%%%%%%%%%%%%%%%%%%%%%%
\subsection{Estimates in the first-order formulation}
%%%%%%%%%%%%%%%%%%%%%%%%%%%%%%%%%%%%%%%%%%%%%%%%%%%%%%%%%%

Now consider $\sigma$ and $w$ to be solutions to the first-order linear system \eqref{first_order_1} with boundary conditions \eqref{first_order_2} and initial conditions $\sigma(x,0)$ and $w(x,0)$.  A simple calculation shows that if we apply $\dt$ to the second equation in \eqref{first_order_1} and then eliminate $\dt \sigma$ by using the first equation in \eqref{first_order_1}, then we arrive at the second-order formulation \eqref{phi_linear_1}--\eqref{phi_linear_2} for $\phi = w$.  Then Theorem \ref{lin_estimates} yields various estimates for $w = \phi$.  We now seek to rewrite these estimates for $w$ and to use them to derive a similar estimate for $\sigma$.

\begin{thm}\label{first_order_estimates}
Let $\sigma,w$ solve the linear system \eqref{first_order_1}--\eqref{first_order_2}.  Then
\begin{multline}\label{foe_0}
\int_0^M \frac{\abs{w(t)}^2}{r_0^4}dx + \int_0^M \gamma K \rho_0^{\gamma-1} \abs{\frac{\sigma(t)}{\rho_0}}^2 dx    \\
+ 16\pi^2 \int_0^M \left( \delta \rho_0 \abs{\dx  w(t)}^2  + \frac{4\ep}{3} \rho_0 \abs{r_0^3 \dx \left( \frac{w(t)}{r_0^3} \right)}^2 \right) dx  
\le Ce^{2 \lambda t} \left[ 
\int_0^M  \frac{\abs{w(0)}^2}{r_0^4} dx    \right. \\
+ \int_0^M \left( \gamma K \rho_0^{\gamma-1} \abs{\frac{\sigma(0)}{\rho_0}}^2 
+\gamma K \rho_0^{\gamma+1} \abs{\dx w(0)}^2 
+  \frac{\abs{\dt w(0)}^2}{r_0^4} \right) dx
\\ \left.
+  \int_0^M \left( \delta \rho_0 \abs{\dx  w(0)}^2  + \frac{4\ep}{3} \rho_0 \abs{r_0^3 \dx \left( \frac{w(0)}{r_0^3} \right)}^2 \right) dx  \right].
\end{multline}
\end{thm}

\begin{proof}
We switch to the second-order formulation for $\phi = w$.  Then the estimates \eqref{lin_es_01}-- \eqref{lin_es_03} of Theorem \ref{lin_estimates} imply that
\begin{multline}\label{foe_1}
\int_0^M \frac{\abs{w(t)}^2}{r_0^4}dx + \int_0^M \gamma K \rho_0^{\gamma+1} \abs{\dx w(t)}^2 dx
\\
+ 16 \pi^2 \int_0^M \left( \delta \rho_0 \abs{\dx  w(t)}^2  + \frac{4\ep}{3} \rho_0 \abs{r_0^3 \dx \left( \frac{w(t)}{r_0^3} \right)}^2 \right) dx  
\le Ce^{2 \lambda t} \left[ \int_0^M \frac{\abs{w(0)}^2}{r_0^4}dx \right. \\ 
+ \int_0^M \left( \gamma K \rho_0^{\gamma+1} \abs{\dx w(0)}^2  
 +   \frac{\abs{\dt w(0)}^2}{r_0^4} \right) dx 
\\ \left. 
+ \int_0^M \left(  \delta \rho_0 \abs{\dx  w(0)}^2  
+ \frac{4\ep}{3} \rho_0 \abs{r_0^3 \dx \left( \frac{w(0)}{r_0^3} \right)}^2 \right)dx \right].
\end{multline}
Let us call the term in the brackets on the right side of this equation $\mathcal{Z}_0$.  Since $\dt \sigma = -4\pi \rho_0^2 \dx w$, we then have that 
\begin{equation}\label{foe_2}
 \int_0^M \gamma K \rho_0^{\gamma+1} \abs{\dx w(t)}^2 dx = \int_0^M \frac{\gamma K \rho_0^{\gamma-1}}{16 \pi^2 \rho_0^2} \abs{\dt \sigma(t)}^2 dx := \norm{\dt \sigma(t)}_4^2.
\end{equation}
We then have that
\begin{equation}
 \dt \norm{\sigma(t)}_4 \le \norm{\dt \sigma(t)}_4 \le \sqrt{C \mathcal{Z}_0} e^{\lambda t}.
\end{equation}
Integrating this in time, we then find that
\begin{equation}\label{foe_3}
 \norm{\sigma(t)}_4 \le \norm{\sigma(0)}_4 + \frac{\sqrt{C \mathcal{Z}_0}}{\lambda} (e^{\lambda t}-1) \le C e^{\lambda t} \sqrt{ \norm{\sigma(0)}_4^2 +  \mathcal{Z}_0 }.
\end{equation}
The estimate \eqref{foe_0} then follows directly from \eqref{foe_1}, \eqref{foe_2}, and \eqref{foe_3}.

\end{proof}

%%%%%%%%%%%%%%%%%%%%%%%%%%%%%%%%%%%%%%%%%%%%%%%%%%%%%%%%
\subsection{Estimates for the inhomogeneous first-order problem}
%%%%%%%%%%%%%%%%%%%%%%%%%%%%%%%%%%%%%%%%%%%%%%%%%%%%%%%%%%

Consider the linear operators
\begin{equation}
\mathcal{L}_1 w = 4 \pi \rho_0^2 \dx w
\end{equation}
\begin{equation}
 \mathcal{L}_2 \sigma = 4 \pi r_0^4 \dx( \gamma K \rho_0^{\gamma-1} \sigma) - 4 r_0 \dx P_0 \int_0^x \frac{\sigma(y)}{\rho_0^2(y)}dy
\end{equation}
\begin{equation}
 \mathcal{L}_3 w = -16 \pi^2 r_0^4 \dx \left[\left(\frac{4\ep}{3}  + \delta \right)\rho_0 \dx w \right],
\end{equation}
and the corresponding matrix of operators
\begin{equation}
 \mathcal{L} = 
\begin{pmatrix}
 0 & \mathcal{L}_1 \\
 \mathcal{L}_2 & \mathcal{L}_3
\end{pmatrix}.
\end{equation}
We also  consider the boundary operator
\begin{equation}\label{bndry_op}
 \mathcal{B}(w) = -\frac{4\ep}{3} \left(4 \pi r_0^3 \rho_0 \dx\left( \frac{w}{r_0^3}\right) \right) - \delta (4 \pi \rho_0 \dx w).
\end{equation}
Notice that the first-order equations \eqref{first_order_1}--\eqref{first_order_2} are equivalent to the equation
\begin{equation}\label{mild_1}
 \dt
\begin{pmatrix}
 \sigma \\ w
\end{pmatrix}
+ \mathcal{L}
\begin{pmatrix}
 \sigma \\ w
\end{pmatrix}
= \begin{pmatrix}
 0 \\ 0
\end{pmatrix}
\end{equation}
with homogeneous boundary conditions
\begin{equation}\label{mild_2}
\frac{w}{r_0^2}(0,t) = \sigma(M,t)=0 \text{ and } \mathcal{B}(w) =0 \text{ at }x=M.
\end{equation}
Let us denote $e^{t\mathcal{L}}$ for the solution operator to \eqref{mild_1}--\eqref{mild_2}, i.e.
\begin{equation}\label{soln_op_def}
 e^{t \mathcal{L}} 
\begin{pmatrix}
 \sigma(0) \\ w(0)
\end{pmatrix}
= \begin{pmatrix}
 \sigma(t) \\ w(t)
\end{pmatrix},
\end{equation}
where $\sigma$ and $w$ solve \eqref{mild_1}--\eqref{mild_2} with initial data $\sigma(0)$ and $w(0)$.

Suppose now that $\sigma$ and $w$ solve the inhomogeneous problem
\begin{equation}\label{mild_form_in_1}
\dt
\begin{pmatrix}
 \sigma \\ w
\end{pmatrix}
+ \mathcal{L}
\begin{pmatrix}
 \sigma \\ w
\end{pmatrix}
= \begin{pmatrix}
 N_1 \\ N_2
\end{pmatrix} 
\end{equation}
along with the boundary conditions
\begin{equation}\label{mild_form_in_2}
\frac{w}{r_0^2}(0,t) = \sigma(M,t)=0 \text{ and } \mathcal{B}(w) = N_{\mathcal{B}} \text{ at }x=M.
\end{equation}
Here we assume that $N_1 = N_1(x,t)$, $N_2 = N_2(x,t)$, but that $N_{\mathcal{B}} = N_{\mathcal{B}}(t)$, i.e. the boundary inhomogeneity only depends on time.  In order to use the linear theory we have developed, we must rewrite this as system with homogeneous boundary conditions.  To accomplish this we will utilize the following lemma.

\begin{lem}\label{boundary_transform}
 Let 
\begin{equation}\label{psi}
 \psi(x,t) = -\frac{N_{\mathcal{B}}(t)}{3\delta} r_0^3(x).
\end{equation}
Then for each $t$, $\psi(t)$ satisfies $\mathcal{L}_3 \psi(t) =0$ for $x \in (0,M)$ and  $\mathcal{B}\psi(t) = N_{\mathcal{B}}(t)$ at $x=M$.  Also, $\mathcal{L}_1 \psi(t)= -N_{\mathcal{B}}(t) \rho_0(x)/\delta$. 
\end{lem}
\begin{proof}
The results follow from simple computations.
\end{proof}

With this $\psi$ in hand, we can reformulate \eqref{mild_form_in_1}--\eqref{mild_form_in_2} so that the resulting problem has homogeneous boundary conditions.  Let $w = \psi + \bar{w}$.  Then Lemma \ref{boundary_transform} implies that
\begin{equation}\label{mild_form_1}
\dt
\begin{pmatrix}
 \sigma \\ \bar{w}
\end{pmatrix}
+ \mathcal{L}
\begin{pmatrix}
 \sigma \\ \bar{w}
\end{pmatrix}
= \begin{pmatrix}
 N_1 \\ N_2
\end{pmatrix} 
+ 
\begin{pmatrix}
 -\mathcal{L}_1 \psi \\ - \dt \psi 
\end{pmatrix}
=  
\begin{pmatrix}
 N_1 + N_{\mathcal{B}} \rho_0/\delta \\ N_2 + \dt N_{\mathcal{B}} r_0^3/(3\delta)
\end{pmatrix}
\end{equation}
along with the boundary conditions
\begin{equation}\label{mild_form_2}
\frac{\bar{w}}{r_0^2}(0,t) = \sigma(M,t)=0 \text{ and } \mathcal{B}(\bar{w}) = 0 \text{ at }x=M.
\end{equation}
Employing the  variation of parameters, we can then solve \eqref{mild_form_1}--\eqref{mild_form_2} via
\begin{equation}
 \begin{pmatrix}
 \sigma(t) \\ \bar{w}(t)
\end{pmatrix} 
= e^{t \mathcal{L}} 
\begin{pmatrix}
 \sigma(0) \\ \bar{w}(0)
\end{pmatrix}
+ \int_0^t e^{(t-s) \mathcal{L}} 
\begin{pmatrix}
 N_1(s) \\ N_2(s)
\end{pmatrix}ds
+ \frac{1}{\delta} \int_0^t e^{(t-s) \mathcal{L}} 
\begin{pmatrix}
N_{\mathcal{B}}(s) \rho_0 \\  \dt N_{\mathcal{B}}(s) r_0^3/3
\end{pmatrix}ds.
\end{equation}
We can then go back to $w = \psi + \bar{w}$:
\begin{multline}\label{duhamel}
 \begin{pmatrix}
 \sigma(t) \\ w(t)
\end{pmatrix} 
= e^{t \mathcal{L}} 
\begin{pmatrix}
 \sigma(0) \\ \bar{w}(0)
\end{pmatrix}
- 
\frac{1}{\delta}
\begin{pmatrix}
 0 \\ N_{\mathcal{B}}(t)  r_0^3 /3
\end{pmatrix}
\\
+ \int_0^t e^{(t-s) \mathcal{L}} 
\begin{pmatrix}
 N_1(s) \\ N_2(s)
\end{pmatrix}ds
+ \frac{1}{\delta} \int_0^t e^{(t-s) \mathcal{L}} 
\begin{pmatrix}
N_{\mathcal{B}}(s) \rho_0 \\  \dt N_{\mathcal{B}}(s) r_0^3/3
\end{pmatrix}ds.
\end{multline}

Now let us define a norm for the pair $\sigma, w$ given by 
\begin{multline}\label{norm_0_def}
 \norm{
\begin{pmatrix}
 \sigma \\ w
\end{pmatrix}
}_0^2
:=
\hal \int_0^M \gamma K \rho_0^{\gamma-1} \abs{\frac{\sigma}{\rho_0}}^2 dx
+ \hal \int_0^M \frac{\abs{w}^2}{r_0^4} dx  \\
+ \hal  \int_0^M 16\pi^2 \left( \delta \rho_0 \abs{\dx  w}^2  + \frac{4\ep}{3} \rho_0 \abs{r_0^3 \dx \left( \frac{w}{r_0^3} \right)}^2 \right) dx.  
\end{multline}
We also define 
\begin{equation}\label{frak_E_def}
 \mathfrak{E}(\sigma,w) : = \norm{\begin{pmatrix}
 \sigma \\ w
\end{pmatrix}
}_0^2
+ \hal \int_0^M \gamma K \rho_0^{\gamma-1} \abs{\frac{\dt \sigma}{\rho_0}}^2 dx
+ \hal \int_0^M \frac{\abs{\dt w}^2}{r_0^4} dx.
\end{equation}
We can then recast the result of Theorem \ref{first_order_estimates} as 
\begin{equation}\label{mild_recast}
 \norm{ e^{t \mathcal{L}}
\begin{pmatrix}
 \sigma(0) \\ w(0)
\end{pmatrix}
}_0^2 
\le Ce^{2 \lambda t} \mathfrak{E}(\sigma(0),w(0)).
\end{equation}
Using these quantities and estimate \eqref{mild_recast}, we can record estimates for solutions to \eqref{mild_form_in_1}--\eqref{mild_form_in_2}.

\begin{thm}\label{mild_estimates}
Suppose that $\sigma$ and $w$ solve the inhomogeneous linear problem \eqref{mild_form_in_1}--\eqref{mild_form_in_2}.   Let $\psi$ be given by Lemma \ref{boundary_transform} and $\bar{w} =w  - \psi$.  Let $\norm{\cdot}_0$ and $\mathfrak{E}(\cdot,\cdot)$ be given by \eqref{norm_0_def} and \eqref{frak_E_def}, respectively.
Then
\begin{multline}\label{me_01}
 \norm{
\begin{pmatrix}
 \sigma(t) \\ w(t)
\end{pmatrix}
- e^{t \mathcal{L}}
\begin{pmatrix}
 \sigma(0) \\ \bar{w}(0)
\end{pmatrix}
}_0 
\le 
\int_0^t C e^{\lambda(t-s)} \sqrt{\mathfrak{E}(N_1(s),N_2(s))} ds
\\
+ \frac{C}{\delta} \abs{N_{\mathcal{B}}(t)} 
+ \frac{C}{\delta} \int_0^t e^{\lambda(t-s)} \left( \abs{N_{\mathcal{B}}(s)} + \abs{\dt N_{\mathcal{B}}(s)} + \abs{\dt^2 N_{\mathcal{B}}(s)} \right)ds.
\end{multline}

\end{thm}
\begin{proof}
From the above analysis, we know that $\sigma$ and $w$ are given by \eqref{duhamel}, where $e^{t\mathcal{L}}$ is the homogeneous  solution operator given by \eqref{soln_op_def}.  Hence \eqref{mild_recast} implies that
\begin{multline}\label{me_1}
 \norm{
\begin{pmatrix}
 \sigma(t) \\ w(t)
\end{pmatrix}
- e^{t \mathcal{L}}
\begin{pmatrix}
 \sigma(0) \\ \bar{w}(0)
\end{pmatrix}
}_0 
\le 
\int_0^t C e^{\lambda(t-s)} \sqrt{\mathfrak{E}(N_1(s),N_2(s))} ds \\
+
\frac{1}{ \delta}  
\norm{
\begin{pmatrix}
 0 \\ N_{\mathcal{B}}(t)  r_0^3 /3
\end{pmatrix}
}_0
+ \frac{1}{\delta} \int_0^t C e^{\lambda(t-s)} \sqrt{\mathfrak{E}(N_{\mathcal{B}}(s) \rho_0 , \dt N_{\mathcal{B}}(s) r_0^3/3) } ds.
\end{multline}
Then, since $N_{\mathcal{B}}(t)$ is only a function of time, not of $x$, we can easily estimate
\begin{equation}\label{me_2}
 \norm{
\begin{pmatrix}
 0 \\ N_{\mathcal{B}}(t)  r_0^3 /3
\end{pmatrix}
}_0 \le C \abs{N_{\mathcal{B}}(t)}
\end{equation}
and
\begin{equation}\label{me_3}
 \sqrt{\mathfrak{E}(N_{\mathcal{B}}(s) \rho_0 , \dt N_{\mathcal{B}}(s) r_0^3/3) } \le C \left( \abs{N_{\mathcal{B}}(s)} + \abs{\dt N_{\mathcal{B}}(s)} + \abs{\dt^2 N_{\mathcal{B}}(s)} \right),
\end{equation}
where $C>0$ in \eqref{me_2}--\eqref{me_3} is a constant depending on various (finite) integrals of $\rho_0$ and $r_0$.  The estimate \eqref{me_01} then follows by combining \eqref{me_1}--\eqref{me_3}.

\end{proof}

%%%%%%%%%%%%%%%%%%%%%%%%%%%%%%%%%%%%%%%%%%%%%%%%%%%%%%%%
\section{Nonlinear energy estimates}\label{4}
%%%%%%%%%%%%%%%%%%%%%%%%%%%%%%%%%%%%%%%%%%%%%%%%%%

%%%%%%%%%%%%%%%%%%%%%%%%%%%%%%%%%%%%%%%%%%%%%%%%%%%%%%%%
\subsection{Definitions}
%%%%%%%%%%%%%%%%%%%%%%%%%%%%%%%%%%%%%%%%%%%%%%%%%%

We are interested in small perturbations $\sigma$, $v$ around the stationary solution $\rho=\rho_0,$ $r=r_0$, and $v=0$.  In particular, we assume that 
\begin{equation}\label{assumption1}
\frac{9}{10}\rho_0\,\le\,\rho_0+\sigma\,\le\, \frac{11}{10}\rho_0.  
\end{equation}
This assumption will be justified later  when we close the nonlinear energy estimates.  For such small solutions, the Navier-Stokes-Poisson system \eqref{lagrangian_1} and \eqref{lagrangian_2} can be written as follows: 
\begin{equation}\label{PNSP}
\begin{split}
\dt\sigma +4\pi\rho^2\dx(r^2 v)&=0\\
\dt v +4\pi r^2\dx P- 4\pi r_0^2 \dx P_0 + \frac{x}{r^2}-\frac{x}{r_0^2}&= 16 \pi^2 r^2 \dx((4\ep/3+\delta)\rho \dx(r^2 v)).
\end{split}
\end{equation} 
The dynamics of $r$ are determined by 
\begin{equation}\label{r}
r(x,t) = \( \frac{3}{4\pi} \int_0^x \frac{dy}{\rho_0(y)+\sigma(y,t) }\)^{1/3}  \;\text{ and  }\; \dt r(x,t)=v(x,t).
\end{equation}
 It turns out that it is convenient to analyze $\sigma/\rho_0$ rather than $\sigma$ itself, so we rewrite the continuity equation as 
\begin{equation}\label{contE}
\frac{\rho_0}{\rho}\dt\left( \frac{\sigma}{\rho_0}\right)+4\pi\rho \dx (r^2v)=0.
\end{equation}
We will also rewrite the momentum equation.  To do so, we first note that 
\begin{equation}
\begin{split}
4\pi r^2\dx P- 4\pi r_0^2 \dx P_0 + \frac{x}{r^2}-\frac{x}{r_0^2}
=4\pi r^2 \dx(P-P_0) + x \( \frac{1}{r^2}-\frac{r^2}{r_0^4} \),
\end{split}
\end{equation}
and then note that for small perturbations satisfying \eqref{assumption1}, $P-P_0= K(\rho^\gamma-\rho_0^\gamma)$ can be written as 
\begin{equation}\label{expP}
P-P_0=\rho_0^\gamma \left\{ K\gamma\frac{\sigma}{\rho_0} + a_\ast \( \frac{\sigma}{\rho_0} \)^2 \right\}
\end{equation}
where $a_\ast$ is the smooth bounded remainder from the Taylor's Theorem.  We then rewrite the momentum equation as 
\begin{equation}\label{momE}
\dt v +4\pi r^2\dx \left\{ K\gamma \rho_0^\gamma \frac{\sigma}{\rho_0}+a_\ast \rho_0^\gamma \(\frac{\sigma}{\rho_0} \)^2 \right\} + x \( \frac{1}{r^2}-\frac{r^2}{r_0^4} \)=\mathcal{V},
\end{equation}
where $\mathcal{V}:=  16 \pi^2 r^2 \dx((4\ep/3+\delta)\rho \dx(r^2 v))$. We give an equivalent expression for  $\mathcal{V}$ so that it appreciates the boundary condition \eqref{bc} in energy estimates:
\begin{equation}\label{visc}
\V = 16 \pi^2 r^2 \dx \W + \frac{4\ep}{3} 12 \pi r^2 \dx \left( \frac{v}{r}\right),
\end{equation}
where
\begin{equation}
 \W =\delta \rho \dx(r^2 v) + \frac{4\ep}{3} \rho r^3 \dx\left(\frac{v}{r} \right) 
\end{equation}
satisfies $\W(M) = 0$ because of the boundary condition \eqref{bc}.  We use $\nu$ to denote the minimal viscosity coefficient: 
\begin{equation}\label{nu}
\nu:= \min \left\{\delta, \frac{4\ep}{3} \right\}.
\end{equation}

We now define instant energy functionals for $\sigma$ and $v$.  In what follows, all of the integrals are understood to be over the interval $[0,M]$. 
\begin{equation*}\label{energies_1}
\begin{split}
\mathcal{E}^0
&:= \hal \int \abs{v}^2 dx 
+ \hal \int \frac{K \gamma \rho_0^{\gamma-1}}{(1 + \frac{\sigma}{\rho_0})^2} \abs{\frac{\sigma}{\rho_0}}^2 dx 
+\hal \int \nu \abs{1-\frac{r_0}{r}}^2 dx :=\mathcal{E}^{0,v}+\mathcal{E}^{0,\sigma}+\mathcal{E}^{0,r}\\
\mathcal{E}^1
&:= \hal \left[\delta \int 16\pi^2 \rho \abs{\dx(r^2v)}^2 dx +\frac{4\ep}{3} \int 16 \pi^2 \rho r^6 \abs{\dx \(\frac{v}{r} \)}^2dx \right] \\
&\quad+ \hal \int \(\delta+\frac{4\ep}{3}\) 16 \pi^2 r^4 \frac{1}{1+\sir}  \abs{\dx \(\frac{\sigma}{\rho_0}\) }^2 dx :=\mathcal{E}^{1,v}+\mathcal{E}^{1,\sigma}\\
\mathcal{E}^2&:= \hal \int \abs{\dt v}^2 dx 
+ \hal \int \frac{K\gamma \rho_0^{\gamma-1}}{(1+\sir)^2} \abs{\dt\(\sir\)}^2 dx  :=\mathcal{E}^{2,v} + \mathcal{E}^{2,\sigma} \\
\mathcal{E}^3& := \hal\left[ \delta \int 16\pi^2 \rho \abs{\dx(r^2\dt v)}^2dx + \frac{4\ep}{3} \int 16 \pi^2 \rho r^6 \abs{\dx \( \frac{\dt v}{r} \)}^2 dx \right] \\
\mathcal{E}^{4}& := \hal \int \(\delta+\frac{4\ep}{3}\) 4\pi\rho_0  \abs{\dx \( r^4 \dx \( \frac{\sigma}{\rho_0} \) \)}^2 dx 
\end{split}
\end{equation*}
The corresponding dissipations are given by 
\begin{equation*}
\begin{split}
\mathcal{D}^0 & := \delta \int 16\pi^2 \rho \abs{\dx(r^2v)}^2 dx + \frac{4\ep}{3} \int 16 \pi^2 \rho r^6 \abs{\dx \( \frac{v}{r} \)}^2 dx  \\
\mathcal{D}^1& := \int \abs{\dt v}^2 dx + \int 16\pi^2 K\gamma r^4 \rho_0^\gamma \abs{\dx \( \sir \)}^2 dx :=\mathcal{D}^{1,v} + \mathcal{D}^{1,\sigma}\\
\mathcal{D}^2& := \delta \int 16\pi^2 \rho \abs{\dx(r^2\dt v)}^2dx + \frac{4\ep}{3} \int 16\pi^2 \rho r^6 \abs{\dx \( \frac{\dt v}{r} \)}^2 dx   \\
\mathcal{D}^3&:= \int \abs{\dt^2 v}^2 dx \\
\mathcal{D}^{4}&:= \int 4\pi K\gamma r^2 \rho \rho_0^\gamma \abs{\dx \( r^4\dx\(\sir\)\)}^2 dx.
\end{split}
\end{equation*}
In addition, we introduce various bootstrapped and auxiliary energies and dissipations (denoted with subscript $b$ and $a$, respectively) that can be controlled with the above instant energies and dissipations: 
\begin{equation}
\begin{split}
\mathcal{E}^{0,r}_b & := \int  \frac{\nu}{\rho} \abs{1-\frac{r_0}{r}}^2dx \\
\mathcal{E}^{1,\sigma}_b& := \hal \int \(\delta+\frac{4\ep}{3}\) 16 \pi^2 \frac{r^2}{\rho}  \abs{\dx\( \frac{\sigma}{\rho_0} \)}^2 dx \\ 
\mathcal{D}^{1,\sigma}_b &:=\int 16\pi^2 K\gamma r^2 \rho_0^{\gamma-1} \abs{\dx\(\sir\)}^2 dx 
\end{split}
\end{equation}
\begin{equation}
\begin{split}
\mathcal{E}^{3,\sigma}_a &:= \int \(\delta+\frac{4\ep}{3}\)^2 16\pi^2 \frac{r^2}{\rho} \abs{ \dx \dt \(\frac{\sigma}{\rho_0}\)}^2 dx \\
\mathcal{E}^{3,v}_{a_1}& :=\int  \frac{r^2}{\rho} \abs{\dx( \rho \dx(r^2v))}^2 dx ; \quad
\mathcal{E}^{3,v}_{a_2}  = \int  \rho r^6 \abs{\dx\(\rho r^3 \dx\(\frac{v}{r}\)\)}^2 dx
\end{split}
\end{equation}
\begin{equation}
\begin{split}
\mathcal{E}^{4}_{a_1}&:= \int \(\delta+\frac{4\ep}{3}\) 4\pi\rho_0   \abs{\dx\( r^4 \dt\dx \(\frac{\sigma}{\rho_0} \) \)}^2 dx \\
\mathcal{E}^4_{a_2} &:=\int 16\pi^2\rho_0 \abs{\dx (r^4\dx (\rho\dx (r^2v)))}^2 dx.
\end{split}
\end{equation}

Finally, we introduce some bootstrap energies that depend on a parameter $\beta \in \Rn{}$:
\begin{equation}
  \mathcal{E}^{0,\sigma}_\beta  := \int \frac{\rho_0^{\beta+1}}{\rho}\abs{\sir}^2dx ; \quad \mathcal{E}^{2,\sigma}_\beta  :=\int  \frac{\rho_0^{\beta+1}}{\rho} \abs{\dt\(\sir\) }^2dx. 
\end{equation}

For the proof of our instability in Section \ref{5}, we will need to invoke higher order energy functionals and dissipations, which we define now.   For $i=2$ and $3$, let
\begin{equation}
\begin{split}
\mathcal{E}^{1+2i}& := \hal\int \abs{\dt^i v}^2 dx +\hal \int \frac{K\gamma \rho_0^{\gamma-1}}{(1+\sir)^2} \abs{\dt^i\(\sir\)}^2 dx\\
\mathcal{D}^{1+2i}& := \delta \int 16\pi^2 \rho \abs{\dx(r^2\dt^iv)}^2 dx +\frac{4\ep}{3}\int 16\pi^2 \rho r^6 \abs{\dx \( \frac{\dt^i v}{r} \)}^2dx \\
\mathcal{E}^{2+2i}&:=\hal \left[ \delta \int 16\pi^2 \rho \abs{\dx(r^2\dt^i v)}^2 dx +\frac{4\ep}{3}\int 16\pi^2 \rho r^6 \abs{\dx \(\frac{\dt^i v}{r}\)}^2dx \right]\\
\mathcal{D}^{2+2i}&:=\int \abs{\dt^{i+1}v}^2dx.
\end{split}
\end{equation}
Next we define bootstrapped energies and auxiliary energies for $i=2$ and $3$
\begin{equation}
\begin{split}
\mathcal{E}^{1+2i,\sigma}_{-1}& := \int \frac{1}{\rho_0}\abs{\dt^i\(\sir\)}^2dx\;;\;  
\mathcal{E}^{1+2i,v}_{a} := \int r^4 \abs{\dx(\rho \dx(\dt^{i-1}[r^2 v]))}^2dx\\
\mathcal{E}^{2+2i,\sigma}_{a}&:= \int  \(\delta+\frac{4\ep}{3} \)^2 16\pi^2r^2 \abs{\dx\dt^i \(\sir\) }^2 dx.
\end{split}
\end{equation}
The then define the total energy by 
\begin{equation}\label{total_energy}
\begin{split}
\mathcal{E}:=& \sum_{i=0}^8 \mathcal{E}^i +\mathcal{E}^{0,r}_b + \mathcal{E}^{1,\sigma}_b + \mathcal{E}^{0,\sigma}_{-1} + \mathcal{E}^{2,\sigma}_{-1}
+ \mathcal{E}^{5,\sigma}_{-1} + \mathcal{E}^{7,\sigma}_{-1} \\
& + \mathcal{E}^{3,\sigma}_a + \mathcal{E}^{3,v}_{a_1} + \mathcal{E}^{3,v}_{a_2} + \mathcal{E}^{4}_{a_1} + \mathcal{E}^{4}_{a_2} + \mathcal{E}^{5,\sigma}_a + \mathcal{E}^{6,\sigma}_{a} + \mathcal{E}^{7,\sigma}_a  + \mathcal{E}^{8,\sigma}_{a}.
\end{split}
\end{equation} 

Throughout the rest of the section, we assume that 
\begin{multline}\label{assumption2}
\pnorm{\sir}{\infty} + \pnorm{ \dt \left(\sir\right)}{\infty} +\pnorm{ \dt^2 \left(\sir\right)}{\infty} + \pnorm{ \dt^3 \left(\sir\right)}{\infty} 
+ \pnorm{1- \frac{r_0}{r}}{\infty} \\ + \pnorm{\rho r^3 \dx\(\frac{v}{r}\)}{\infty} + \pnorm{\rho r^3\dx\left(\sir\right)}{\infty}+ \pnorm{\frac{v}{r}}{\infty} + \pnorm{\frac{\dt v}{r}}{\infty} +\pnorm{\frac{\dt^2 v}{r}}{\infty}  \le  \theta_1
\end{multline}
for sufficiently small constant $\theta_1$, where the norm $\pnorm{\cdot}{\infty}$ is over the spatial region $[0,M]$.  The validity of this assumption within the total energy $\mathcal{E}$ will be justified in Lemma \ref{lem4.6}. 

Since $r$ is determined through an integral of $\sigma$ as in \eqref{r}, for small perturbations satisfying \eqref{assumption1} we may use Taylor's Theorem to write $\frac{r_0}{r}$  as 
\begin{equation}\label{r_0r}
\frac{r_0}{r}= 1+\frac{1}{4\pi r_0^3}\int_0^x \frac{\sigma}{\rho_0^2} dy + \frac{c_1}{r_0^3} \int_0^x \frac{1}{\rho_\ast}\(\sir\)^2dy + \frac{c_2}{r_0^6}\( \int_0^x \frac{\sigma}{\rho_0^2} dy\)^2
\end{equation}
where $\rho_\ast/\rho_0\sim 1$ is a bounded smooth function of $\sir$.  Hence the $1-r_0/r$ estimate (up to a constant) in \eqref{assumption2} can actually be guaranteed by the smallness of the other terms in \eqref{assumption2}.

The  relation \eqref{r_0r} will be useful in various places.  We now record a couple other useful identities.

\textit{Dynamics of $\frac{r_0}{r}$}: From \eqref{r} we have that
\begin{equation}\label{rr}
\begin{split}
\dt\(\frac{r_0}{r}\) &=-\frac{r_0v}{r^2}=-\(\frac{v}{r}\) \(\frac{r_0}{r}\)\\
\dx\( \frac{r_0}{r} \) &=\frac{1}{4\pi\rho_0 r_0^2r}-\frac{r_0}{4\pi\rho r^4}=\frac{1-(\frac{r_0}{r})^3+\frac{\sigma}{\rho_0}}{4\pi\rho r_0^2r}.
\end{split}
\end{equation}

\textit{Some useful inequalities and identities:} For any $v$ (not just solutions), 
\begin{equation}
\begin{split}\label{ID}
&\frac{v}{r}=\frac{4\pi}{3}\left \{\rho \dx(r^2v)-\rho r^3\dx\(\frac{v}{r}\) \right\} \;\Rightarrow\; \frac{v^2}{\rho r^2}\le \frac{32\pi^2}{9}\left \{\rho \abs{\dx(r^2v)}^2 + \rho r^6 \abs{\dx\( \frac{v}{r} \)}^2 \right\}\\
&\rho\dx(r v^2)= \rho\dx \left[(r^2v)^2 \cdot \frac{1}{r^3}\right] = 2\frac{v}{r}\rho\dx(r^2v) -\frac{3}{4\pi}\frac{v^2}{ r^2}\\
&\quad\quad\quad\;\;\, = \frac{v}{r} \left\{ \rho\dx(r^2v) + \rho r^3\dx\(\frac{v}{r}\) \right\}.  \end{split}
\end{equation}

%%%%%%%%%%%%%%%%%%%%%%%%%%%%%%%%%%%%%%%%%%%%%%%%%%%%%%%%
\subsection{Estimates}
%%%%%%%%%%%%%%%%%%%%%%%%%%%%%%%%%%%%%%%%%%%%%%%%%%

Throughout the rest of the section, we use $C$ to denote a generic constant that may differ from line to line, and $\eta$ to denote a sufficiently small fixed constant which will be determined later.  The constants $C$ are allowed to depend on $\eta$, which presents no trouble in our ultimate analysis since first we will fix an $\eta$, which then fixes the constants. 

In the following series of lemmas, we provide the energy inequalities for ${\mathcal{E}}$. 
We start with $\mathcal{E}^0$ and $\mathcal{D}^0$. 

\begin{lem}\label{lem4.1}
\begin{equation}\label{E^0}
\frac{d}{dt}\mathcal{E}^0 +\mathcal{D}^0 \leq C(1+ \theta_1)\mathcal{E}^0+ \hal\mathcal{D}^0
\end{equation}
\end{lem}

\begin{proof}
Multiply \eqref{momE} by $v$ and integrate to get 
\begin{multline}
\hal\frac{d}{dt}\int \abs{v}^2 dx\underbrace{ -\int 4\pi \dx(r^2 v) \left\{K\gamma\rho_0^\gamma \frac{\sigma}{\rho_0} +a_\ast\rho_0^\gamma \( \frac{\sigma}{\rho_0} \)^2 \right\} dx}_{(i)} + \underbrace{\int v\frac{x(r_0^4-r^4)}{r^2r_0^4} dx}_{(ii)} \\
= \underbrace{\int v \mathcal{V}dx}_{(iii)}.
\end{multline}
For $(i)$ we use \eqref{contE} to see that 
\begin{equation}
\begin{split}
(i) &= \int \frac{\rho_0}{\rho^2} \dt\( \frac{\sigma}{\rho_0}\) \left\{ K\gamma \rho_0^\gamma \frac{\sigma}{\rho_0} +a_\ast\rho_0^\gamma \(\frac{\sigma}{\rho_0}\)^2 \right\} dx \\
&=\frac{1}{2}\frac{d}{dt} \int \frac{K\gamma\rho_0^{\gamma-1}}{(1+\frac{\sigma}{\rho_0})^2} \abs{\frac{\sigma}{\rho_0}}^2 dx 
+ \int \frac{(K\gamma+(1+\sir)a_\ast)\rho_0^{\gamma-1}}{(1+\sir)^3} \abs{\sir}^2\dt\(\sir\)dx.
\end{split}
\end{equation}
However, 
\begin{equation}
\abs{\int \frac{(K\gamma+(1+\sir)a_\ast)\rho_0^{\gamma-1}}{(1+\sir)^3} \abs{\sir}^2\dt\(\sir\)dx }\le  C(1+\theta_1) \E^0.
\end{equation}
For $(ii)$, the Cauchy-Schwarz inequality yields 
\begin{equation}
\begin{split}
\abs{(ii)} &\le \nu\int \frac{\abs{v}^2}{\rho r^2} dx+ \frac{1}{\nu} \int \rho  \abs{\frac{x}{r_0^4}(r^2+r_0^2)(r+r_0)}^2 \abs{1-\frac{r_0}{r}}^2 dx \\
&\le  \frac29\mathcal{D}^0+C \mathcal{E}^{0,r},
\end{split}
\end{equation}
where we have used \eqref{ID} at the second inequality. From \eqref{visc} and the boundary condition $\W(M,t)=0$, we get 
\begin{equation}
(iii)=- \delta \int 16\pi^2 \rho \abs{\dx(r^2v)}^2dx -\frac{4\ep}{3}\int 16\pi^2 \rho r^6 \abs{\dx\(\frac{v}{r}\)}^2dx = -\D^0. 
\end{equation}
Next, from \eqref{rr}, 
\begin{multline}
\frac{\nu}{2}\frac{d}{dt}\int \abs{1-\frac{r_0}{r}}^2dx= -\nu\int \frac{v}{r}\frac{r_0}{r} \(1-\frac{r_0}{r}\)dx \\ 
\le \nu \int \frac{\abs{v}^2}{\rho r^2}dx +\nu\int \rho \abs{\frac{r_0}{r}}^2 \abs{1-\frac{r_0}{r}}^2 dx
\le \frac29\mathcal{D}^0+C\mathcal{E}^{0,r}.
\end{multline}
The desired estimate then follows by combining these estimates.
\end{proof}\

With Lemma \ref{lem4.1} in hand, we can bootstrap to control $\sir$ with an improved weight.  Multiply \eqref{contE} by $\rho_0^\beta \sir$ and integrate to get 
\begin{equation}
\int \frac{\rho_0^{\beta+1}}{\rho} \sir \dt \(\sir\) dx=-\int\rho^{\hal} \rho_0^{\beta}\sir \cdot  4\pi \rho^{\hal}\dx(r^2v) dx.
\end{equation}
Thus 
\begin{multline}
\hal\frac{d}{dt}\int \frac{\rho_0^{\beta+1}}{\rho} \abs{\sir}^2 dx \le \frac{C}{\eta}\int 16\pi^2\rho \abs{\dx(r^2v)}^2 dx + \eta \int \rho\rho_0^{2\beta}\abs{\sir}^2 dx \\
 -\hal\int \frac{\rho_0^{\beta+1}}{\rho^2}\dt\sigma \abs{\sir}^2 dx,
\end{multline}
which means that
\begin{equation}\label{E^0b}
\frac{d}{dt}\mathcal{E}^{0,\sigma}_\beta \le \frac{C}{\eta} \mathcal{D}^0 + \eta \mathcal{E}^{0,\sigma}_{2\beta+1}+C\theta_1\mathcal{E}^{0,\sigma}_\beta.
\end{equation}

Next we consider $\mathcal{E}^1$ and $\mathcal{D}^1$. 

\begin{lem}\label{lem4.2}
We have that
\begin{equation}\label{E^1}
\frac{d}{dt}\mathcal{E}^1 +\mathcal{D}^1\leq (\eta+C\theta_1)\mathcal{E}^1 +\hal\mathcal{D}^1+  C(\mathcal{E}^0+\mathcal{E}^{0,\sigma}_0)+q\mathcal{E}^{2,v},
\end{equation}
where 
\begin{equation}
q:=q_1+q_1 := \pnorm{ 16\pi^2 \( K\gamma +\frac{4}{K\gamma}a_\ast^2\(\sir\)^2  \) \rho_0^\gamma}{\infty}  + \pnorm{\frac{\(1+\sir\)}{\eta\(\delta+\frac{4\ep}{3}\)} }{\infty}
\end{equation}
is bounded due to \eqref{assumption2}.
\end{lem}

\begin{proof} 

We divide the proof into steps.

\textbf{Step 1} -- $\E^{1,v}$ and $\D^{1,v}$
 
Multiply \eqref{momE} by $\dt v$ and integrate to get 
\begin{multline}
\int |\dt v|^2 dx \underbrace{+\int 4\pi r^2 \dt v \dx \left\{ K\gamma\rho_0^\gamma \frac{\sigma}{\rho_0} +a_\ast\rho_0^\gamma \(\frac{\sigma}{\rho_0}\)^2 \right\} dx}_{(iv)} \\
+ \underbrace{\int \dt v\frac{x(r_0^4-r^4)}{r^2r_0^4} dx}_{(v)}
-\underbrace{\int \dt v \mathcal{V}dx}_{(vi)}=0.
\end{multline}
For $(iv)$ we first expand
\begin{equation}
\begin{split}
(iv)&=\int 4\pi r^2 \dt v \left\{K\gamma \dx(\rho_0^\gamma) \frac{\sigma}{\rho_0} + K \gamma \rho_0^\gamma \dx\(\frac{\sigma}{\rho_0}\) \right\} dx\\
&+ \int 4\pi r^2 \dt v \left\{ \dx\(a_\ast\rho_0^\gamma \) \(\frac{\sigma}{\rho_0}\)^2 + 2 a_\ast \rho_0^\gamma \frac{\sigma}{\rho_0}\dx\(\sir\) \right\} dx\\
&:=(iv)_1+(iv)_2+(iv)_3+(iv)_4
\end{split}
\end{equation}
and then estimate
\begin{equation}
|(iv)_1+(iv)_3|\le \frac{1}{4} \int |\dt v|^2 dx +{ C(1+\theta_1^2)\int  \abs{\sir}^2 dx} \le \frac{1}{4} \mathcal{D}^{1,v}+C\mathcal{E}^{0,\sigma}_0
\end{equation}
and
\begin{equation}
\begin{split}
 |(iv)_2+(iv)_4|&\le \int 8\pi^2 K\gamma r^4 \rho_0^{\gamma} \abs{\dx\(\sir\)}^2 dx \\ 
&+ \int 8\pi^2\(K\gamma +\frac{4}{K\gamma}  a_\ast^2\(\sir\)^2\) \rho_0^\gamma \abs{\dt v}^2 dx \\
& \le \frac{1}{2}\mathcal{D}^{1,\sigma} + q_1 \E^{2,v}.
\end{split}
\end{equation}

For $(v)$, we get 
\begin{equation}
|(v)|\le \frac{1}{4} \int \abs{\dt v}^2 dx+\int \abs{\frac{x}{r^2r_0^4}(r^2+r_0^2)(r+r_0)}^2 \abs{r-r_0}^2 dx \le \frac{1}{4} \mathcal{D}^{1,v} + C\mathcal{E}^{0,r}.
\end{equation}
The term $(vi)$ forms the energy  $\mathcal{E}^{1,v}$ and nonlinear commutators:  
\begin{equation}
\begin{split}
(vi)&=-\hal\frac{d}{dt} \left[ \delta \int 16\pi^2 \rho |\dx(r^2 v)|^2dx + \frac{4\ep}{3}\int 16\pi^2 \rho r^6 \abs{\dx\(\frac{v}{r}\) }^2 dx \right] \\
&\quad + \hal \left[ \delta \int 16\pi^2 \dt \sigma \abs{\dx(r^2 v)}^2dx +\frac{4\ep}{3}\int 16\pi^2 \dt(\rho r^6) \abs{\dx\(\frac{v}{r}\)}^2dx \right]\\
&\quad+\delta\int 16\pi^2 \dx(v\cdot 2rv) \rho\dx(r^2 v) dx 
+ \frac{4\ep}{3}\int 16\pi^2 \rho r^6  \dx\(v\cdot \(-\frac{v}{r^2}\)\) \dx \(\frac{v}{r}\) dx. 
\end{split}
\end{equation}
Using \eqref{assumption2} and the fact that
\begin{equation}
 \begin{split}
  & \dx(v\cdot 2rv) = 2\( \frac{v}{r}\dx(r^2v)+r^2v\dx(\frac{v}{r}) \) \\
  & \dx\(v\cdot \(-\frac{v}{r^2}\)\) \dx\(\frac{v}{r}\) = -2\frac{v}{r} \abs{\dx\(\frac{v}{r}\)}^2,
 \end{split}
\end{equation}
the absolute values of the second and third lines may be bounded by $C\theta_1\mathcal{E}^{1,v}$.

We may now combine the above to deduce that
\begin{equation}\label{l42_1}
 \frac{d}{dt} \E^{1,v} + \D^{1,v} \le \hal \mathcal{D}^{1} + C\theta_1\mathcal{E}^{1,v} + C\mathcal{E}^{0,r}   + C\mathcal{E}^{0,\sigma}_0 + q_1 \E^{2,v}.
\end{equation}

\textbf{Step 2} -- $\E^{1,\sigma}$ and $\D^{1,\sigma}$

For the estimate of $\dx(\sir)$, we first rewrite \eqref{momE} by replacing $\rho\dx(r^2v)$ in $\mathcal{V}$  by $\dt(\sir)$ through the continuity equation \eqref{contE}: 
\begin{equation}\label{dxSig}
\begin{split}
&\(\delta + \frac{4\ep}{3}\) 4\pi r^2 \left\{ \frac{\rho_0}{\rho} \dt\dx \(\sir\) + \dx \(\frac{\rho_0}{\rho}\) \dt\(\sir\) \right\} + 
\dt v+ \frac{x(r_0^4-r^4)}{r^2r_0^4} \\&
+ 4\pi r^2 \left\{ K\gamma\rho_0^\gamma \dx \(\frac{\sigma}{\rho_0} \) + K \gamma \dx \(\rho_0^\gamma\) \frac{\sigma}{\rho_0} + \dx (a_\ast\rho_0^\gamma) \(\frac{\sigma}{\rho_0}\)^2 +
2a_\ast \rho_0^\gamma \frac{\sigma}{\rho_0} \dx\(\sir\) \right\} =0.
\end{split}
\end{equation}
Note that $\dx(\frac{\rho_0}{\rho})=-(1+\sir)^{-2}\dx(\sir)$. 
Multiplying \eqref{dxSig} by $4\pi r^2\dx(\sir)$ and integrating, we are led to the estimate 
\begin{equation}\label{l42_2}
\begin{split}
\hal\frac{d}{dt} \int \(\delta+\frac{4\ep}{3}\) 16 \pi^2 r^4 \frac{1}{1+\sir}  \abs{ \dx\( \frac{\sigma}{\rho_0} \)}^2 dx  + \int 16\pi^2 K\gamma r^4 \rho_0^\gamma \abs{\dx\(\sir\)}^2 dx \\
 \le  (\eta+C\theta_1)\mathcal{E}^{1,\sigma}+C(\mathcal{E}^{0,r}+ \mathcal{E}^{0,\sigma}_0)+\int \frac{(1+\sir)}{2\eta(\delta+\frac{4\ep}{3})} |\dt v|^2 dx.
\end{split}
\end{equation}
Note that the last term in \eqref{l42_2} may be bounded by $q_2 \E^{2,v}$.  We then obtain \eqref{E^1} by combining \eqref{l42_1} and \eqref{l42_2}.
\end{proof}

The estimate \eqref{E^1} is not of a closed form by itself. Its use will come by coupling to the result of the following lemma.  

\begin{lem}\label{lem4.3}
\begin{equation}\label{E^2}
\frac{d}{dt}\mathcal{E}^2+\mathcal{D}^2\le (\eta+C\theta_1)\mathcal{E}^2+C\theta_1\mathcal{D}^0 +(\frac14+C\theta_1)\mathcal{D}^2+ C(\theta_1\mathcal{E}^{0,\sigma}_0+\mathcal{E}^0 + \E^1)
\end{equation}
\end{lem}

\begin{proof} 
We take $\dt$ of \eqref{momE} to see that 
\begin{equation}\label{dt2v}
\begin{split}
\dt^2 v + 4\pi r^2 \dx \left\{ K\gamma \rho_0^\gamma \dt\( \frac{\sigma}{\rho_0}\) + 2 a_\ast \rho_0^\gamma \sir \dt\( \frac{\sigma}{\rho_0} \)+ \dt a_\ast \rho_0^\gamma \(\frac{\sigma}{\rho_0}\)^2 \right\} \\
+8 \pi rv \left\{ K\gamma \rho_0^\gamma \dx\( \frac{\sigma}{\rho_0} \) +K\gamma \dx\( \rho_0^\gamma \) \frac{\sigma}{\rho_0}+ \dx (a_\ast\rho_0^\gamma) \(\frac{\sigma}{\rho_0} \)^2 + 
2 a_\ast\rho_0^\gamma \frac{\sigma}{\rho_0} \dx\(\sir\) \right\} \\
-2 \frac{xv(r_0^4-r^4)}{r^3r_0^4} - 4\frac{xrv}{r_0^4}
= \dt \V.
\end{split}
\end{equation}
The energy estimate \eqref{E^2} may be derived from \eqref{dt2v} as in Lemma \ref{lem4.1}: we multiply \eqref{dt2v} by $\dt v$, integrate over $x \in  [0,M]$, and integrate various terms by parts in order to identify  $ d\mathcal{E}^2/dt$, $\mathcal{D}^2$, and some error (lower order or commutator) terms, the latter of which may be estimated by the right side of \eqref{E^2}.  Since the argument is essentially the same as that of Lemma \ref{lem4.1}, we present only a sketch.

First observe that the product of $\dt v$ with the first two terms in the first line in \eqref{dt2v} forms the energy term $\dt \mathcal{E}^2$ and some error terms:
\begin{multline}
 \int \dt v \left[ \dt^2 v + 4\pi r^2 \dx\left(K\gamma \rho_0^\gamma \dt\left(\frac{\sigma}{\rho_0}\right)\right)  \right]dx \\ 
= 
\hal\frac{d}{dt}\left\{\int |\dt v|^2 dx + \int \frac{K\gamma \rho_0^{\gamma-1}}{(1+\sir)^2}\abs{\dt\(\sir\)}^2dx  \right\} 
+ \Z,
\end{multline}
where $\Z$ is a term whose absolute value may be estimated by the right side of \eqref{E^2}.  Here we have used the continuity equation \eqref{contE} and an integration by parts on the second term.

Next, we compute
\begin{equation}\label{dtV_comp}
 \dt \V = 16 \pi^2 r^2 \dx \dt \W + 16 \pi^2 (2 r v) \dx \W + \frac{4\ep}{3} 12 \pi r^2 \dx \left( \frac{\dt v}{r} \right)
\end{equation}
and note that the boundary condition $\W(M,t) = 0$ implies that $\dt \W(M,t) =0$ as well.  This allows us to integrate by parts without introducing boundary terms:
\begin{equation}
 \int 16 \pi^2 r^2 \dx \dt \W \dt v dx = - \int 16 \pi^2 (r^2 \dt v)  \dt \W dx.
\end{equation}
Using this, we find that
\begin{equation}
 \int \dt \V \dt v dx = -\delta \int 16\pi^2 \rho |\dx(r^2\dt v)|^2dx + \frac{4\ep}{3}\int 16\pi^2 \rho r^6 \abs{\dx\(\frac{\dt v}{r}\)}^2dx + \Z,
\end{equation}
where again $\Z$ is an error term with the property that $\abs{\Z}$ is bounded by the right side of \eqref{E^2}.

Finally, all of the remaining terms that arise when we multiply \eqref{dt2v} by $\dt v$ can also be estimated by the right side of \eqref{E^2}.  For example, the second term in the third line can be estimated by noting that ${xr}/{r_0^4}$ is bounded, which means that 
\begin{equation}
-\int 4 \frac{xrv}{r_0^4} \dt v dx\le \eta\int |\dt v|^2 dx+ C\int |v|^2 dx\le \eta \mathcal{E}^2+ C\mathcal{E}^0.
\end{equation}
Combining all of this, we find that
\begin{equation}
\begin{split}
&\hal\frac{d}{dt} \left\{\int |\dt v|^2 dx + \int \frac{K\gamma \rho_0^{\gamma-1}}{(1+\sir)^2} \abs{\dt\(\sir\)}^2 dx  \right\} \\
&\quad+   \delta \int 16\pi^2 \rho |\dx(r^2\dt v)|^2dx +\frac{4\ep}{3}\int 16\pi^2 \rho r^6 \abs{\dx\(\frac{\dt v}{r}\)}^2 dx\\
&\le  \( \frac{1}{4} +C \theta_1 \) \mathcal{D}^2 + \eta \mathcal{E}^2 + C\theta_1 (\mathcal{E}^2 +\mathcal{D}^0) +C \theta_1 \mathcal{E}^{0,\sigma}_0 + C (\mathcal{E}^0 + \E^1),
\end{split} 
\end{equation}
which yields \eqref{E^2}. 
\end{proof}

We now derive bootstrapped estimates for  $\dt(\sir)$.   We take $\dt$ of \eqref{contE} to get 
\begin{equation}\label{dtSig}
\frac{\rho_0}{\rho} \dt^2\(\sir\) = - 4 \pi \rho \dx(r^2 \dt v ) - 8 \pi \rho \dx (r v^2) - 4 \pi \dt \sigma \dx(r^2v) + \frac{\rho_0^2}{\rho^2}\(\dt\(\sir\)\)^2.
\end{equation}
Next, we multiply \eqref{dtSig} by $\rho_0^\beta \dt( \sir)$ and integrate to see that 
\begin{multline}
\hal\frac{d}{dt}\int \frac{\rho_0^{\beta+1}}{\rho} \abs{\dt\(\sir\)}^2 dx \le 
\frac{C}{\eta}\int 16\pi^2\rho \abs{\dx(r^2\dt v)}^2 dx 
+ \eta \int \rho\rho_0^{2\beta} \abs{\dt\(\sir\)}^2 dx \\
-\int 8\pi \frac{v}{r} \left\{ \rho \dx(r^2v)+ \rho r^3 \dx\(\frac{v}{r}\) \right\} \rho_0^\beta \dt\(\sir\)dx  + \frac{3}{2} \int \frac{\rho_0^{\beta+2}}{\rho^2}\( \dt \(\sir\)\)^3 dx \\
-4\pi \int \dt \sigma \dx(r^2 v) \rho_0^\beta \dt\( \sir\) dx.
\end{multline}
Then we estimate
\begin{multline}
 -\int 8\pi \frac{v}{r} \left\{ \rho \dx(r^2v)+ \rho r^3 \dx\(\frac{v}{r}\) \right\} \rho_0^\beta \dt\(\sir\)dx \\
-4\pi \int \dt \sigma \dx(r^2 v) \rho_0^\beta \dt\( \sir\) dx  \le C\theta_1\mathcal{D}^0 + C \theta_1 \E^{2,\sigma}_{2\beta+1}   
\end{multline}
to obtain 
\begin{equation}\label{E^2b}
\frac{d}{dt}\mathcal{E}^{2,\sigma}_\beta \le \frac14\mathcal{D}^2+C\theta_1\mathcal{D}^0+(\eta + C\theta_1) \mathcal{E}^{2,\sigma}_{2\beta+1} + C\theta_1\mathcal{E}^{2,\sigma}_\beta.
\end{equation}

Next we estimate $\mathcal{E}^3$ and $\mathcal{D}^3$. 

\begin{lem}\label{lem4.4}
There exists an energy $\F^3$ so that
\begin{equation}\label{E^3}
\frac{d}{dt}[\mathcal{E}^3+\F^3]+\mathcal{D}^3\le C\theta_1\mathcal{E}^3 +\(\frac{3}{8} + \frac{\theta_1}{4}\) \mathcal{D}^3 + C(\mathcal{E}^{2,\sigma}_0+\mathcal{E}^2 + \mathcal{E}^1 + \mathcal{E}^0).
\end{equation}
Moreover, we have the estimate $\abs{\F^3} \le C\theta_1(\mathcal{E}^3+\mathcal{E}^1)$.
\end{lem}

\begin{proof}

First recall \eqref{dt2v} and rewrite it as 
\begin{equation}\label{l44_1}
\begin{split}
\dt^2 v &+ \underbrace{4\pi r^2 \left\{ K\gamma \rho_0^\gamma \dt\dx\(\sir\) + K \gamma \dx(\rho_0^\gamma) \dt\(\sir\)  \right\} }_{(a_1)}\\
&+\underbrace{4\pi r^2 \left\{ \dt \left[ \dx (a_\ast \rho_0^\gamma)\(\sir\)^2 + 2(a_\ast\rho_0^\gamma) \sir \dx\(\sir\) \right] \right\} }_{(a_2)}\\
&+\underbrace{8\pi rv \left\{ K\gamma\rho_0^\gamma \dx(\frac{\sigma}{\rho_0}) +K\gamma \dx(\rho_0^\gamma) \frac{\sigma}{\rho_0}+\dx(a_\ast\rho_0^\gamma)(\frac{\sigma}{\rho_0})^2 + 
2a_\ast\rho_0^\gamma\frac{\sigma}{\rho_0}\dx(\sir) \right\}}_{(b)} \\
&\underbrace{-2\frac{xv(r_0^4-r^4)}{r^3r_0^4} - 4\frac{xrv}{r_0^4}}_{(c)}
= \dt \mathcal{V}
\end{split}
\end{equation}
where $\dt \mathcal{V}$ is given in  \eqref{dtV_comp}.  To derive \eqref{E^3} we will multiply by $\dt^2 v$ and integrate over $x$.  We divide the estimates into the following steps.

\textbf{Step 1} -- $(a_1)$, $(a_2)$, $(b)$, and $(c)$

We begin with an estimate of the product of $\dt^2 v$ with the terms  $(a_1),$ $(a_2),$ $(b),$ and $(c)$.   First, we use \eqref{dxSig} to replace $ \dt\dx(\sir)$ by lower order terms: 
\begin{equation}
\begin{split}
(a_1) +(a_2)  &= 4\pi r^2 
\left\{ K\gamma\rho_0^\gamma \dt\dx\( \sir \)+ K \gamma \dx(\rho_0^\gamma) \dt\(\sir\)  \right \} \\
& + 4\pi r^2 \left\{ \dt \left[ \dx \(a_\ast\rho_0^\gamma\)\(\sir\)^2 + 2(a_\ast\rho_0^\gamma) \sir \dx\(\sir\) \right] \right\} \\
&= -4\pi r^2\(K\gamma +2 a_\ast \sir\) \rho_0^\gamma \frac{\rho}{\rho_0} \dx\( \frac{\rho_0}{\rho} \) \dt\(\sir\) 
\\
&\quad-\( K\gamma + 2 a_\ast \sir\) \rho_0^\gamma \frac{\rho}{\rho_0(\delta+\frac{4\ep}{3})}
\left\{ \dt v + \frac{x(r_0^4-r^4)}{r^2r_0^4} + 4\pi r^2 \left[ K\gamma \rho_0^\gamma \dx\(\frac{\sigma}{\rho_0}\) \right. \right.\\
&\quad  \left. \left.  + K\gamma \dx(\rho_0^\gamma) \frac{\sigma}{\rho_0} + \dx(a_\ast \rho_0^\gamma) \(\frac{\sigma}{\rho_0}\)^2 +2 a_\ast \rho_0^\gamma \frac{\sigma}{\rho_0} 
\dx\(\sir\) \right]  \right\} \\
&\quad+  4\pi r^2 \left \{ K \gamma \dx(\rho_0^\gamma) \dt\(\sir\) + 
\dt \left[ \dx (a_\ast\rho_0^\gamma) \(\sir\)^2 \right] \right.   \\
&\quad+   \left. 2 \rho_0^\gamma \dt \(a_\ast \sir\)\dx\(\sir\) \right\} \\
&:= (A_1)+(A_2)+(A_3).
\end{split}
\end{equation}
Then $\int \dt^2v \cdot [(a_1) +(a_2)] dx$ can be estimated as follows: 
\begin{equation}
\begin{split}
\int \dt^2v \cdot (A_1)dx &\le \frac{\theta_1}{8} \int \abs{\dt^2 v}^2 dx + C\theta_1 \int \frac{\rho_0^{2\gamma+2}}{\rho^2} r^4 \abs{\dx\(\sir\)}^2 dx \le \frac{\theta_1}{8}\mathcal{D}^3+C \theta_1\mathcal{E}^{1,\sigma}\\
\int \dt^2 v\cdot (A_2)dx
&\le \frac{3}{32} \int \abs{\dt^2 v}^2 dx + C[\mathcal{E}^{2,v} + \mathcal{E}^{0,r} + (1+\theta_1) (\mathcal{E}^{1,\sigma} + \mathcal{E}^{0,\sigma}) ]\\
\int \dt^2 v\cdot (A_3)dx &\le \frac{3}{32}\int \abs{\dt^2 v}^2 dx + C\theta_1\mathcal{E}^{1,\sigma} + C(1+\theta_1)\mathcal{E}^{2,\sigma}_0.
\end{split}
\end{equation}
Next for $(b)$ and $(c)$, we may estimate 
\begin{equation}
\begin{split}
\int \dt^2 v\cdot (b)\,dx&\le \frac{\theta_1}{8}\int |\dt^2v|^2 dx + C\theta_1(\mathcal{E}^{1,\sigma}+\mathcal{E}^{0,\sigma}_0)\\
\int \dt^2 v\cdot (c)\,dx&\le \frac{3}{32}\int  |\dt^2v|^2 dx +C(\mathcal{E}^{0,v}+\mathcal{E}^{0,r}).
\end{split}
\end{equation}
Combining the above, we arrive at an estimate for $\int \dt^2 v \cdot [(a_1)+(a_2)+(b)+(c)]dx$.

\textbf{Step 2} -- The viscosity term

Now we consider the viscosity term, $\dt \V$.   We claim that there exist  $\F^3, G$ so that
\begin{equation}\label{claim4.5}
\int \dt^2 v\cdot \dt \mathcal{V} dx = -\frac{d}{dt}\mathcal{E}^3 -\frac{d}{dt}\F^3 + G, 
\end{equation}
where 
\begin{equation}\label{claim4.5_2}
 \begin{split}
 \abs{\F^3} & \le C\theta_1(\mathcal{E}^3+\mathcal{E}^1) \text{ and }\\
 \abs{G} &\le \frac{3}{32} \D^3 + C\theta_1(\mathcal{E}^3+\mathcal{E}^{2,\sigma}_0+\mathcal{E}^1).
 \end{split}
\end{equation}

Recall that $\dt \V$ may be computed as in \eqref{dtV_comp}, and that $\dt \W(M,t)=0$.  Then a simple but lengthy computation, using integration by parts, reveals that
\begin{equation}
\begin{split}
 \int \dt^2 v\cdot \dt \mathcal{V} dx & =-\hal\frac{d}{dt}\int  \delta 16\pi^2 \rho \abs{\dx(r^2 \dt v)}^2 dx + \frac{4\ep}{3} 16\pi^2  \rho r^6 \abs{\dx\(\frac{\dt v}{r}\)}^2  dx \\
& + G_0 + Y, 
\end{split}
\end{equation}
where
\begin{equation}
\begin{split}
 G_0 & = 
\int \(\frac{\dt\sigma}{2\rho} + \frac{2v}{r}\) \left\{ \delta 16\pi^2 \rho \abs{\dx(r^2 \dt v)}^2 + \frac{4\ep}{3} 16\pi^2   \rho r^6 \abs{\dx\(\frac{\dt v}{r}\)}^2 \right\} dx \\
& + \int \rho r^3 \dx\(\frac{v}{r}\) \left\{ \delta 32\pi^2 \frac{\dt v}{r} \dx(r^2 \dt v) -\frac{4\ep}{3} 16\pi^2 \frac{ \dt v}{r} r^3 \dx\( \frac{\dt v}{r} \) \right\} dx,
\end{split}
\end{equation}
and $Y = Y_1 + Y_2$ with 
\begin{equation}
\begin{split}
 Y_1 & = -16 \pi^2 \int \left[ \delta \dt \sigma \dx(r^2 v)  + \delta \rho \dx(2rv^2) \right]\dx(r^2 \dt^2 v)  dx\\
&  -16 \pi^2 \int \left[\frac{4\ep}{3} \dt(\rho r^3) \dx \left(\frac{v}{r}\right) - \frac{4\ep}{3} \rho r^3 \dx \left(\frac{v^2}{r^2} \right) \right]\dx(r^2 \dt^2 v) dx, \\
Y_2 & =   32 \pi^2 \int rv \dt^2 v \dx \left[ \delta \rho \dx(r^2 v) + \frac{4\ep}{3} \rho r^3 \dx \left( \frac{v}{r}\right)  \right]dx.
\end{split}
\end{equation}
Let us define the quantity $Q$ so that $Y_1 =  -16 \pi^2 \int Q \dx(r^2 \dt^2 v) dx$, i.e. $Q$ is the sum of the bracketed terms in the $Y_1$ integrand.  Then we may compute
\begin{equation}
\begin{split}
 Y_1 &= \frac{d}{dt} \int -16 \pi^2 \dx(r^2 \dt v) Q dx + \int 16 \pi^2 \left(\dx(2rv \dt v)Q + \dx(r^2 \dt v) \dt Q \right)  dx \\
 & := -\frac{d}{dt} \F^3_1 + G_1.
\end{split}
\end{equation}
Similarly, we have that
\begin{equation}
\begin{split}
 Y_2 &= \frac{d}{dt}\int -16 \pi^2 \dx(r^2 \dt v) \frac{2v}{r} \left[ \delta \rho \dx(r^2 v) + \frac{4\ep}{3} \rho r^3 \dx\left(\frac{v}{r}\right)\right] dx \\
 & + 16\pi^2 \int \left[\dx(2 r v \dt v) \frac{2 v}{r} \W + \dx(r^2 \dt v) \dt\left(  \frac{2 v}{r}\right) \W + \dx(r^2 \dt v)  \frac{2 v}{r} \dt \W\right] dx  \\
& -16 \pi^2 \int r^2 \dt^2 v \W \dx\left(  \frac{2 v}{r}\right)  dx \\
& =  -\frac{d}{dt} \F^3_2 + G_2.
\end{split}
\end{equation}

Combining the above, we find that \eqref{claim4.5} holds with $\F^3 = \F^3_1 + \F^3_2$ and $G = G_0 + G_1 + G_2$.  To complete the proof of the claim, we note that the estimates \eqref{claim4.5_2} follow from the definition of $\F^3$ and $G$, using  \eqref{dtSig} to replace $\dt^2\sigma$ by other terms.

\textbf{Step 3} -- Conclusion

The only term that remains is 
\begin{equation}
 \int \dt^2 v \dt^2 v dx = \D^3.
\end{equation}
With this, all of the terms in \eqref{l44_1} are accounted for.  We may then combine the analysis of Steps 1 and 2 to deduce the estimate \eqref{E^3}.

\end{proof}

We now bootstrap more estimates. First, we multiply \eqref{rr} by $\frac{1}{\rho}(1-\frac{r_0}{r})$ and integrate to get 
\begin{equation}\label{E^0rb}
\frac{d}{dt}\mathcal{E}^{0,r}_b\le (\eta+C\theta_1)\mathcal{E}^{0,r}_b+ C\mathcal{E}^{1,v},
\end{equation}
where we have used \eqref{ID} to control $\int \abs{v}^2/(r^2 \rho) dx \le C \E^{1,v}$.  Next, by multiplying the equation \eqref{dxSig} by $\frac{1}{\rho_0}\dx(\sir)$ and integrating, we get 
\begin{equation}
\begin{split}\label{E^1b}
\frac{d}{dt}\mathcal{E}^{1,\sigma}_b+\mathcal{D}^{1,\sigma}_b\le(\eta+C\theta_1)\mathcal{E}^{1,\sigma}_b + C(\mathcal{E}^3+\mathcal{E}^{0,\sigma}_{-1}+\mathcal{E}^{0,r}_b).
\end{split}
\end{equation}
Note that here we have again used \eqref{ID} to control $\int \abs{\dt v}^2/(r^2\rho) dx$, which is possible since \eqref{ID} is valid for any choice of $v$, not just solutions.  From the equation \eqref{dxSig} we also see that 
\begin{equation}\label{E^3Sig}
\mathcal{E}^{3,\sigma}_a = \int \(\delta+\frac{4\ep}{3}\)^2 16\pi^2 \frac{r^2}{\rho} \abs{\dx\dt\( \frac{\sigma}{\rho_0} \)}^2 dx \le C ( \mathcal{E}^3 + \mathcal{E}^{1,\sigma}_b + \mathcal{E}^{0,\sigma}_{-1} + \mathcal{E}^0).
\end{equation}
Next, by applying $\dx$ to equation \eqref{contE}, we find that
\begin{equation}\label{dxcontE}
 \frac{\rho_0}{\rho} \dx \dt\(\sir\) +\dx\( \frac{\rho_0}{\rho} \)\dt\(\sir\) +4\pi \dx\( \rho \dx (r^2v)\)=0. 
 \end{equation}
We then use this to get 
\begin{equation}\label{E^3b}
\mathcal{E}^{3,v}_{a_1} = \int  \frac{r^2}{\rho} \abs{\dx( \rho\dx(r^2v))}^2 dx \le C(\mathcal{E}^{3,\sigma}_a + \mathcal{E}^{1,\sigma}_b).
\end{equation}
Since $\int\rho r^6|\dx(\frac{v}{r})|^2 dx\le C\mathcal{E}^{1}$, the equation \eqref{momE} implies that 
\begin{equation}\label{B1}
\mathcal{E}^{3,v}_{a_2} = \int  \rho r^6 \abs{\dx\(\rho r^3 \dx\(\frac{v}{r}\)\)}^2 dx \le C (\mathcal{E}^2 + \mathcal{E}^1 + \mathcal{E}^0).
\end{equation}

We now illustrate how the higher order energy estimates of spatial derivatives of $\dx(\sir)$ and $\dx(\rho \dx(r^2v))$ work. The following lemma concerns the estimate of $\dx(r^4\dx(\sir))$.

\begin{lem} 
\begin{equation}
\begin{split}\label{E^4Sig}
\frac{d}{dt}\mathcal{E}^4+\mathcal{D}^4 &\le({\eta}+C\theta_1)\mathcal{E}^4 + C(\mathcal{E}^3+\mathcal{E}^{0,r}+\mathcal{E}^{0,r}_b+\mathcal{E}^{0,\sigma}_{-1})+C\theta_1(\mathcal{E}^{1,\sigma}_b+\mathcal{E}^{3,\sigma}_a)
\end{split}
\end{equation}
\end{lem}

\begin{proof}
 First, we multiply \eqref{dxSig} by $r^2$ and apply $\dx$ to get
\begin{multline}\label{dxSigx}
\(\delta + \frac{4\ep}{3}\) 4\pi \left\{\underbrace{\frac{\rho_0}{\rho}\dx\( r^4 \dt \dx \( \sir \) \)}_{(i)} + 2\underbrace{\dx\(\frac{\rho_0}{\rho}\) r^4 \dt \dx \(\sir\)}_{(ii)} \right. \\
\left. 
+ \underbrace{ \dx \(r^4\dx \(\frac{\rho_0}{\rho} \) \) \dt\(\sir\)}_{(iii)} \right\} 
+ \underbrace{\dx(r^2\dt v)}_{(iv)}
+ \underbrace{\frac{(r_0^4-r^4)}{r_0^4}}_{(v)} 
- \underbrace{ \frac{xr^4}{\pi r_0^7 \rho} \( 1-\(\frac{r_0}{r}\)^3 + \frac{\sigma}{\rho_0} \)}_{(vi)} \\
+ 4\pi  \underbrace{K\gamma\rho_0^\gamma  \dx\( r^4 \dx\( \frac{\sigma}{\rho_0} \) \) }_{(vii)} 
+  4\pi \left\{ 2K\gamma r^4 \dx( \rho_0^\gamma) \dx\( \frac{\sigma}{\rho_0} \) +K\gamma \dx(r^4\dx(\rho_0^\gamma)) \sir \right\}\\
+4\pi \left\{ 2 r^4 \dx(a_\ast \rho_0^\gamma) \sir \dx\(\frac{\sigma}{\rho_0} \) +  
 \dx(r^4\dx(a_\ast\rho_0^\gamma)) \(\frac{\sigma}{\rho_0}\)^2 \right\} \\
+ 4\pi \left\{  2 a_\ast \rho_0^\gamma \frac{\sigma}{\rho_0} \dx\(r^4 \dx\(\sir\)\) 
+ 2r^4 \dx\(a_\ast\rho_0^\gamma \frac{\sigma}{\rho_0} \)\dx\(\sir\) \right\} =0.
\end{multline}
The energy inequality \eqref{E^4Sig} can be derived as in  Step 2 of Lemma \ref{lem4.2} by multiplying \eqref{dxSigx} by $\rho \dx(r^4\dx(\sir))$ and integrating over $x$.   We provide the details on how $(i)-(vii)$ can be treated; other terms can be estimated similarly. 
\begin{multline}
\int \( \delta + \frac{4\ep}{3} \)4\pi\,(i)\cdot \rho \dx\(r^4\dx\(\sir\)\) dx \\
= \hal \frac{d}{dt} \int \(\delta + \frac{4\ep}{3}\) 4\pi \rho_0 \abs{\dx\(r^4\dx\(\sir\)\)}^2 dx \\ - \(\delta + \frac{4\ep}{3} \) 4 \pi \underbrace{\int \rho_0  \dx\(\dt(r^4) \dx\(\sir\)\)  \dx\(r^4\dx\(\sir\)\)dx}_{(\ast)}.
\end{multline}
Since $\dt(r^4)=4\frac{v}{r}r^4$, 
\begin{equation}
(\ast) = \int 4 \frac{v}{r}\rho_0 \abs{\dx\(r^4 \dx \(\sir\)\)}^2 dx + \int 4  \dx\(\frac{v}{r}\) r^4 \dx\(\sir\) \rho_0  \dx \(r^4\dx\(\sir\)\)dx,
\end{equation}
and since $\abs{\frac{v}{r}}$ and $\abs{\rho r^3\dx(\frac{v}{r})}$ are bounded by $\theta_1$, 
\begin{equation}
|(\ast)|\le C\theta_1 (\mathcal{E}^4+\mathcal{E}^{1,\sigma}_b).
\end{equation}
For $(ii)$ we write  
\begin{equation}
\int (ii)\cdot \rho \dx\(r^4\dx\(\sir\)\) dx = -\int  \frac{\rho_0 r^3\dx\(\frac{\sigma}{\rho_0}\) }{1+\sir} \cdot \frac{r}{\sqrt{\rho}} \dt \dx\(\sir\) \cdot \sqrt{\rho} \dx\(r^4\dx\(\sir\)\) dx
\end{equation}
and therefore,
\begin{equation} 
\abs{\int (ii)\cdot \rho \dx(r^4\dx(\sir)) dx } \le C\theta_1({\mathcal{E}^{3,\sigma}_a}+{\mathcal{E}^4}).
\end{equation}
It is easy to see that 
\begin{equation}
\begin{split}
\abs{ \int [(iii)+(iv)+(v)-(vi)]\cdot \rho \dx\(r^4\dx\(\sir\)\) dx} \\
\le (\frac{\eta}{2}+C\theta_1)\mathcal{E}^4 + C(\mathcal{E}^3+\mathcal{E}^{0,r}+\mathcal{E}^{0,r}_b+\mathcal{E}^{0,\sigma}_{-1}).
\end{split}
\end{equation}
Finally, $(vii)$ forms the dissipation $\mathcal{D}^4$. 
\end{proof}

We also get an estimate for $\dx(r^4\dt\dx(\frac{\sigma}{\rho_0}))$ from \eqref{dxSigx}: 
\begin{equation}
\begin{split}\label{E^5}
\mathcal{E}^4_{a} &= \int \(\delta+\frac{4\ep}{3}\) 4\pi\rho_0  \abs{ \dx \( r^4 \dt\dx\( \frac{\sigma}{\rho_0} \)\)}^2dx\\
& \le 
\theta_1^2\mathcal{E}^{3,\sigma}_{a_1} +\mathcal{E}^4+\mathcal{E}^3+\mathcal{E}^{1}+\mathcal{E}^{0,r}_b+\mathcal{E}^{0,\sigma}_{-1}.
\end{split}
\end{equation}

To derive an estimate of the third spatial derivatives of $v$, we first multiply 
\eqref{dxcontE} by $r^4$ and then apply $\dx$: 
\begin{equation}
\begin{split}
 \frac{\rho_0}{\rho} \dx\( r^4\dx \dt\(\sir\)\) + 2r^4 \dx\(\frac{\rho_0}{\rho}\) \dx\dt\(\sir\) +\dx\( r^4 \dx\(\frac{\rho_0}{\rho}\)\) \dt\(\sir\) \\ 
+4\pi \dx\(r^4\dx\(\rho\dx (r^2v)\)\)=0. 
 \end{split}
 \end{equation}
Thus, we obtain 
\begin{equation}\label{E^6}
\mathcal{E}^4_{a_2} = \int 16\pi^2\rho_0 \abs{\dx(r^4\dx(\rho\dx (r^2v)))}^2 dx\le \mathcal{E}^{4}_{a_1}+\theta_1^2(\mathcal{E}^{3,\sigma}_a+\mathcal{E}^4).
\end{equation}

We now present the higher order energy estimates. We start with $\mathcal{E}^5$ and $\mathcal{E}^5_a$.

\begin{lem}\label{lem4.7}
\begin{equation}\label{E^7}
\begin{split}
\frac{d}{dt}\mathcal{E}^5+\mathcal{D}^5\le \,& (\eta+C\theta_1)\mathcal{E}^5 + (\hal+C\theta_1)\mathcal{D}^5 + C\theta_1^2\mathcal{E}^3
\\&+ C\theta_1(\mathcal{E}^{3,\sigma}_a+\mathcal{E}^3+\mathcal{E}^2+\mathcal{E}^1)+ C\mathcal{E}^2
\end{split}
\end{equation}
\end{lem}

\begin{proof}
We apply $\dt$ to \eqref{dt2v} to see that
\begin{multline}\label{dt3v}
\dt^3 v + 4\pi r^2 \dx \left\{ \(K\gamma +2 a_\ast \sir\) \rho_0^\gamma \dt^2\(\frac{\sigma}{\rho_0}\) + \rho_0^\gamma \left[a_\ast \(\dt\(\frac{\sigma}{\rho_0}\)\)^2 
\right. \right. \\ \left. \left.
+ 4\dt a_\ast \frac{\sigma}{\rho_0}\dt\(\sir\) + \dt^2 a_\ast \(\frac{\sigma}{\rho_0}\)^2 \right] \right\} 
+ 16\pi v r \left \{ K\gamma\rho_0^\gamma \dt\dx\(\sir\) + K\gamma \dx(\rho_0^\gamma) \dt\(\sir\)  \right\}  \\
+ 16\pi v r    \dt \left[ \dx(a_\ast\rho_0^\gamma)\(\sir\)^2 + 2(a_\ast\rho_0^\gamma) \sir \dx\(\sir\) \right]  \\
+8 \pi (r\dt v+ v^2) \left\{ K\gamma\rho_0^\gamma \dx\(\frac{\sigma}{\rho_0}\) + K \gamma \dx(\rho_0^\gamma) \frac{\sigma}{\rho_0} + \dx(a_\ast\rho_0^\gamma) \(\frac{\sigma}{\rho_0}\)^2 + 
2 a_\ast \rho_0^\gamma \frac{\sigma}{\rho_0} \dx\(\sir\) \right\} \\
-\frac{2x(r_0^4-r^4)\dt v}{r^3r_0^4} -\frac{4x r\dt v}{r_0^4}- \frac{12xv^2}{r_0^4}+\frac{6x(r_0^4-r^4)v^2}{r^4r_0^4} 
= \dt^2 \V. 
\end{multline} 
We derive the energy estimate of \eqref{E^7} from \eqref{dt2v} by proceeding as  in the proofs of  Lemmas \ref{lem4.1} and \ref{lem4.3}.  That is, we multiply the resulting equation  by $\dt^2 v$ and integrate over $x$, integrating by parts in some terms to recover  $d \E^5/dt$, $\D^5$, and some error terms that can be estimated by the right side of \eqref{E^7}.  Since the method of proof is already recorded in Lemmas \ref{lem4.1} and \ref{lem4.3}, we omit further details.

\end{proof}

An estimate of $\dx(\rho \dx(\dt[r^2 v])) $ can be obtained through the equation \eqref{dt2v}: 
\begin{equation}\label{E^8}
\mathcal{E}^{5,v}_a\le C(\mathcal{E}^5+ \mathcal{E}^{3,\sigma}_{a}+\mathcal{E}^1) +C\theta_1 ( \mathcal{E}^{3,v}_{a_1}+\mathcal{E}^{3,\sigma}_a+\mathcal{E}^1+\mathcal{E}^0).
\end{equation}
We now bootstrap to control $\dt^2(\sir)$.  We apply $\dt$ to \eqref{dtSig} to get 
\begin{equation}\label{dt3Sig}
\begin{split}
\dt^3\(\sir\)=&-4\pi\frac{\rho}{\rho_0}\rho\dx(r^2\dt^2v)- 24\pi\frac{\rho}{\rho_0}\underbrace{\rho\dx(rv\dt v)}_{(a)}-8\pi
\frac{\rho}{\rho_0}\underbrace{\rho\dx(v^3)}_{(b)} \\
&+ 6\frac{\dt\sigma}{\rho}\dt^2\(\sir\)-6\frac{(\dt\sigma)^3}{\rho_0\rho^2}.
\end{split}
\end{equation}
Note that $(a)=\frac{v}{r}\rho\dx(r^2\dt v) +\rho r^3\dx(\frac{v}{r}) \frac{\dt v}{r}$ and $(b)=3 r^3 \rho (\frac{v}{r})^2\dx(\frac{v}{r}) +\frac{3}{4\pi}(\frac{v}{r})^3 $ and thus by multiplying \eqref{dt3Sig} by $\frac{1}{\rho_0}\dt^2(\sir)$ and integrating, we obtain 
\begin{equation}\label{E^7Sig}
\frac{d}{dt}\mathcal{E}^{5,\sigma}_{-1} \le (\eta+C\theta_1)\mathcal{E}^{5,\sigma}_{-1}+ C\mathcal{D}^5 + C\theta_1^2(\mathcal{E}^3+\mathcal{E}^1+\mathcal{E}^{2,\sigma}_{-1}).
\end{equation}

Next, we take $\dt$ of \eqref{dxSig} to see that 
\begin{multline}\label{dxdtSig}
\(\delta + \frac{4\ep}{3}\) 4\pi r^2 \left\{ \frac{\rho_0}{\rho} \dt^2 \dx \(\sir\) +\dx\( \frac{\rho_0}{\rho} \) \dt^2\(\sir\)  +\dt\( \frac{\rho_0}{\rho} \) \dt\dx \( \sir \) 
\right. \\ \left.
+ \dt\dx \( \frac{\rho_0}{\rho} \) \dt\(\sir\) \right\} 
+\( \delta + \frac{4\ep}{3}\)  \frac{v}{r} 8 \pi r^2 \left\{ \frac{\rho_0}{\rho} \dt\dx\(\sir\) + \dx\(\frac{\rho_0}{\rho}\) \dt\(\sir\) \right\} \\
+ \dt^2 v-2\frac{v}{r} \(\frac{x}{r^2}+\frac{xr^2}{r_0^4}\)
+ 4\pi r^2 \left \{ K\gamma\rho_0^\gamma \dt\dx \(\frac{\sigma}{\rho_0}\) + K\gamma \dx(\rho_0^\gamma) \dt\( \frac{\sigma}{\rho_0} \) \right\} \\  
+ 4\pi r^2 \dt\left[ \dx(a_\ast\rho_0^\gamma) \(\frac{\sigma}{\rho_0}\)^2 
+ 2a_\ast \rho_0^\gamma \frac{\sigma}{\rho_0} \dx\(\sir\) \right] \\
+\frac{v}{r} 8\pi r^2 \left\{ K\gamma \rho_0^\gamma \dx\(\frac{\sigma}{\rho_0}\) + K \gamma \dx\( \rho_0^\gamma \) \frac{\sigma}{\rho_0} + \dx(a_\ast\rho_0^\gamma) \(\frac{\sigma}{\rho_0} \)^2 +
2 a_\ast \rho_0^\gamma \frac{\sigma}{\rho_0} \dx\(\sir\) \right\} =0.
\end{multline}
Therefore, by squaring \eqref{dxdtSig} and integrating, we find that
\begin{equation}\label{4.44}
\begin{split}
\int \(\delta + \frac{4\ep}{3}\)^2 16\pi^2 r^4 \abs{\dt^2 \dx\(\sir\)}^2 dx \le C(\mathcal{E}^5+\mathcal{E}^1) + C\theta_1^2(\mathcal{E}^{5,\sigma}_{-1} + \mathcal{E}^{3,\sigma}_a + \mathcal{E}^1+\mathcal{E}^{0,\sigma}_0 ). 
\end{split}
\end{equation}
Also, by first dividing  \eqref{dxdtSig} by $r$ and then squaring, we obtain 
\begin{equation}\label{E^9b}
\begin{split}
\mathcal{E}^{6,\sigma}_a = \int \(\delta + \frac{4\ep}{3}\)^2 16\pi^2 r^2 \abs{\dt^2\dx\(\sir\)}^2 dx & \le C(\mathcal{E}^6+\mathcal{E}^{3,\sigma}_a+\mathcal{E}^{2,\sigma}_{-1}+\mathcal{E}^1)\\
 &\;+C\theta_1^2(\mathcal{E}^{1,\sigma}_{b}+\mathcal{E}^{3,\sigma}_a+\mathcal{E}^1+\mathcal{E}^{0,\sigma}_{-1}).
\end{split}
\end{equation}

Now we record an estimate of $\mathcal{E}^6$.   

\begin{lem}\label{lem4.8}
There exists an $\F^6$ so that
\begin{equation}\label{E^9}
\begin{split}
\frac{d}{dt}[\mathcal{E}^6+&\widetilde{F}] +\mathcal{D}^6\le (\eta+C\theta_1)\mathcal{E}^6 +(\frac{1}{2}+C\theta_1)\mathcal{D}^6 \\
&+ C(\mathcal{E}^5+\mathcal{E}^{5,\sigma}_{-1}+\mathcal{E}^{1,\sigma}_b+\mathcal{E}^{2,\sigma}_0+\mathcal{E}^{0,\sigma}_0+\mathcal{E}^3+\mathcal{E}^2+\mathcal{E}^1).
\end{split}
\end{equation}
Moreover,  $\abs{\F^6} \le C\theta_1(\mathcal{E}^6+\mathcal{E}^3+\mathcal{E}^1)$. 
\end{lem}
\begin{proof} The energy inequality \eqref{E^9} can be obtained by multiplying \eqref{dt3v} by $\dt^3v$ and integrating over $x$ as done in Lemma \ref{lem4.4}.  We omit further details.
\end{proof}

As seen in the previous estimates in Lemmas \ref{lem4.3}, \ref{lem4.7},  \ref{lem4.4}, and \ref{lem4.8}, the time differentiation of the equation keeps the main structure of the highest order terms as well as  the boundary condition.  Using  the time differentiated equations  \eqref{dt3v} and \eqref{dt3Sig}, we can follow the line of analysis presented in these four lemmas to derive energy inequalities for $\mathcal{E}^7$, $ \mathcal{E}^7_a$,  $\mathcal{E}^{7,\sigma}_{-1}$, $\mathcal{E}^8$ and $\mathcal{E}^{8,\sigma}_a$.  We record these in the following lemma but omit a proof.

\begin{lem}
Let $\E$ be given by \eqref{total_energy}.  We have the following estimates.
\begin{align}\label{E^77}
&\frac{d}{dt}\mathcal{E}^7+\mathcal{D}^7\le \, (\eta+C\theta_1)\mathcal{E}^7 + (\hal+C\theta_1)\mathcal{D}^7 
\\&\quad\quad\quad+ C\theta_1(\mathcal{E}^{6,\sigma}_a+\mathcal{E}^6+\mathcal{E}^5+\mathcal{E}^3+\mathcal{E}^2+\mathcal{E}^1)+ C\mathcal{E}^5\notag\\
\
\label{E^7a}
&\;\mathcal{E}^7_a\le C(\mathcal{E}^7+\mathcal{E}^{6,\sigma}_a+\mathcal{E}^2)+ C\theta_1(\mathcal{E}^{6,\sigma}_a+\mathcal{E}^{3,\sigma}_a+\mathcal{E}^1+\mathcal{E}^{2}+\mathcal{E}^{3,v}_{a_1}+ \mathcal{E}^{5}_a)\\
\label{E^77Sig}
&\frac{d}{dt}\mathcal{E}^{7,\sigma}_{-1} \le (\eta+C\theta_1)\mathcal{E}^{7,\sigma}_{-1}+ C\mathcal{D}^{7} + C\theta_1^2(\mathcal{E}^6+\mathcal{E}^3+\mathcal{E}^{5,\sigma}_{-1})\\
\label{E^88}
&\frac{d}{dt}[\mathcal{E}^8+{\F^8 }] +\mathcal{D}^8\le (\eta+C\theta_1)\mathcal{E}^8 +(\frac{1}{2}+C\theta_1)\mathcal{D}^8 +C(\mathcal{E}-\mathcal{E}^8-\mathcal{E}^{8,\sigma}_a)+C|\mathcal{E}|^2\\
&\quad\quad\text{ where }|{\F^8}|\le C\theta_1\mathcal{E} +C|\mathcal{E}|^2\notag\\
\label{E^8b}
&\;\mathcal{E}^{8,\sigma}_a \le C(\mathcal{E}^8+\mathcal{E}^{6,\sigma}_a+\mathcal{E}^{5,\sigma}_{-1}+\mathcal{E}^3)+C\theta_1^2(\mathcal{E}^{1,\sigma}_{b}+\mathcal{E}^{3,\sigma}_a+\mathcal{E}^1+\mathcal{E}^{0,\sigma}_{-1}) 
\end{align}

\end{lem}

The next lemma ensures that the assumption \eqref{assumption2} is valid within our energy ${\mathcal{E}}$.

\begin{lem} \label{lem4.6}
There exists a constant $\kappa >0$ so that if $\E \le \kappa$, then
\begin{multline}\label{l46_0}
\pnorm{\sir}{\infty} + \pnorm{ \dt \left(\sir\right)}{\infty} +\pnorm{ \dt^2 \left(\sir\right)}{\infty} + \pnorm{ \dt^3 \left(\sir\right)}{\infty} \\
+ \pnorm{1- \frac{r_0}{r}}{\infty} + \pnorm{\rho r^3\dx\left(\sir\right)}{\infty}+ \pnorm{\frac{v}{r}}{\infty} + \pnorm{\frac{\dt v}{r}}{\infty} +\pnorm{\frac{\dt^2 v}{r}}{\infty} \le C \sqrt{{\mathcal{E}}}
\end{multline}
for some constant $C>0$.  Here $\E$ is given by \eqref{total_energy}.
\end{lem}

\begin{proof} 

The proof proceeds through several steps.

\textbf{Step 1} -- $\dt^k (\sigma/\rho_0)$ estimates

We begin by estimating $\sigma/\rho_0$ in $W^{1,1}((0,M))$.  First, we use H\"{o}lder's inequality to estimate
\begin{equation}
 \int \abs{\frac{\sigma}{\rho_0}} dx \le \sqrt{M} \left( \int \abs{\frac{\sigma}{\rho_0}} dx \right)^{1/2} \le C \sqrt{\E^{0,\sigma}_{-1}} \le C \sqrt{\E}.
\end{equation}
On the other hand, we may estimate
\begin{equation}
 \int \abs{\dx\left( \frac{\sigma}{\rho_0} \right)} dx \le 
\left( \int \frac{r^2}{\rho} \abs{\dx\left( \frac{\sigma}{\rho_0} \right)}^2 dx  \right)^{1/2}
\left( \int \frac{\rho}{r^2} dx \right)^{1/2} \le C \left( \E^{1,\sigma}_b \right)^{1/2} \le C \sqrt{\E}.
\end{equation}
Here we have used the fact that $r^2(x) \sim x^{2/3}$ for $x \sim 0$, which follows from the definition of $r(x)$ and L'Hospital's rule, to see that $\int (\rho/r^2)dx < \infty$.  Combining these estimates with the usual $1-D$ Sobolev embedding $W^{1,1}((0,M)) \hookrightarrow C^0((0,M))$, we find that $\sigma/\rho_0 \in C^0$ and
\begin{equation}\label{l46_1}
 \pnorm{\frac{\sigma}{\rho_0}}{\infty} \le C \sqrt{\E}.
\end{equation}

Now to control $\dt \sigma/\rho_0$ we argue similarly to estimate
\begin{equation}
 \int   \abs{\frac{\dt \sigma}{\rho_0}} + \abs{\dx \left(\frac{ \dt \sigma}{\rho_0}\right)} dx \le  C \sqrt{\E^{2,\sigma}_{-1}} + \sqrt{\E^{3,\sigma}_a} \le C \sqrt{\E}.
\end{equation}
Then $\dt \sigma/\rho_0 \in C^0$ and 
\begin{equation}\label{l46_2}
 \pnorm{\frac{\dt \sigma}{\rho_0}}{\infty} \le C \sqrt{\E}.
\end{equation}
A similar argument, employing $\E^{1+2i,\sigma}_{-1}$ and $\E^{2+2i,\sigma}_a$ for $i=1,2$, then implies that 
\begin{equation}\label{l46_3}
\dt^2 \sigma/\rho_0, \dt^3 \sigma/\rho_0 \in C^0 \text{ and }  \pnorm{\frac{\dt^2 \sigma}{\rho_0}}{\infty} + \pnorm{\frac{\dt^3 \sigma}{\rho_0}}{\infty}\le C \sqrt{\E}.
\end{equation}
We thus deduce from \eqref{l46_1} and \eqref{l46_2}--\eqref{l46_3} that 
\begin{equation}\label{l46_4}
 \pnorm{\sir}{\infty} + \pnorm{ \dt \left(\sir\right)}{\infty} +\pnorm{ \dt^2 \left(\sir\right)}{\infty} + \pnorm{ \dt^3 \left(\sir\right)}{\infty} \le C \sqrt{\E}
\end{equation}

\textbf{Step 2}  -- $1-r_0/r$ estimate

Let us now suppose that $\E \le \kappa$ with $\kappa$ small enough so that $C \sqrt{\E}\le \hal$, where $C>0$ is  the constant appearing on the right side of \eqref{l46_4}.  In particular, this implies that $\pnorm{\sigma/\rho_0}{\infty} \le 1/2 < 1$.  With this estimate  in hand, we can derive an estimate for $r_0/r$.  Indeed, the Taylor expansion \eqref{r_0r} easily implies the estimate
\begin{equation}\label{l46_5}
 \pnorm{1- \frac{r_0}{r}}{\infty} \le C \pnorm{\sir}{\infty}^{1+k} \le  C \sqrt{\E}
\end{equation}
for some $k \ge 0$. This is the $1-r_0/r$ estimate in \eqref{l46_0}.

\textbf{Step 3} -- $\dt^k v/r$ estimates

We now turn to estimates for $\dt^k v/r$, $k=0,1,2$.  From Step 1, we know that $\sir$ and $\dt \sir$ are continuous and bounded, while from Step 2 we know that $\pnorm{\sir}{\infty} \le 1/2$ so that $1 + \sir$ is also continuous and bounded.  From the boundary conditions at $x=0$, we also have that $r^2 v (0,t) =0$. Hence we may spatially integrate the continuity equation \eqref{contE} to see that
\begin{equation}
 (r^2 v)(x,t) = \frac{-1}{4\pi} \int_0^x \frac{1}{\rho_0(y) \left(1+ \frac{\sigma(y,t)}{\rho_0(y)} \right)^2} \frac{\dt \sigma(y,t)}{\rho_0(y)} dy.
\end{equation}

Notice now that, due to the asymptotics \eqref{LEL}, we have that
\begin{equation}
 \int_0^M \frac{dy}{\rho_0(y)} < \infty.
\end{equation}
This and the estimates \eqref{l46_4} then imply that $v/r \in C^0$ and
\begin{equation}\label{l46_6}
 \pnorm{r^2 v}{\infty} \le C \pnorm{\dt \sir}{\infty}  \le C \sqrt{\E}. 
\end{equation}
On the other hand, due to L'Hospital, we have that
\begin{equation}
 \frac{1}{r^3(x,t)} \int_0^x \frac{dy}{\rho_0(y)} \sim \frac{4\pi \rho(x,t)}{3\rho_0(x)} = \frac{4\pi}{3} \left( 1+ \frac{\sigma(x,t)}{\rho_0(y)} \right)< \infty \text{ for }x \sim 0,
\end{equation}
which means that
\begin{equation}
 \sup_{x\in(0,M)} \frac{1}{r^3(x,t)} \int_0^x \frac{dy}{\rho_0(y)} <\infty.
\end{equation}
We may then deduce that $v/r \in C^0$ and
\begin{equation}\label{l46_7}
 \pnorm{\frac{v}{r}}{\infty} \le C \pnorm{\dt \sir}{\infty} \sup_{x\in(0,M)} \frac{1}{r^3(x,t)} \int_0^x \frac{dy}{\rho_0(y)}  \le C \sqrt{\E}.
\end{equation}

Now we apply $\dt$ to \eqref{contE} and argue as above to see that
\begin{multline}
 (r^2 \dt v)(x,t) =  -\int_0^x \frac{1}{4\pi \rho_0(y) \left(1+ \frac{\sigma(y,t)}{\rho_0(y)} \right)^2} \frac{\dt^2 \sigma(y,t)}{\rho_0(y)}dy  \\
+ \int_0^x \frac{1}{2\pi \rho_0(y) \left(1+ \frac{\sigma(y,t)}{\rho_0(y)} \right)^3} \abs{\frac{\dt \sigma(y,t)}{\rho_0(y)}}^2 dy  
- \int_0^x 2 (r^2 v)(y,t) \frac{v(y,t)}{r(y,t)} dy.
\end{multline}
Using this, we may argue as above (using the estimates \eqref{l46_6} and \eqref{l46_7}) to deduce that $r^2 \dt v, \dt v/r \in C^0$ and 
\begin{equation}\label{l46_8}
 \pnorm{r^2 \dt v}{\infty} + \pnorm{\frac{\dt v}{r}}{\infty} \le C \sqrt{\E}.
\end{equation}
An iterative argument, using $\dt^2$ applied to \eqref{contE} in conjunction with the estimates \eqref{l46_8}, then allows us to see that $r^2 \dt^2 v, \dt^2 v / r \in C^0$ with
\begin{equation}\label{l46_9}
 \pnorm{r^2 \dt^2 v}{\infty} + \pnorm{\frac{\dt^2 v}{r}}{\infty} \le C \sqrt{\E}.
\end{equation}
Then \eqref{l46_6}, \eqref{l46_7}, and \eqref{l46_8}--\eqref{l46_9} may be combined to derive the $\dt^k v/r$ estimates recorded in \eqref{l46_0}.

\textbf{Step 4} -- $\rho r^3 \dx(\sigma/\rho_0)$ estimate

Since $\pnorm{\sigma/\rho_0}{\infty} \le 1/2$, to prove the $\rho r^3 \dx(\sigma/\rho_0)$ estimate listed in \eqref{l46_0}, it suffices to estimate this term with $\rho$ replaced by $\rho_0$.  We claim that
\begin{equation}\label{l46_10}
\pnorm{\rho_0 r^3 \dx(\sir)}{\infty} \le C \left( \sqrt{\mathcal{E}^{1,\sigma}_b} +  \sqrt{\mathcal{E}^4} \right) \le C \sqrt{\E}.
\end{equation}
To prove \eqref{l46_10}, we will use the $1-D$ Sobolev embedding $W^{1,1}\hookrightarrow C^0$.  First note that
\begin{equation}
 \int \rho_0 r^3 \abs{\dx \left(\sir\right)} dx  \le \left(\int \rho \rho_0^2 r^2 dx \right)^{1/2}  \left(\int \frac{r^2}{\rho}\abs{\dx \left(\sir\right)}^2 dx\right)^{1/2} 
\le C \sqrt{\E^{1,\sigma}_b}.
\end{equation}

On the other hand, we may compute
\begin{equation}\label{l46_11}
\begin{split}
\dx\left( \rho_0 r^3 \dx\left( \sir \right) \right) &= \frac{\rho_0}{r} \dx \left( r^4\dx\left(\sir\right)\right) + \dx \rho_0 r^3\dx\left(\sir\right )- \frac{\rho_0}{4\pi\rho} \dx \left( \sir \right)\\
\text{ and }\; \dx \rho_0 &=-\frac{x}{4\pi K\gamma \rho_0^{\gamma-1}r_0^4}.
\end{split}
\end{equation}
Then since  $\frac{\rho_0}{\rho_0^{2\gamma-2}}\leq \frac{C}{\rho_0}$ as long as $\gamma<2$, we may estimate 
\begin{multline}
 \int \rho_0 r_0^2 \abs{ \dx\left( \rho_0 r^3 \dx \left( \sir \right)\right)  }^2 dx \\
\le C \int \rho_0 r_0^2 \left[ \frac{\rho_0^2}{r^2} \abs{\dx \left( r^4 \dx \left( \sir \right)  \right)}^2  + \frac{\rho_0^2}{\rho^2} \abs{\dx \left( \sir\right)}^2 +    \frac{x^2 r^6}{\rho_0^{2\gamma-2} r_0^8}  \abs{\dx \left( \sir \right)}^2    \right] dx \\
\le C \E^4 + \E^{1,\sigma}_b.
\end{multline}
Then from this and H\"{o}lder's inequality we get
\begin{multline}\label{l46_12}
 \int \abs{ \dx\left( \rho_0 r^3 \dx \left( \sir \right)\right)  } dx \le \left(\int \frac{dx}{\rho_0 r_0^2}  \right)^{1/2} \left( \int \rho_0 r_0^2 \abs{ \dx\left( \rho_0 r^3 \dx \left( \sir \right)\right)  }^2 dx \right)^{1/2} \\
\le C \sqrt{\E^4} + \sqrt{\E^{1,\sigma}_b}.
\end{multline}
Together, the estimates \eqref{l46_11} and \eqref{l46_12} constitute a $W^{1,1}$ estimate for 
$\rho_0 r^3 \dx ( \sigma/\rho_0 )$, so we then  obtain \eqref{l46_10} via the Sobolev embedding. 
\end{proof}

%%%%%%%%%%%%%%%%%%%%%%%%%%%%%%%%%%%%%%%%%%%%%%%%%%%
\section{Nonlinear instability}\label{5}
%%%%%%%%%%%%%%%%%%%%%%%%%%%%%%%%%%%%%%%%%%%%%%%%%%%

%%%%%%%%%%%%%%%%%%%%%%%%%%%%%%%%%%%%%%%%%%%%%%%%%%%
\subsection{The bootstrap argument}
%%%%%%%%%%%%%%%%%%%%%%%%%%%%%%%%%%%%%%%%%%%%%%%%%%%

Based on the nonlinear estimates in the previous section, we now establish a bootstrap argument that allows us to control the growth of $\E$ in terms of the linear growth rate $\lambda$, constructed in Theorem \ref{growing_mode}.  The idea is to assume small data and that the lowest order energy, $\E^0$, grows no faster than the linear growth rate; then the inequalities in the last section allow for a bootstrap argument that shows that all of $\E$ grows no faster than the linear growth rate.

\begin{prop}\label{prop} 
Let $\sigma$ and $v$ be a solution to the Navier-Stokes-Poisson system \eqref{PNSP}. Assume that 
\begin{equation}\label{pr_01}
\sqrt{{\mathcal{E}}}(0)\le C_0\iota\; \text{ and }\; \sqrt{\mathcal{E}^0}(t)\le C_0\iota e^{\lambda t}\; \text{ for } \;0\le t\le T,
\end{equation}
where $\E^0$ and $\E$ are as defined in \eqref{energies_1} and \eqref{total_energy}.  Then there exist $C_\star$ and  $\theta_\star>0$ such that  if  $0\le t\le \min \{T,T(\iota,\theta_\star)\}$, then
\begin{equation} 
\sqrt{{\mathcal{E}}}(t) \le C_\star\iota e^{\lambda t}\le C_\star \theta_\star,
\end{equation}
where we have written $T(\iota,\theta_\star) = \frac{1}{\lambda}\ln\frac{\theta_\star}{\iota}$.
\end{prop}

\begin{proof} 
To prove the result we will employ a bootstrap argument using all of the nonlinear energy estimates derived in the previous section.  We now choose  $\theta_1$ and $\eta$ sufficiently small in all of these estimates so that $C\theta_1 + \eta\le \frac{\lambda}{2}$ and $C \theta_1 \le 1/8$ in all of the energy inequalities.  Throughout this proof we will write $\tilde{C}$ for a generic constant; we write this in place of $C$ to distinguish the constants from those appearing in the nonlinear energy estimates.

To begin the bootstrapping, we show that the estimate \eqref{pr_01} allows us to control an integral of the $\D^0$ dissipation.  Indeed, we use \eqref{E^0} and \eqref{pr_01} along with Gronwall's inequality to see that for $0 \le t \le T$, 
\begin{equation}\label{pr_1}
 \frac{d}{dt} \E^0 + \hal \D^0 \le C \E^0 \le C C_0^2 (\iota e^{\lambda t})^2 + \frac{\lambda}{2} \E^0 \Rightarrow \hal \int_0^t e^{\lambda/2(t-s)} \D^0(s) ds \le  \tilde{C} \iota^2 e^{2\lambda t}.
\end{equation}
Then we employ \eqref{E^0b} with $\beta=-1$ in conjunction with \eqref{pr_1} to see that
\begin{multline}
 \frac{d}{dt}\( e^{-t\lambda/2} \E^{0,\sigma}_{-1}(t)  \) \le C e^{-t \lambda/2} \D^0(t) \\
\Rightarrow \E^{0,\sigma}_{-1}(t) \le \E^{0,\sigma}_{-1}(0) e^{t \lambda/2} + C \int_0^t e^{\lambda/2(t-s)} \D^0(s) ds \\
\Rightarrow \E^{0,\sigma}_{-1}(t) \le \tilde{C} \iota^2 e^{2\lambda t} \Rightarrow \E^{0,\sigma}_0(t) \le \tilde{C} \E^{0,\sigma}_{-1} (t) \le \tilde{C} \iota^2 e^{2\lambda t}.
\end{multline}

Next, let $q>0$ be the constant from estimate \eqref{E^1} and choose $k = (4q)/\lambda$.  Then \eqref{E^1} and \eqref{E^2}, together with \eqref{E^0} and the above estimates, imply that $0\le t\le T$
\begin{multline}
\frac{d}{dt}[k\mathcal{E}^2 + k\mathcal{E}^0+\mathcal{E}^1]  + \frac{k}{2} \D^2 \le (\eta+C\theta_1)(\mathcal{E}^1+k\mathcal{E}^2) + q\mathcal{E}^2+ C(\mathcal{E}^0+\mathcal{E}^{0,\sigma}_0) \\
\le \frac{\lambda}{2} (\mathcal{E}^1+k\mathcal{E}^2 + k\E^0) + \tilde{C} \iota^2 e^{2 \lambda t}.
\end{multline}
Using Gronwall's inequality again, we obtain from this that for $0\le t\le T$
\begin{equation}
 \mathcal{E}^2(t) + \mathcal{E}^1(t) + \hal \int_0^t e^{\lambda/2(t-s)} \D^2(s)ds \le \tilde{C} \iota^2 e^{2\lambda t}.
\end{equation}
Here we have used the fact that $k$ is bounded and non-zero to absorb it into the constant $\tilde{C}$.  We then employ \eqref{E^2b} with $\beta=-1$  to see that
\begin{multline}
 \frac{d}{dt}\( e^{-t\lambda/2} \E^{2,\sigma}_{-1}(t)  \) \le C e^{-t \lambda/2} (\D^2(t) +\D^0(t)) \\
\Rightarrow \E^{2,\sigma}_{-1}(t) \le \E^{2,\sigma}_{-1}(0) e^{t \lambda/2} + C \int_0^t e^{\lambda/2(t-s)} (\D^2(s) + \D^0(s)) ds \\
\Rightarrow \E^{2,\sigma}_{-1}(t) \le \tilde{C} \iota^2 e^{2\lambda t} \Rightarrow \E^{2,\sigma}_0(t) \le \tilde{C} \E^{2,\sigma}_{-1} (t) \le \tilde{C} \iota^2 e^{2\lambda t}.
\end{multline}

Bootstrapping further, \eqref{E^3} gives rise to 
\begin{equation}
\mathcal{E}^3(t) \le  \tilde{C} \iota^2 e^{2\lambda t}.
\end{equation}
Similarly, from \eqref{E^0rb}, \eqref{E^1b}, \eqref{E^3Sig},  \eqref{E^3b} and \eqref{B1}, we also obtain for $0\le t\le T$
\begin{equation}
\mathcal{E}^{0,r}_b(t) + \mathcal{E}^{1,\sigma}_b (t) + \mathcal{E}^{3,\sigma}_a(t) + \mathcal{E}^{3,v}_{a_1}(t) + \mathcal{E}^{3,v}_{a_2}(t) \le \tilde{C} \iota^2 e^{2\lambda t}.
\end{equation}
Next, from \eqref{E^4Sig} we get 
\begin{equation}
\frac{d}{dt}\mathcal{E}^4\le (\eta+C\theta_1)\mathcal{E}^4 +C\iota^2e^{2\lambda t} \Rightarrow 
\mathcal{E}^4(t) \le  \tilde{C} \iota^2 e^{2\lambda t}.
\end{equation}
In turn, from \eqref{E^5} and \eqref{E^6}, we find that 
\begin{equation}
\mathcal{E}^4_{a_1}(t) + \mathcal{E}^4_{a_2}(t) \le  \tilde{C} \iota^2 e^{2\lambda t}.
\end{equation}
Similarly, the energy inequalities \eqref{E^7}, \eqref{E^8}, \eqref{E^7Sig} yield  
\begin{equation}
\mathcal{E}^5(t) +\mathcal{E}^5_a(t) + \mathcal{E}^{5,\sigma}_{-1}(t) \le  \tilde{C} \iota^2 e^{2\lambda t},
\end{equation}
and \eqref{E^9} and \eqref{E^9b} yield 
\begin{equation}
\mathcal{E}^6(t) + \mathcal{E}^{6,\sigma}_a(t) \le  \tilde{C} \iota^2 e^{2\lambda t}.
\end{equation}
Successively, \eqref{E^77}, \eqref{E^7a}, and \eqref{E^77Sig} imply 
\begin{equation}
\mathcal{E}^7(t) + \mathcal{E}^7_a(t) +\mathcal{E}^{7,\sigma}_{-1}(t) \le  \tilde{C} \iota^2 e^{2\lambda t}.
\end{equation}
To get the bound of $\mathcal{E}^8$, we first note that  $\mathcal{E}^8$ satisfies the following inequality from  \eqref{E^88} and \eqref{E^8b}
\begin{equation}\label{E^92}
\frac{d}{dt}[\mathcal{E}^8+\widetilde{F}_\ast] \le (\eta+C\theta_1)(\mathcal{E}^8+\widetilde{F}_\ast  )+C_1|{\mathcal{E}}|^2+\widetilde{C}\iota^2e^{2\lambda t}
\end{equation}
for some constants $C_1>0$ and $\widetilde{C}>0$. 
We now define $T^\ast$ by 
\begin{equation}
T^\ast:=\sup \left \{ t \;\left\vert\; \mathcal{E}(s) \le \min \left\{\theta_1, \frac{\lambda}{4C_1} \right\} \text{ for } s\in [0,t] \right\}\right. .
\end{equation}
Let $0\le t\le \min\{T,T^\ast\}$. Then by the Gronwall inequality, \eqref{E^92} implies that 
\begin{equation}
\mathcal{E}^8(t) + \widetilde{F}_\ast(t) \le  \tilde{C} \iota^2 e^{2\lambda t}  \Rightarrow 
\mathcal{E}^8(t) + \mathcal{E}^{8,\sigma}_a(t) \le  \tilde{C} \iota^2 e^{2\lambda t} \text{ for } 0\le t\le \min\{T,T^\ast\}.
\end{equation}
Thus, combining all of the above analysis, we finally obtain 
\begin{equation}\label{5.1}
\mathcal{E}(t)  \le C_2 \iota^2 e^{2\lambda t} \text{ for } 0\le t\le \min\{T,T^\ast\}
\end{equation}
for a constant $C_2>0$ independent of $\iota$. 

We now choose $\theta_\star$ such that $C_2 (\theta_\star)^2< \min\{\theta_1, \frac{\lambda}{4C_1}  \}$.   We consider the following two cases. 

(i) $T(\iota,\theta_\star) \le  \min\{T,T^\ast\}$.  In this case, the conclusion follows without any additional work. 

(ii) $T(\iota,\theta_\star) >  \min\{T,T^\ast\}$.  We claim that it must hold that $T\le T^\ast< T(\iota,\theta_*)$, in which case  the conclusion directly follows.   To prove the claim, we note that otherwise we would have $T^\ast < T < T(\iota,\theta_\star)$. Letting $t=T^\ast$, from \eqref{5.1}, we get 
\begin{equation}
{\mathcal{E}}(T^\ast) \le C_2 \iota^2 e^{2\lambda T^\ast} <  C_2 \iota^2 e^{2\lambda T^\iota} = C_2(\theta_\star)^2 \text{ by the definition of }T(\iota,\theta_\star),
\end{equation}
but this is impossible due to our choice of $\theta_\star$ since it would then contradict the definition of $T^\ast$.  Since we then find our desired estimate in both cases, this concludes the proof of the proposition. 
\end{proof}

%%%%%%%%%%%%%%%%%%%%%%%%%%%%%%%%%%%%%%%%%%%%%%%%%%%
\subsection{Further nonlinear estimates}
%%%%%%%%%%%%%%%%%%%%%%%%%%%%%%%%%%%%%%%%%%%%%%%%%%%

As a preparation of the proof of our main theorem, we recall the Navier-Stokes-Poisson system \eqref{PNSP} can be written in perturbed form  as in \eqref{mild_form_in_1} and \eqref{mild_form_in_2} in terms of $\sigma$ and $w:=r^2v$:
\begin{equation}\label{N_1}
\begin{split}
&\dt
\begin{pmatrix}
 \sigma \\ w
\end{pmatrix}
 + \mathcal{L}
\begin{pmatrix}
 \sigma \\ w
\end{pmatrix}
= 
\begin{pmatrix}
 N_1 \\ N_2.
\end{pmatrix}
\end{split}
\end{equation}
with the boundary conditions  
\begin{equation}\label{N_2} 
\(\frac{w}{r^2}\)(0,t)=0\text{ and }\sigma(M,t)=0\,,\;\mathcal{B}(w)=N_\mathcal{B} \text{ at }x=M,
\end{equation}
where the boundary operator $\mathcal{B}(w)$ is defined by \eqref{bndry_op},  $N_\mathcal{B}$ is given as 
\begin{equation}\label{NB_def}
N_\mathcal{B}=\left\{  \(\delta+\frac{4\ep}{3}\) 4\pi\sigma\dx w-4\ep \left.\left[ \(\frac{r_0}{r}\)^3-1 \right] \frac{w}{r_0^3} \right\} \right\vert_{x=M},
\end{equation}
and $N_1$ and $N_2$ read as follows:
\begin{equation}\label{N12_def}
\begin{split}
N_1&=-4\pi(2\rho_0+\sigma)\sigma\dx w \\
N_2&=\frac{2w^2}{r^3}-4\pi(r^4-r_0^4)\dx\(K\gamma \rho_0^\gamma \sir\) -4\pi r^4\dx\(\rho_0^\gamma a_\ast\(\sir\)^2\)-M_1-M_2,
\end{split}
\end{equation}
where 
\begin{equation}
\begin{split}
M_1&= \frac{x}{r_0^4} \(r_0^4-r^4-\frac{r_0}{\pi } \int_0^x \frac{\sigma}{\rho_0^2} dy\)\\
&=x
\left\{ \frac{c_1}{ r_0^3} \int_0^x \frac{1}{\rho_\ast}\(\frac{\sigma}{\rho_0}\)^2dy + \frac{c_2}{r_0^6}\(\int_0^x \frac{\sigma}{\rho_0^2}dy\)^2\right\} \text{ by Taylor expansion}\\
&\quad\text{where } \rho_\ast/\rho_0\sim 1\text{ is a bounded smooth function of } \sir 
\end{split}
\end{equation} 
and 
\begin{equation}
\begin{split}
M_2&= 16\pi^2\(\delta+\frac{4\ep}{3}\) \{ r_0^4\dx(\rho_0\dx w)-r^4\dx(\rho\dx w)\}\\
&= 16\pi^2\(\delta+\frac{4\ep}{3}\) \{(r_0^4-r^4)\dx(\rho_0\dx w)-r^4\dx(\sigma\dx w)
\}.
\end{split}
\end{equation}

It is possible to estimate these nonlinearities in terms of the energy $\E$ given by \eqref{total_energy}.  We present these estimates now.

\begin{lem}\label{lem5.3} 
For each $t$,
\begin{equation}
\mathfrak{E}(N_1,N_2) \le C|\mathcal{E}|^2\; \text{ and }\; |N_\mathcal{B}|+ |\dt N_\mathcal{B}|+|\dt^2N_\mathcal{B}|\le C|\mathcal{E}|^2
\end{equation}
where $\mathfrak{E}$ is defined in \eqref{frak_E_def}.
\end{lem}

\begin{proof} The second inequality  follows directly from Lemma \ref{lem4.6}. For the first inequality, we only provide the details for the highest-order nonlinear term $M_2$ in $N_2$. Lower order terms may be estimated similarly.  Throughout the  proof we will write $\theta_1$ to denote the left side of estimate \eqref{l46_0}; Lemma \ref{lem4.6} them implies that $\theta_1 \le C \sqrt{\E}$.

By rewriting  $M_2$ as 
\begin{equation}
\begin{split}
\frac{ M_2}{16\pi^2(\delta+\frac{4\ep}{3})}&= 
(r_0^4-r^4)\dx(\rho_0 \dx w)-r^4\dx(\sigma\dx w)\\
&=\left[\(\frac{r_0}{r}\)^4-1\right] \left\{ \frac{\rho_0}{\rho}r^4\dx(\rho \dx w)+r^4\dx\( \frac{\rho_0}{\rho} \)\rho\dx w \right\} \\
&\quad- \frac{\sigma}{\rho} r^4\dx(\rho \dx w)- r^4 \dx\(\frac\sigma\rho\) \rho\dx w
\end{split}
\end{equation}
it is easy to see that 
\begin{equation}
\int  \frac{|N_2|^2}{r_0^4}dx\le C\theta_1^2\mathcal{E}\le C|\mathcal{E}|^2.
\end{equation}
Next, 
\begin{multline}
\frac{\dx M_2}{16\pi^2(\delta+\frac{4\ep}{3})} = \(\frac{r_0}{r}\)^3 \frac{1-(\frac{r_0}{r})^3 + \frac{\sigma}{\rho_0}}{\pi\rho r_0^2r} \left\{ \frac{\rho_0}{\rho}r^4\dx(\rho \dx w)+r^4\dx\(\frac{\rho_0}{\rho}\) \rho\dx w \right\} \\
 + \left[ \(\frac{r_0}{r}\)^4 - 1\right] \left\{ \frac{\rho_0}{\rho}\dx(r^4\dx(\rho \dx w)) + 2r^4\dx\(\frac{\rho_0}{\rho}\) \dx(\rho \dx w) + \dx\(r^4\dx\(\frac{\rho_0}{\rho}\)\)\rho\dx w \right\} \\
 - \frac{\sigma}{\rho} \dx(r^4\dx(\rho \dx w))- 2r^4 \dx\(\frac\sigma\rho\) \dx(\rho\dx w)-\dx\( r^4 \dx\( \frac\sigma\rho \)\) \rho\dx w.
\end{multline}
Hence from the definition of the energies and from the estimates in the previous section, 
\begin{equation}
\begin{split}
\int \rho_0|\dx M_2|^2dx  \le C\theta_1^2 (\mathcal{E}^4_{a_2} +\mathcal{E}^{3,v}_{a_1}+\mathcal{E}^{1,\sigma}_b)+C\theta_1^4(\mathcal{E}^4 +\mathcal{E}^{1,\sigma}_b)+ C(\mathcal{E}^{1,\sigma}_b+\mathcal{E}^4)\mathcal{E}^{3,v}_{a_1}\le C|{\mathcal{E}}|^2.
\end{split}
\end{equation}
On the other hand, $\dt M_2$ reads as 
\begin{multline}
\frac{\dt M_2}{16\pi^2(\delta+\frac{4\ep}{3})} = -4r^3v\dx(\rho\dx w) -r^4\dx(\dt\sigma \dx w)-r^4\dx( \sigma \dx\dt w)+(r_0^4-r^4)\dx(\rho_0\dx\dt w)\\
= -4 \frac{v}{r}r^4\dx (\rho\dx w) -\frac{\dt\sigma}{\rho} r^4\dx(\rho \dx w) -r^4 \dx\(\frac{\dt\sigma}{\rho}\) \rho \dx w-r^4 \frac{\sigma}{\rho} \dx(\rho\dx\dt w) \\
-r^4\dx\(\frac{\sigma}{\rho}\) \rho \dx\dt w +(r_0^4-r^4)\frac{\rho_0}{\rho}\dx(\rho\dx\dt w) + (r_0^4-r^4)\dx\(\frac{\rho_0}{\rho}\) \rho\dx\dt w.
\end{multline}
Thus 
\begin{equation}
\int \frac{|\dt M_2|^2}{r_0^4}dx \le C\theta_1^2(\mathcal{E}^{3,v}_{a_1}+\mathcal{E}^{3,\sigma}_a+\mathcal{E}^5_a) +C(\mathcal{E}^4+\mathcal{E}^{1,\sigma}_b)\mathcal{E}^{6,\sigma}_a\le C|{\mathcal{E}}|^2.
\end{equation}
\end{proof}

%%%%%%%%%%%%%%%%%%%%%%%%%%%%%%%%%%%%%%%%%%%%%%%%%%%
\subsection{Data analysis}
%%%%%%%%%%%%%%%%%%%%%%%%%%%%%%%%%%%%%%%%%%%%%%%%%%%

In order to prove our nonlinear instability result, we want to use the linear growing mode solutions constructed in Theorem \ref{growing_mode} to construct small initial data for the nonlinear problem, written in the perturbation formulation \eqref{N_1}.  Small data in the perturbation formulation correspond to initial data for  \eqref{lagrangian_1}--\eqref{bc} that are close to the stationary solutions $\rho = \rho_0$, $v =0$, $r=r_0$.  Unfortunately, due to the regularity framework (given by $\E$ as in \eqref{total_energy}) in which we have proved our nonlinear estimates, we cannot simply set the initial data for the nonlinear problem \eqref{N_1} to be a small constant times the linear growing modes.  The reason for this is that the initial data for the nonlinear problem must satisfy certain nonlinear compatibility conditions in order for us to guarantee local existence in the energy space defined by $\E$.  Until now we have taken the local well-posedness theory for the nonlinear problem for granted, but we must now say a few words about the compatibility conditions in order to construct our desired initial data.

Recall that we can rewrite the nonlinear problem \eqref{lagrangian_1}--\eqref{bc} in the form \eqref{N_1}--\eqref{N_2} with nonlinearities given by \eqref{NB_def}--\eqref{N12_def}.  Let us concisely rewrite \eqref{N_1} as
\begin{equation}\label{da_1}
 \dt \SX + \mathcal{L} \SX = \SN(\SX) \text{ for } 
\SX =
\begin{pmatrix}
 \sigma \\ \bar{w}
\end{pmatrix},
\end{equation}
where $\SN(\SX)$ is the nonlinearity given in terms of $N_1$ and $N_2$ by the right side of \eqref{N_1}.  We will also rewrite the  boundary conditions \eqref{N_2} as
\begin{equation}\label{da_2}
  \SC(\SX) := 
\begin{pmatrix}
 (w/r_0^2)\vert_{x=0} \\
 \sigma \vert_{x=M} \\
 \mathcal{B}( w )\vert_{x=M}
\end{pmatrix} 
= 
\begin{pmatrix}
 w(r_0^{-2} - r^{-2})\vert_{x=0} \\
 0 \\
 N_{\B}
\end{pmatrix}
:=
\SN_{\B}(\SX)
.
\end{equation}
Here $r$ is determined as a nonlinear function of $\sigma$ as usual.

Rewriting the nonlinear problem as \eqref{da_1}--\eqref{da_2} now allows us to easily describe the compatibility conditions for the initial data.  Given $\SX(0)$ as initial data for $\SX$ at $t=0$, we can use \eqref{da_1} to iteratively solve for $\dt^j \SX(0)$ for $j\ge 1$:
\begin{equation}\label{da_3}
\begin{split}
 \dt \SX(0) &=  - \mathcal{L} \SX(0) + \SN(\SX(0)) \\
 \dt^2 \SX(0) &= - \mathcal{L} \dt \SX(0) + D\SN(\SX(0)) \cdot \dt \SX(0) \\
 &=  -\mathcal{L} (- \mathcal{L} \SX(0) + \SN(\SX(0))) + D\SN(\SX(0)) \cdot (- \mathcal{L} \SX(0) + \SN(\SX(0))),
\end{split}
\end{equation}
and so on for higher derivatives, where here $D$ is the derivative of the nonlinearity.  We may similarly compute $\dt^j \SN_{\B}(\SX)(0)$:
\begin{equation}\label{da_3_5}
\dt \SN_{\B}(\SX)(0) = D \SN_{\B}(\SX(0)) \cdot \dt X(0) = D \SN_{\B}(\SX(0)) \cdot [-\mathcal{L} \SX(0) + \SN(\SX(0))]\vert_{x=M},
\end{equation}
continuing as above for higher derivatives.  This procedure may be carried out indefinitely as long as $\SX(0)$ is sufficiently smooth.  However, we may also differentiate the boundary condition \eqref{da_2} with respect to time and then set $t=0$ to see that the data must satisfy the boundary conditions
\begin{equation}\label{da_4}
 \SC(\dt^j \SX(0)) = \dt^j \SN_{\B}(\SX)(0)  \text{ for } j\ge 0.
\end{equation}
Since the terms $\dt^j \SX(0)$ and  $\dt \SN_{\B}(\SX)(0)$ constructed in \eqref{da_3}--\eqref{da_3_5} are determined entirely by $\SX(0)$, we then find that the data $\SX(0)$ must satisfy the nonlinear compatibility conditions given by substituting \eqref{da_3}--\eqref{da_3_5} into \eqref{da_4}.  

For completely smooth solutions to the nonlinear problem, the compatibility conditions would have to hold for all $j\ge 0$.  In our case, we only require solutions to remain in the  energy space defined by $\E$, and as such we must only solve for $\dt^j \SX(0)$ for $j=1,2,3$, given $\SX(0)$.  This then requires the compatibility condition from \eqref{da_4} only for $0\le j\le 3$.  Of course, in order to guarantee that $\E(0)$ is finite, we must have that $\dt^j \SX(0)$, $0\le j\le 3$, satisfy the integrability conditions in the definition of $\E(0)$.  This in turn gives us a natural Hilbert function space $\mathbb{H}$ with the following three properties.  First, if $\SX(0) \in \mathbb{H}$, then we have the trace estimates needed to make sense of the boundary conditions in \eqref{da_4} for $0\le j \le 3$.  Second, if $\norm{\SX(0)}_{\mathbb{H}}$ is sufficiently small, then $\E(0) \le C \norm{\SX(0)}_{\mathbb{H}}^2$ for some $C>0$.  Here the smallness assumption is needed to deal with the nonlinearities in \eqref{da_3}--\eqref{da_3_5} and the $r$ terms in $\E$.  Third, the linear growing modes produced in Theorem \ref{growing_mode} are in $\mathbb{H}$. It is straightforward to extract the proper definition of $\mathbb{H}$ from $\E$ and to work out the details of the estimate of $\E(0)$; as such, for the sake of brevity we omit these.   With $\mathbb{H}$ defined in this way, it is then easy to use estimate \eqref{gm_03} of Theorem \ref{growing_mode} in conjunction with \eqref{growing_1}--\eqref{growing_3} to see that the growing modes are in $\mathbb{H}$.

Now that we have stated the nonlinear compatibility conditions, we see why we cannot simply set $\SX(0) = \iota \SX_0$ with
\begin{equation}
 \SX_0 = 
\begin{pmatrix}
 \sigma_\star \\ \bar{w}_\star
\end{pmatrix}
\end{equation}
for $\sigma_\star$ and $v_\star = \bar{w}_\star / r_0^2$ the growing mode solution constructed in Theorem \ref{growing_mode} and $\iota >0$ a small parameter.  Indeed, these solve
\begin{equation}\label{da_5}
 \lambda \SX_0 + \mathcal{L} \SX_0 =0 \text{ and } \SC(\SX_0)=0 \Rightarrow \SC(\mathcal{L}^j \SX_0) =0 \text{ for all } j\ge 0,
\end{equation}
which in particular means that $\SX(0) = \iota \SX_0$ does not satisfy the nonlinear compatibility condition \eqref{da_4} for $j\ge 1$.   

To get around this obstacle, we will use the implicit function theorem to produce a curve of initial data satisfying the compatibility conditions, close to the linear growing modes.  To this end, let us define the map  $F:\mathbb{H} \to \Rn{12}$ via
\begin{equation}
 F(\SX) = 
\begin{pmatrix}
 \SC(\SX) \\ \SC(\dt \SX) \\ \SC(\dt^2 \SX) \\ \SC(\dt^3 \SX)
\end{pmatrix}
-
\begin{pmatrix}
 \SN_{\B}(\SX) \\ \dt \SN_{\B}(\SX) \\ \dt^2 \SN_{\B}(\SX) \\ \dt^3 \SN_{\B}(\SX)
\end{pmatrix},
\end{equation}
where we understand that $\dt^j \SX$ and $\dt^j \SN_{\B}(\SX)$ for $j=1,2,3$ are computed in terms of $\SX$ as in \eqref{da_3}--\eqref{da_3_5}.  Let $\SX_0$ be the linear growing modes as above and let $\SX_i\in \mathbb{H}$,  $i=1,\dotsc,12$, be arbitrary for now, with exact values to be chosen later.  We then define $f:\Rn{1+12} \to \Rn{12}$ via
\begin{equation}
f(t,\tau) = F\(t \SX_0 + \sum_{i=1}^{12} \tau_i  \SX_i\) \text{ for } t \in \Rn{} \text{ and } \tau \in \Rn{12}.
\end{equation}
It is straightforward to see, given the structure of the nonlinearities $\SN(\cdot)$ and $\SN_{\B}(\cdot)$, that $f \in C^2(\Rn{1+12}; \Rn{12})$.  Also, $f(0,0) = 0$ and 
\begin{equation}\label{da_6}
\frac{\partial f}{\partial t}(0,0) = 
\begin{pmatrix}
  \SC(\SX_0) \\ \SC(\lambda  \SX_0) \\ \SC(\lambda^2 \SX_0) \\ \SC(\lambda^3  \SX_0)
\end{pmatrix}
=0
\text{ and }
\frac{\partial f}{\partial \tau_i}(0,0) = 
\begin{pmatrix}
  \SC(\SX_i) \\ \SC(-\mathcal{L}  \SX_i) \\ \SC(\mathcal{L}^2 \SX_i) \\ \SC(-\mathcal{L}^3  \SX_i)
\end{pmatrix}.
\end{equation}
From this it is then straightforward to choose the  $\SX_i$ for $i=1,\dotsc,12$ so that the $12 \times 12$ matrix $\partial f/\partial \tau(0,0)$ is invertible.   The implicit function theorem then provides a small constant $\iota_0>0$ and a function  $\xi: (-\iota_0, \iota_0) \to \Rn{12}$ so that $f(t,\xi(t)) = 0$ for all $t \in (-\iota_0,\iota_0)$ and so that $\xi \in C^2$ and $\xi(0) =0$.  We may then differentiate the equation $f(t,\xi(t)) = 0$ with respect to $t$, set $t=0$, and use the first equation in \eqref{da_6} to see that
\begin{equation}
0 = \frac{\partial f}{\partial t}(0,0) + \frac{\partial f}{\partial \tau}(0,0) \frac{d\xi(0)}{dt} = \frac{\partial f}{\partial \tau}(0,0) \frac{d\xi}{dt}(0) \Rightarrow \frac{d\xi}{dt}(0) =0
\end{equation}
since the matrix $\partial f / \partial \tau (0,0)$ is invertible.  Then $\xi \in C^2$ with $\xi(0) = \dot{\xi}(0) =0$ so that $\xi(t)/t^2$ is well-defined and continuous on $(-\iota_0,\iota_0)$.   Using this, we may then deduce the existence of a small parameter $\iota_0>0$ and a curve $\mathfrak{Y}: (-\iota_0,\iota_0) \to \mathbb{H}$ given by
\begin{equation}
\mathfrak{Y}(\iota) = \iota \SX_0 + \iota^2 \sum_{i=1}^{12} \SX_i \frac{\xi_i(\iota)}{\iota^2} :=  \iota \SX_0 + \iota^2 \bar{\mathfrak{Y}}(\iota)
\end{equation}
so that for all $\iota \in (-\iota_0,\iota_0),$
\begin{equation}\begin{split}
 F(\mathfrak{Y}(\iota)) &=0, \text{ i.e. } \mathfrak{Y}(\iota) \text{ satisfies the nonlinear compatibility conditions} \\
\sqrt{\E(\mathfrak{Y}(\iota))} & \le C \norm{ \mathfrak{Y}(\iota) }_{\mathbb{H}} \le C \iota, \text{ and }\\
 \mathfrak{E}(\bar{\mathfrak{Y}}(\iota)_1, \bar{\mathfrak{Y}}(\iota)_2)   &\le C,
\end{split}
\end{equation}
where here the norm  $\norm{\cdot}_0 \le \norm{\cdot}_{\mathbb{H}}$ is given by \eqref{norm_0_def}, the term $\mathfrak{E}$ is defined by \eqref{frak_E_def}, and in the second line we have written $\E(\mathfrak{Y}(\iota))$ for $\E(0)$ computed from the initial data $\SX(0) = \mathfrak{Y}(\iota)$.

We now recast the above discussion as a lemma.

\begin{lem}\label{data_analysis}
Let $\sigma_\star, v_\star$  be the growing mode solution constructed in Theorem \ref{growing_mode}, write and $ \bar{w}_\star =r_0^2 v_\star$, and assume the normalization
\begin{equation}\label{data_0}
 \norm{
\begin{pmatrix}
 \sigma_\star \\ \bar{w}_\star 
\end{pmatrix}
}_0 = 1,
\end{equation}
for $\norm{\cdot}_0$ the norm defined by \eqref{norm_0_def}.  Then there exists a number $\iota_0 >0$ and a family of initial data
\begin{equation}
\begin{pmatrix}
 \sigma^\iota(0) \\  w^\iota(0)
\end{pmatrix}
=
 \SX(\iota) = 
\iota 
\begin{pmatrix}
 \sigma_\star \\ \bar{w}_\star
\end{pmatrix}
+ \iota^2
\begin{pmatrix}
 \sigma_0(\iota) \\ w_0(\iota)
\end{pmatrix}
\end{equation}
for $\iota \in [0,\iota_0)$ so that the following hold.
\begin{enumerate}
 \item $\SX(\iota)$ satisfies the nonlinear compatibility conditions required for a solution to the nonlinear problem \eqref{da_1} to exist in the energy space defined by $\E$.
 \item If $\E(0)$ denotes the value of $\E$ determined at $t=0$ from the data $\SX(\iota)$, then $\E(0)  \le C \iota^2$ for a constant $C>0$.
 \item For all $\iota \in [0,\iota_0)$, it holds that
\begin{equation}\label{data_1}
  \norm{
\begin{pmatrix}
 \sigma_0(\iota) \\ w_0(\iota) 
\end{pmatrix}
}_0^2 \le 
\mathfrak{E}(\sigma_0(\iota), w_0(\iota))
\le  C
\end{equation}
for a constant $C >0$ independent of $\iota$, where $\mathfrak{E}$ is given by \eqref{frak_E_def}.

\item Let $\psi^\iota$ denote the function given by \eqref{psi} with $N_{\B} = N_{\B}(\SX(\iota))$ determined by the data $\SX(\iota)$ at $t=0$.  Then $\bar{w}^\iota(0) = w^\iota(0) - \psi^\iota$ satisfies the homogeneous boundary condition $\B(\bar{w}^\iota(0)) = 0$ and 
\begin{equation}\label{data_2}
  \norm{
\begin{pmatrix}
 0 \\ \psi^\iota 
\end{pmatrix}
}_0^2 \le 
\mathfrak{E}(0, \psi^\iota)
\le  C \iota^4
\end{equation}
for a constant $C>0$ independent of $\iota$.
\end{enumerate}

\end{lem}
\begin{proof}
Everything except for the last item is proved above.  The last item follows from Lemma \ref{boundary_transform} and the fact that $N_{\B}$ is at least a quadratic nonlinearity.
\end{proof}

%%%%%%%%%%%%%%%%%%%%%%%%%%%%%%%%%%%%%%%%%%%%%%%%%%%
\subsection{Instability}
%%%%%%%%%%%%%%%%%%%%%%%%%%%%%%%%%%%%%%%%%%%%%%%%%%%

We are now ready to prove our main result.

\begin{thm}\label{thm}
There exist $\theta_0 > 0$, $C>0$, and $0< \iota_0 < \theta_0$ such that for any  $0 < \iota \le \iota_0$, there exists a family of solutions $\sigma^\iota(t)$ and $v^\iota(t)$ to the Navier-Stokes-Poisson system \eqref{PNSP} so that 
\begin{equation}\label{th_00}
\sqrt{{\mathcal{E}}}(0)\le C\iota\; \text{ but }\; \sup_{0\le t\le T^\iota}\sqrt{\mathcal{E}^0}(t)\ge \theta_0.
\end{equation}
Here $T^\iota$ is given by $T^\iota = \frac{1}{\lambda} \ln \frac{\theta_0}{\iota}$.
\end{thm}

\begin{proof} 

We divide the proof into steps.  At several points in the proof we will restrict the size of $\theta$.  Whenever we do so, we assume that $\iota$ is also restricted so that $0<\iota \le \iota_0 \le \theta$.  We will choose the value of $\theta_0$ in the final step of the proof.

\textbf{Step 1} -- Data and the solutions

Let us assume that $\iota_0$ is as small as the $\iota_0$ appearing in Lemma \ref{data_analysis} and then let $\SX(\iota)$ for $\iota \le \iota_0$ be the family of initial data for the nonlinear problem \eqref{da_1}--\eqref{da_2} given in the lemma.    For $0 < \iota \le \iota_0$, we now let  $\binom{\sigma^\iota}{w^\iota}$ be solutions  to the Navier-Stokes-Poisson system \eqref{da_1}--\eqref{da_2} with a family of initial data 
\begin{equation}
\left.
\begin{pmatrix}
 \sigma^\iota \\ w^\iota 
\end{pmatrix}
\right\vert_{t=0} 
=
\begin{pmatrix}
 \sigma^\iota(0) \\ w^\iota(0) 
\end{pmatrix}
= \SX(\iota)  =
\iota 
\begin{pmatrix}
 \sigma_\star \\ \bar{w}_\star
\end{pmatrix}
+ \iota^2
\begin{pmatrix}
 \sigma_0(\iota) \\ w_0(\iota)
\end{pmatrix}.
\end{equation}
The solution satisfies $\sqrt{\E}(0) \le C \iota$.

Note that since
\begin{equation}
 r^\iota(x,0) = \left(\frac{3}{4\pi} \int_0^x \frac{dy}{\rho_0(y) + \iota \sigma_*(y) + \iota^2 \sigma_0(\iota)(y)} \right)^{1/3},
\end{equation}
a Taylor expansion and  item 2 of Lemma \ref{data_analysis} allow us to estimate
\begin{equation}
 \pnorm{1 - \frac{r_0(x)}{r^\iota(x,0)}}{\infty}^2 +  \frac{\nu}{2} \pnorm{1 - \frac{r_0(x)}{r^\iota(x,0)}}{2}^2  \le A_1 \iota^2
\end{equation}  
for a constant $A_1>0$, independent of $\iota$.  From this, the normalization \eqref{data_0}, and the estimate \eqref{data_1},  we may assume that $\iota < \iota_0$ with $\iota_0$ small enough so that
\begin{equation}\label{th_1}
\frac{\iota}{2} \le \sqrt{ \E^{0,\sigma^\iota}(0) +\E^{0,v^\iota}(0) + \E^{1,v^\iota}(0) }  
+ \sqrt{ \E^{0,r^\iota}(0)} \le 2\iota. 
\end{equation}

Throughout the rest of the proof we will let $\E(t)$ denote the total energy, defined by \eqref{total_energy}, associated to the solutions $\sigma^\iota$ and $w^\iota$ at time $t$.

\textbf{Step 2} -- Control of the energy

Let us define the constant
\begin{equation}\label{th_B}
B_0 := \left(2 + \frac{27}{8\sqrt{2} \lambda} \pnorm{\rho_0}{\infty}^{1/2}  \right).
\end{equation}
It will be useful in determining the time-scale in which instability begins.  Indeed, we define $T$ by 
\begin{equation}\label{th_2}
T:=\sup\left\{s \left \vert \sqrt{ \E^{0,\sigma^\iota}(t) +\E^{0,v^\iota}(t) + \E^{1,v^\iota}(t) }
+ \sqrt{ \E^{0,r^\iota}(t)} \le (4+B_0) \iota e^{\lambda t} \text{ for }0\le t\le s \right\} \right..
\end{equation}
The estimate \eqref{th_1} guarantees that $T>0$.  Then by Proposition \ref{prop} and \eqref{th_1}, there exist $C_\star$ and $\theta_\star>0$ such that for $0\le t\le \min\{T,T(\iota,\theta_\star) \}$ (with $T(\iota,\theta_\star)$ given in the Proposition), 
\begin{equation}\label{th_3}
\sqrt{\mathcal{E}} (t)\le C_\star \iota e^{\lambda t}.
\end{equation}

Let us assume that $\theta \le \theta_\star$, which means that $T^\iota \le T(\iota,\theta_\star)$, and hence that the estimate \eqref{th_3} also holds for $0 \le t \le \min\{T,T^\iota \}$.    Let us further assume that $\theta$ is small enough so that $\sqrt{\mathcal{E}} (t)\le C_\star \theta$ is small enough so that the right side of the estimate in Lemma \ref{lem4.6} is smaller than $1/2$.  In particular, this implies that
\begin{equation}\label{th_5}
\pnorm{\frac{\sigma^\iota(t) }{\rho_0}}{\infty} + \pnorm{1-\frac{r_0}{r^\iota(t) }}{\infty} \le \frac{1}{2}
\end{equation}
for all $0 \le t \le \min\{T,T^\iota\}$.  By further restricting $\theta$ to decrease the bound of the terms in \eqref{th_5}, and using the identities in \eqref{rr} and \eqref{ID}, we can also bound
\begin{equation}\label{th_9}
 \frac{1}{4}  \left( \E^{0,\sigma^\iota}(t) +\E^{0,v^\iota}(t) + \E^{1,v^\iota}(t)  \right) \le
\norm{
 \begin{pmatrix}
 \sigma^\iota(t) \\ w^\iota(t)
\end{pmatrix}
}_0^2  \le 2  \left( \E^{0,\sigma^\iota}(t) +\E^{0,v^\iota}(t) + \E^{1,v^\iota}(t)  \right) 
\end{equation}
for $0 \le t \le T^\iota = \min\{T,T^\iota\}$.

\textbf{Step 3} -- Linear estimates for $\sigma^\iota$ and $w^\iota$

Notice that because of the estimate \eqref{th_5},  the boundary condition $w^\iota/(r^\iota)^2 \vert_{x=0}$ is equivalent to $w^\iota/r_0^2 \vert_{x=0}$.  We can then modify the problem  \eqref{da_1}--\eqref{da_2} to have the form  \eqref{mild_form_1}--\eqref{mild_form_2}, the latter of which has the homogeneous boundary conditions \eqref{mild_form_2}.  This leads us to consider $\bar{w}^\iota(0) = w^\iota(0) - \psi^\iota$ as in Lemma \ref{data_analysis}, which satisfies $\B(\bar{w}^\iota(0))=0$ at $x=M$.  We then have that
\begin{equation}
 e^{t \mathcal{L}} 
\begin{pmatrix}
\sigma^\iota(0) \\  \bar{w}^\iota(0)
\end{pmatrix}
 = 
\iota e^{\lambda t} 
\begin{pmatrix}
 \sigma_\star \\ \bar{w}_\star
\end{pmatrix}
+
\iota^2
e^{t \mathcal{L}}
\begin{pmatrix}
 \sigma_0(\iota) \\ w_0(\iota)
\end{pmatrix}
-
e^{t \mathcal{L}}
\begin{pmatrix}
 0 \\ \psi^\iota
\end{pmatrix}
.
\end{equation}
Then the solutions $\binom{\sigma^\iota}{w^\iota}$ to \eqref{da_1} can be written as given in \eqref{duhamel}:  
\begin{multline}\label{duhamel_2}
 \begin{pmatrix}
 \sigma^\iota(t) \\ w^\iota(t)
\end{pmatrix} 
=  \iota e^{\lambda t} 
\begin{pmatrix}
 \sigma_\star \\ \bar{w}_\star
\end{pmatrix}
+
\iota^2
e^{t \mathcal{L}}
\begin{pmatrix}
 \sigma_0(\iota) \\ w_0(\iota)
\end{pmatrix}
-
e^{t \mathcal{L}}
\begin{pmatrix}
 0 \\ \psi^\iota
\end{pmatrix}
- 
\frac{1}{\delta}
\begin{pmatrix}
 0 \\ N_{\mathcal{B}}^\iota(t)  r_0^3 /3
\end{pmatrix}
\\
+ \int_0^t e^{(t-s) \mathcal{L}} 
\begin{pmatrix}
 N_1^\iota(s) \\ N_2^\iota(s)
\end{pmatrix}ds
+ \frac{1}{\delta} \int_0^t e^{(t-s) \mathcal{L}} 
\begin{pmatrix}
N_{\mathcal{B}}^\iota(s) \rho_0 \\  \dt N_{\mathcal{B}}^\iota(s) r_0^3/3
\end{pmatrix}ds.
\end{multline}
Here the nonlinear terms $N_\B^\iota$, $N_1^\iota$, and $N_2^\iota$ are defined in terms of $w^\iota$ and $\sigma^\iota$ via \eqref{NB_def} and \eqref{N12_def}.

Theorem \ref{mild_estimates}, together with the  nonlinear estimates of Lemma \ref{lem5.3}, imply that if $t\le \min\{T,T(\iota,\theta_\star) \}$, then
\begin{equation}\label{5.9}
\begin{split}
 \norm{
\begin{pmatrix}
 \sigma^\iota(t) \\ w^\iota(t)
\end{pmatrix}
- \iota e^{t \lambda}
\begin{pmatrix}
 \sigma_\star \\ \bar{w}_\star
\end{pmatrix}
-
\iota^2
e^{t \mathcal{L}}
\begin{pmatrix}
 \sigma_0(\iota) \\ w_0(\iota)
\end{pmatrix}
+
e^{t \mathcal{L}}
\begin{pmatrix}
 0 \\ \psi^\iota
\end{pmatrix}
}_0 &\le   C\abs{ \mathcal{E}(t)}^2 + C\int_0^t e^{\lambda(t-s)} \mathcal{E}(s) ds  \\
&\le   C(\iota e^{\lambda t})^2 + C \int_0^t e^{\lambda(t-s)} \iota^2 e^{2\lambda s} ds \\
& \le  A_2(\iota e^{\lambda t})^2
\end{split}
\end{equation}
for a constant  $A_2>0$ independent of $\iota$.  On the other hand, because of the estimates \eqref{mild_recast} and \eqref{data_1}--\eqref{data_2}, we may estimate
\begin{multline}\label{th_16}
\norm{\iota^2
e^{t \mathcal{L}}
\begin{pmatrix}
 \sigma_0(\iota) \\ w_0(\iota)
\end{pmatrix}}_0
+
\norm{
e^{t \mathcal{L}}
\begin{pmatrix}
 0 \\ \psi^\iota
\end{pmatrix}}_0
\\
\le 
\iota^2 C e^{\lambda t} \sqrt{ \mathfrak{E}(\sigma_0(\iota), w_0(\iota)) }
+ C e^{\lambda t} \sqrt{\mathfrak{E}(0, \psi^\iota) }
\le \iota^2 A_3 e^{\lambda t}
\end{multline}
for a constant $A_3 >0$ independent of $\iota$.   Then we may then deduce from \eqref{5.9}, \eqref{th_16}, and the normalization \eqref{data_0} that
\begin{equation}\label{th_4}
\norm{
 \begin{pmatrix}
 \sigma^\iota(t) \\ w^\iota(t)
\end{pmatrix}
}_0 
\le  \iota e^{\lambda t} +\iota^2 A_3 e^{\lambda t} + A_2(\iota e^{\lambda t})^2.
\end{equation}

\textbf{Step 4} -- Control of the $r$ energy

We now turn to control of the term $\E^{0,r^\iota}$.  First note that
\begin{equation}
 \frac{d}{dt} \int \frac{\nu}{2} \abs{1 - \frac{r_0}{r^\iota}}^2 dx \le \left(\int \nu \abs{1 - \frac{r_0}{r^\iota}}^2 dx \right)^{1/2} \left(  \int \nu \abs{\frac{w^\iota}{r_0^3}}^2 dx \right)^{1/2} \pnorm{\frac{r_0}{r^\iota}}{\infty}^4,
\end{equation}
which together with \eqref{th_5} implies that
\begin{equation}\label{th_6}
 \frac{d}{dt} \sqrt{\E^{0,r^\iota}(t)} \le \frac{81}{16\sqrt{2}}  \left(  \int \nu \abs{\frac{w^\iota } {r_0^3 } }^2 dx \right)^{1/2}.
\end{equation}
We may then argue as in \eqref{ID} to see that
\begin{equation}\label{th_7}
\int \frac{\nu}{\rho_0} \abs{\frac{w^\iota}{r_0^3}}^2 dx \le \int \frac{32 \pi^2}{9} \left( \delta \rho_0 \abs{\dx w^\iota}^2 + \frac{4\ep}{3} \rho_0 r_0^6 \abs{\dx \left(\frac{w^\iota}{r_0^3} \right)}^2  \right) dx \le \frac{4}{9} 
\norm{
 \begin{pmatrix}
 \sigma^\iota(t) \\ w^\iota(t)
\end{pmatrix}
}_0^2.
\end{equation}
Combining \eqref{th_6} and \eqref{th_7} with \eqref{th_4}, we find that
\begin{equation}
 \frac{d}{dt} \sqrt{\E^{0,r^\iota}(t)} \le \frac{27}{8\sqrt{2}} \pnorm{\rho_0}{\infty}^{1/2}
\norm{
 \begin{pmatrix}
 \sigma^\iota(t) \\ w^\iota(t)
\end{pmatrix}
}_0
\le 
\frac{27}{8\sqrt{2}} \pnorm{\rho_0}{\infty}^{1/2} \left(
 \iota e^{\lambda t} + A_3 \iota^2 e^{\lambda t} + A_2(\iota e^{\lambda t})^2 \right)
\end{equation}
for $0 \le t \le \min\{T,T^\iota\}$.  Integrating this from $0$ to $t \le \min\{T,T^\iota\}$ and employing \eqref{th_1} then yields the estimate 
\begin{multline}\label{th_8}
 \sqrt{\E^{0,r^\iota}(t)} \le \left(2 + \frac{27}{8\sqrt{2} \lambda} \pnorm{\rho_0}{\infty}^{1/2}  \right)(\iota e^{\lambda t} +  A_3 \iota^2 e^{\lambda t})  
+ \frac{27 A_2}{16\sqrt{2} \lambda} \pnorm{\rho_0}{\infty}^{1/2}(\iota e^{\lambda t})^2 \\
 = B_0 (\iota e^{\lambda t}) + (\iota A_4) (\iota e^{\lambda t})  + A_5 (\iota e^{\lambda t})^2
\end{multline}
for $0 \le t \le \min\{T,T^\iota\}$, where $B_0$ is the constant defined above in \eqref{th_B}, and $A_4,A_5$ are  constants independent of $\iota$.

\textbf{Step 5} -- The bound $T^\iota \le T$

We now claim that if $\theta$ is taken to be small enough, then $T^\iota = \frac{1}{\lambda} \ln \frac{\theta}{\iota}\le T$.  Suppose by way of contradiction that $T^\iota > T$.  Then the first bound in \eqref{th_9}, \eqref{th_4}, and \eqref{th_8}  imply that
\begin{multline}
 \sqrt{\E^{0,\sigma^\iota}(t) +\E^{0,v^\iota}(t) + \E^{1,v^\iota}(t) } + \sqrt{\E^{0,r^\iota}(t)} \le (2 + B_0 + \iota A_4) (\iota e^{\lambda t}) + ( 2A_2 + A_5) (\iota e^{\lambda t})^2 \\
\le [2 + B_0  + \iota A_4 + ( 2A_2 + A_5) \theta] (\iota e^{\lambda t})  
\le [3 + B_0 ] (\iota e^{\lambda t}) 
\end{multline}
for $t \le T^\iota$, if we assume that $\theta$ is small enough so that $\theta(2 A_2 + A_5) \le 1/2$ and $\iota_0 A_4 \le 1/2$.  For this choice of $\theta$, we then find from the definition of $T$ that $T \ge T^\iota$, a contradiction.  Hence $T^\iota \le T$ for $\theta$ sufficiently small.

\textbf{Step 6} -- Conclusion: instability

We now define the $L^2$ part of the norm $\norm{\cdot}_{0}$ by
\begin{equation}
\norm{
\begin{pmatrix}
 \sigma \\ w
\end{pmatrix}
}_{00}^2
:=
 \hal\int K\gamma\rho_0^{\gamma-1} \abs{\frac{{\sigma}}{\rho_0}}^2 dx  + \hal\int \abs{\frac{w}{r_0^2}}^2 dx.   
\end{equation}
Note that $\norm{\cdot}_{00} \le \norm{\cdot}_0$ and that by the normalization  \eqref{data_0}, we have that the data satisfy 
\begin{equation}\label{th_10}
\norm{
\begin{pmatrix}
 \sigma_\star \\ \bar{w}_\star
\end{pmatrix}
}_{00}^2
:= C_{00} \in (0,1). 
\end{equation}
Also, we may argue as in the derivation of \eqref{th_9} to see that
\begin{equation}\label{th_11}
 \sqrt{ \E^{0,\sigma^\iota}(t) + \E^{0,v^\iota}(t)  }  \ge \frac{1}{\sqrt{2}} \norm{
\begin{pmatrix}
 \sigma^\iota \\ w^\iota
\end{pmatrix}
}_{00}
\end{equation}
for $0 \le t \le T^\iota = \min\{T,T^\iota\}$.

Let us now further assume that $\theta$ is small enough so that $A_2 \theta \le C_{00} /4$ and  $\iota_0 A_3 \le C_{00}/4$.  We can then combine \eqref{th_10}, \eqref{th_11},  \eqref{5.9}, and \eqref{th_16} to deduce that
\begin{multline}
  \sqrt{ \E^{0,\sigma^\iota}(T^\iota) + \E^{0,v^\iota}(T^\iota) }  
 \ge 
\frac{1}{\sqrt{2}}
\norm{
 \begin{pmatrix}
 \sigma^\iota(T^\iota) \\ w^\iota(T^\iota)
\end{pmatrix}
}_{00} 
 \ge 
\frac{1}{\sqrt{2}}  
\iota e^{\lambda T^\iota}
\norm{
 \begin{pmatrix}
 \sigma_\star \\ \bar{w}_\star 
\end{pmatrix}
}_{00}  
\\
- \frac{1}{\sqrt{2}}
\norm{
 \begin{pmatrix}
 \sigma^\iota(T^\iota) \\ w^\iota(T^\iota)
\end{pmatrix}
-
\iota e^{\lambda T^\iota}
 \begin{pmatrix}
 \sigma_\star \\ \bar{w}_\star 
\end{pmatrix}
-
\iota^2 e^{T^\iota \mathcal{L}}
\begin{pmatrix}
 \sigma_0(\iota) \\ w_0(\iota)
\end{pmatrix}
+
e^{T^\iota \mathcal{L}}
\begin{pmatrix}
 0 \\ \psi^\iota
\end{pmatrix}
}_{00}
\\
-
\frac{1}{\sqrt{2}}
\norm{
\iota^2 e^{T^\iota \mathcal{L}}
\begin{pmatrix}
 \sigma_0(\iota) \\ w_0(\iota)
\end{pmatrix}
}_{00}
-
\frac{1}{\sqrt{2}}
\norm{
 e^{T^\iota \mathcal{L}}
\begin{pmatrix}
 0 \\ \psi^\iota
\end{pmatrix}
}_{00}
\\
\ge \frac{1}{\sqrt{2}}
\iota e^{\lambda T^\iota} C_{00}
-
\frac{1}{\sqrt{2}}
\norm{
\iota^2 e^{T^\iota \mathcal{L}}
\begin{pmatrix}
 \sigma_0(\iota) \\ w_0(\iota)
\end{pmatrix}
}_{0}
-
\frac{1}{\sqrt{2}}
\norm{
 e^{T^\iota \mathcal{L}}
\begin{pmatrix}
 0 \\ \psi^\iota
\end{pmatrix}
}_{0}
\\
- \frac{1}{\sqrt{2}}
\norm{
 \begin{pmatrix}
 \sigma^\iota(T^\iota) \\ w^\iota(T^\iota)
\end{pmatrix}
-
\iota e^{\lambda T^\iota}
 \begin{pmatrix}
 \sigma_\star \\ \bar{w}_\star 
\end{pmatrix}
-
\iota^2 e^{T^\iota \mathcal{L}}
\begin{pmatrix}
 \sigma_0(\iota) \\ w_0(\iota)
\end{pmatrix}
+
e^{T^\iota \mathcal{L}}
\begin{pmatrix}
 0 \\ \psi^\iota
\end{pmatrix}
}_{0}
\\
 \ge \frac{1}{\sqrt{2}} \left( 
\iota e^{\lambda T^\iota} C_{00} 
- A_3 \iota^2 e^{\lambda T^\iota} - 
A_2 ( \iota e^{\lambda T^\iota} )^2
\right)
\\ = 
\frac{1}{\sqrt{2}} \left( \theta C_{00} - \theta \iota A_3 - A_2 \theta^2 \right)
\ge \frac{C_{00}}{2\sqrt{2}} \theta.
\end{multline}
Setting $\theta_0 = (\theta C_{00}) /(2\sqrt{2})$, we find that \eqref{th_00} holds.  This completes the proof of the theorem. 
\end{proof}

%\pagebreak
%%%%%%%%%%%%%%%%%%%%%%%%%%%%%%%%%%%%%%%%%%%%%%%%%%%%%%%%%%%%%%%%%%%%%%%%
%References
%%%%%%%%%%%%%%%%%%%%%%%%%%%%%%%%%%%%%%%%%%%%%%%%%%%%%%%%%%%%%%%%%%%%%%%%

\vskip 1cm
\noindent
{\sc Juhi Jang}\\
Department of Mathematics\\
University of California, Riverside\\
Riverside, CA 92521, USA \\ 
{\tt juhijang@math.ucr.edu}
\vskip 1cm
\noindent
{\sc Ian Tice}\\
Division of Applied Mathematics\\
Brown University \\
182 George St., Providence, RI 02912, USA \\ 
{\tt tice@dam.brown.edu}

\end{document}